\documentclass[reqno,11pt]{article}
\usepackage{a4wide}
\usepackage{color,eucal,enumerate,mathrsfs, amsthm}
\usepackage[normalem]{ulem}
\usepackage{amsmath,amssymb,epsfig,bbm}
\numberwithin{equation}{section}

\usepackage[pdfborder={0 0 0}]{hyperref}  

\usepackage[latin1]{inputenc}
\usepackage[english]{babel}

\newcommand{\N}{\mathbb{N}}
\newcommand{\Q}{\mathbb{Q}}
\newcommand{\R}{\mathbb{R}}
\newcommand{\Z}{\mathbb{Z}}

\newcommand{\mm}{{\mbox{\boldmath$m$}}}

\newcommand{\ppi}{{\mbox{\boldmath$\pi$}}}

\newcommand{\sppi}{{\mbox{\scriptsize\boldmath$\pi$}}}

\newcommand{\sfd}{{\sf d}}

\newcommand{\sfh}{{\sf h}}

\newcommand{\Kliminf}{K\kern-3pt-\kern-2pt\mathop{\rm lim\,inf}\limits}  
\newcommand{\supp}{\mathop{\rm supp}\nolimits}   
\renewcommand{\d}{{\mathrm d}}

\newcommand{\restr}[1]{\lower3pt\hbox{$|_{#1}$}}
\newcommand{\la}{\left<}                  
\newcommand{\ra}{\right>}
\newcommand{\eps}{\varepsilon}  
\newcommand{\nchi}{{\raise.3ex\hbox{$\chi$}}}

\newcommand{\limi}{\varliminf}
\newcommand{\lims}{\varlimsup}

\newcommand{\gopt}{{\rm{OptGeo}}}                   
\newcommand{\prob}[1]{\mathscr P(#1)}                   
\newcommand{\probt}[1]{\mathscr P_2(#1)}                   
\newcommand{\e}{{\rm{e}}}                           
\newcommand{\geo}{{\rm Geo}}                      

\renewcommand{\mm}{\mathfrak m}                                

\newcommand{\weakgrad}[1]{|\nabla #1|_w} 

\renewenvironment{proof}{\removelastskip\par\medskip   
\noindent{\em proof} \rm}{\penalty-20\null\hfill$\square$\par\medbreak}
\newenvironment{sketch}{\removelastskip\par\medskip   
\noindent{\em Sketch of the proof} \rm}{\penalty-20\null\hfill$\square$\par\medbreak}

\newtheorem{theorem}{Theorem}[section]

\newtheorem{corollary}[theorem]{Corollary}
\newtheorem{lemma}[theorem]{Lemma}
\newtheorem{proposition}[theorem]{Proposition}
\newtheorem*{theoremno}{Theorem}

\newtheorem{definition}[theorem]{Definition}

\newcommand{\bd}{{\mathbf\Delta}}

\newcommand{\s}{{\rm S}}

\newcommand{\test}[1]{{\rm Test}(#1)}

\renewcommand{\b}{{\rm b}}
\newcommand{\X}{{\rm F}}

\newcommand{\h}{{\sfh}}
\newcommand{\mau}{{\sf T}}
\newcommand{\mad}{{\sf S}}

\renewcommand{\ae}{{\textrm{\rm{-a.e.}}}}

\newcommand{\CD}{{\sf CD}}
\newcommand{\RCD}{{\sf RCD}}

\newcommand{\dom}{{\rm DH}}
\newcommand{\lip}{{\rm lip}}

\renewcommand{\u}{\mathcal U}
\renewcommand{\weakgrad}[1]{|\nabla #1|}

\title{An overview on the proof of the splitting theorem in non-smooth context}
\begin{document}

\author{
   Nicola Gigli\
   \thanks{Universit\'e de Nice. email: \textsf{nicola.gigli@unice.fr}}
   }

\maketitle




\tableofcontents

\section*{Introduction}
\addcontentsline{toc}{section}{\protect\numberline{} Introduction}

In the recent paper \cite{Gigli13} the Cheeger-Colding-Gromoll splitting theorem has been generalized to the abstract class of metric measure spaces with Riemannian Ricci curvature bounded from below, the analysis being  based on some definitions and results contained in \cite{Gigli12}. These two papers add up to almost 200 pages and as such they are not suitable for getting a quick idea of the techniques used to work in the non-smooth setting. This is the aim of this note:  to provide an as short as possible yet comprehensive proof of the splitting in such abstract framework. The focus here is thus to prove:
\begin{theoremno}
Let $(X,\sfd,\mm)$ be a  $\RCD(0,N)$ space containing a line, i.e. such that there is a map $\bar\gamma:\R\to X$ satisfying
\[
\sfd(\bar\gamma_t,\bar\gamma_s)=|t-s|,\qquad\forall t,s\in\R.
\]
Then $(X,\sfd,\mm)$  is isomorphic to the product of the Euclidean line $(\R,\sfd_{\rm Eucl},\mathcal L^1)$ and another space $(X',\sfd',\mm')$, where the product distance $\sfd'\times\sfd_{\rm Eucl}$ is defined as
\begin{equation}\tag{\#}
\label{eq:intropit}
\sfd'\times\sfd_{\rm Eucl}\big((x',t),(y',s)\big)^2:=\sfd'(x',y')^2+|t-s|^2,\qquad\forall x',y'\in X',\ t,s\in\R.
\end{equation}
Moreover:
\begin{itemize}
\item if $N\geq 2$, then $(X',\sfd',\mm')$ is a  $\RCD(0,N-1)$ space,
\item if $N\in[1,2)$, then $X'$ is just a point.
\end{itemize}
Here `isomorphic' means that there is a measure preserving isometry between the spaces.
\end{theoremno}

Given that one of the scopes of this paper is to be reasonably short, all the necessary definitions and intermediate results are stated in the form needed to get the splitting theorem, without any aim of covering general situations as done in \cite{Gigli12}, \cite{Gigli13}. Also, the proof of some statements are only sketched: in these cases the main idea for the proof is given, but technical details are only briefly mentioned.   On the other hand, the exposition here is quite self-contained in the sense that all the recently introduced tools of differential calculus on metric measure spaces are recalled and discussed. The preliminary notions that are required are contained in sections labeled as `things to know'. Here is their list together with relative references:
\begin{itemize}
\item[-] The definition of Sobolev space $W^{1,2}$ of real valued Sobolev functions defined on a metric measure space. There is a quite large literature concerning this now classical object, see for instance \cite{Heinonen07} and references therein. Here we recall a definition proposed in \cite{AmbrosioGigliSavare11} - equivalent to the previous ones - which best suits our discussion.
\item[-] Some knowledge of optimal transport and of the curvature dimension condition in sense of Lott-Sturm-Villani. General references for these topics are \cite{Villani09} and \cite{AmbrosioGigli11}. We shall also make use of the recently proved (\cite{RajalaSturm12}, \cite{Gigli12a}) generalization of Brenier-McCann's theorems about optimal maps in a way that simplifies the original arguments given in \cite{Gigli13}.
\item[-] The strong maximum principle for superminimizers, proved in the context of metric measure spaces with doubling measures and supporting a weak local Poincar\'e inequality by Bjorn-Bjorn in \cite{Bjorn-Bjorn11}. Very shortly and roughly said, the argument of the proof is based on a  generalization of De Giorgi-Moser-Nash techniques for regularity  of solutions to elliptic  PDEs.
\item[-] The Gaussian estimates for the heat kernel and the Bakry-\'Emery contraction rate for the heat flow. The Gaussian estimates have been proved by Sturm in \cite{Sturm96III}, again as  generalization of  De Giorgi-Moser-Nash techniques. The Bakry-\'Emery estimate is instead a consequence of the lower bound on the Ricci curvature (in a smooth world the two are in fact equivalent) and has been proved in \cite{Gigli-Kuwada-Ohta10} on Alexandrov spaces with an approach which has been then generalized to $\RCD(K,\infty)$ spaces in \cite{AmbrosioGigliSavare11}.
\item[-] The fact that the product of two $\RCD(K,\infty)$ spaces is again $\RCD(K,\infty)$. This natural but surprisingly non-obvious result has been proved in \cite{AmbrosioGigliSavare11-2}, see also \cite{Sturm06II} for the case of $\CD(K,\infty)$ spaces.
\end{itemize}

We now turn to the description of the statement of our main result and of the general plan for its proof. 

The original version of the splitting theorem is a classical and celebrated result in Riemannian geometry proved by Cheeger-Gromoll in \cite{Cheeger-Gromoll-splitting}. Among its numerous generalizations, a crucial one has been obtained by Cheeger-Colding in \cite{Cheeger-Colding96} which extends the splitting to spaces which are measured-Gromov-Hausdorff limits of smooth Riemannian manifolds.

In \cite{Lott-Villani09} and \cite{Sturm06I},\cite{Sturm06II} Lott-Villani on one side and Sturm on the other independently proposed a definition of `having Ricci curvature bounded from below by $K$ and dimension bounded above by $N$' for metric measure spaces, these being called $\CD(K,N)$ spaces (in \cite{Lott-Villani09} only the cases $K=0$ or $N=\infty$ were considered). Here $K$ is a real number and $N$ a real number $\geq 1$, the value $N=\infty$ being also allowed. In the technically simpler case $N=\infty$ the $\CD(K,\infty)$ condition simply reads as the $K$-convexity w.r.t. the distance $W_2$ of the entropy functional relative to the reference measure.

The crucial properties of their definition are the compatibility with the smooth Riemannian case and the stability w.r.t. measured-Gromov-Hausdorff convergence. Due to such stability property and to the almost rigidity result granted by Cheeger-Colding version of the splitting, it is natural to ask whether the splitting theorem holds on $\CD(0,N)$ spaces. Unfortunately this is not the case:  as shown by Cordero-Erasquin, Sturm and Villani (see the last theorem in \cite{Villani09}), the metric measure space $(\R^d,\sfd_{\|\cdot\|},\mathcal L^d)$, where $\mathcal L^d$ is the Lebesgue measure and $\sfd_{\|\cdot\|}$ is the distance induced by the norm $\|\cdot\|$, is always a $\CD(0,d)$ space, regardless of the choice of the norm (see also \cite{Ohta09} for the curved Finsler case). In particular, if we take $d=2$ and consider a norm not coming from a scalar product, we see that although the space contains a line (many, in fact) the splitting cannot hold, because ``Pythagoras' theorem" stated in formula \eqref{eq:intropit} fails. 

The fact that geometric properties like the splitting fail in the class of $\CD(K,N)$ spaces has been source of some criticism, especially in the community of geometers (see e.g. \cite{Petrunin11}).  The question is then whether there exists another - more restrictive - synthetic notion of lower Ricci bound which  retains the stability properties and rules out Finsler-like geometries.

A proposal in this direction has been made in \cite{AmbrosioGigliSavare11-2} by the author, Ambrosio and Savar\'e for the case $N=\infty$, where the class $\RCD(K,\infty)$ has been introduced. The basic idea  is to enforce the $\CD(K,\infty)$ condition with the requirement that the heat flow is linear (see also  \cite{AmbrosioGigliMondinoRajala12}). We briefly recall the genesis of this definition. In the celebrated paper \cite{JKO98}, Jordan-Kinderlehrer-Otto showed that the heat flow can be seen as gradient flow of the relative entropy w.r.t. the $W_2$ distance on probability measures. On $\CD(K,\infty)$ spaces, the information that we have, which is in fact the only information available, is that the relative entropy is $K$-convex w.r.t. the distance $W_2$ and is therefore quite natural to study its gradient flow w.r.t. $W_2$. This has been done by the author in \cite{Gigli10}, where it has been shown that such gradient flow is unique. Notice that according to the analysis done by Ohta-Sturm in \cite{Ohta-Sturm12}, despite the fact that the normed space  $(\R^d,\sfd_{\|\cdot\|},\mathcal L^d)$ is $\CD(0,\infty)$, the distance $W_2$ never decreases along two heat flows unless the norm comes from a scalar product, in this sense the stated uniqueness result is non-trivial and obtained with a very ad-hoc argument. In \cite{Gigli10} it has been also proved that such gradient flow is stable w.r.t. mGH-convergence of compact spaces (see \cite{AmbrosioGigliSavare11-2} and \cite{Gigli-Mondino-Savare13} for generalizations). On the Euclidean space, there is at least one other way of seeing the heat flow as gradient flow: the classical viewpoint of gradient flow in $L^2$ of the Dirichlet energy. The fact that these two gradient flows produce the same evolution has been generalized in various directions. Among others, one important contribution to the topic has been made by Ohta-Sturm in \cite{Sturm-Ohta-CPAM}, where they proved that the two approaches produce the same evolution on Finlser manifolds, leading in non-Riemannian manifolds   to a non-linear evolution. It is therefore reasonable to ask whether the same sort of identification holds on general $\CD(K,\infty)$ spaces. In such setting, the role of the Dirichlet form is taken by the functional $f\mapsto E(f):=\frac12\int|\nabla f|^2\,\d\mm$, where the object $|\nabla f|$ is the 2-minimal weak upper gradient behind the definition of Sobolev functions, see Section \ref{se:sob}. Notice that $E$ is in general not a quadratic form, in line with the case of Finsler geometries. Following the strategy proposed in \cite{Gigli-Kuwada-Ohta10} for the case of Alexandrov spaces, in \cite{AmbrosioGigliSavare11} it has been proved that indeed on $\CD(K,\infty)$ spaces the two gradient flows produce the same evolution, which we can therefore undoubtedly call heat flow.

With this understanding of the heat flow, the definition of $\RCD(K,\infty)$ spaces as $\CD(K,\infty)$ spaces where such flow is linear comes out quite naturally: not only in the smooth case it singles out Riemannian manifolds from Finsler ones, but in the non-smooth world also provides a natural bridge between optimal transport theory and Sobolev calculus. Indeed, to require that the heat flow is linear is  equivalent to require that the energy functional $E$ is a quadratic form or, which is the same, that the Sobolev space $W^{1,2}$ built on our metric measure space is Hilbert. Also, the fact that on $\RCD(K,\infty)$ spaces the energy $E$ is a Dirichlet energy, allows to make connections with the the Bakry-\'Emery $\Gamma_2$ calculus, which furnishes a way to speak about lower Ricci curvature bounds for diffusion operators in the abstract context of Dirichlet forms. It turns out that the two approaches to lower Ricci curvature bounds, via optimal transport and via $\Gamma_2$ calculus, are in fact equivalent in high generality (\cite{Gigli-Kuwada-Ohta10}, \cite{AmbrosioGigliSavare11-2}, \cite{AmbrosioGigliSavare12}).

Then the appropriate finite dimensional notion of $\RCD(K,N)$ space can be introduced as\footnote{Bacher and Sturm introduced in \cite{BacherSturm10} a different notion of curvature-dimension bounds: the so called reduced-curvature-dimension, denoted as $\CD^*(K,N)$. This condition has better local-to-global properties but might produce slightly worse constants in some inequalities (an issue mitigated by the work of Cavalletti-Sturm \cite{Cavalletti-Sturm12}). Hence, one can also define the $\RCD^*(K,N)$ condition as $\CD^*(K,N)\cap \RCD(K,\infty)$. This has been the approach in \cite{Erbar-Kuwada-Sturm13} and \cite{AmbrosioMondinoSavare13}, where the link between this notion, the `dimensional' $\Gamma_2$-calculus and the `dimensional'  Bochner inequality $\Delta\frac{|\nabla f|^2}{2}\geq \frac{(\Delta f)^2}N+\nabla\Delta f\cdot\nabla f+K|\nabla f|^2$ has been established. In the particular case $K=0$ the two notions $\CD(0,N)$ and $\CD^*(0,N)$ coincide, so for what concerns the splitting theorem this distinction does not really matter.}:
\[
\RCD(K,N):=\CD(K,N)\cap \RCD(K,\infty),
\]
and the question becomes whether in this new class of spaces geometric rigidity results like the splitting hold.

Let us informally notice that  in principle it should be not too hard to prove the splitting (and the other expected geometric properties) in the non-smooth setting:  it should be sufficient to `just' follow the arguments giving  the proof in the smooth case. In a sense, \emph{if we were able to make analysis on non-smooth spaces as we are on smooth manifolds, then we would be able to deduce the same results}.
 
The problem in doing so is not really, or at least not just, the fact that the setting is non-smooth, because we already know by Cheeger-Colding that   the splitting holds in non-smooth limits of Riemannian manifolds. The point is rather the lack of all the analytic tools available in the smooth world which allow to `run the necessary computation'. Worse than this, a priori one doesn't even have the algebraic vocabulary needed to formulate those identities/inequalities that he needs. To give an example, recall that  a first ingredient of the proof of the splitting in the smooth setting is the Laplacian comparison estimate for the distance function $\Delta \sfd(\cdot,x_0)\leq \frac{N-1}{\sfd(\cdot,x_0)}$ valid in the weak sense (either viscosity or barrier or distribution sense) on manifolds with non-negative Ricci curvature and dimension bounded from above by $N$. Hence any attempt to prove the splitting in the non-smooth setting should reasonably start from proving the same inequality. However, before doing so one needs to define what such inequality means in a setting where a priori differentiation operators are not available. In other words the  path is the following:
\begin{itemize}
\item[1)] First there is \emph{algebra}, i.e. we need to develop a machinery which allows us to formally manipulate differential objects in the same way as we do in the smooth setting.
\item[2)] Then it comes \emph{analysis}, i.e. we need to show that in presence of a curvature-dimension bound, for these differential objects the same kind of estimates valid in the smooth world hold.
\item[3)] Finally there is \emph{geometry}, i.e.  once the analytic setup is established, we can try to mimic the arguments valid in the smooth world to deduce the desired geometric consequences.
\end{itemize}
These notes have been written following this  heuristic plan.

For what concerns the first `algebraic' step, it is worth to underline that the differential calculus must be developed without relying on any sort of analysis in charts, because lower bounds on the Ricci seem not sufficient to directly produce existence of charts (compare with the case of Alexandrov spaces where Perelman \cite{PerelmanDC}, improving earlier results by Otsu and Shioya \cite{OtsuShioya}, proved the existence of charts with $DC$ regularity, i.e. coordinates are Difference of Convex functions). Recall that on non-smooth limits of Riemannian manifolds, Cheeger-Colding proved in \cite{Cheeger-Colding97III} (see also \cite{Cheeger-Colding97I}, \cite{Cheeger-Colding97II}) the existence of charts with Lipschitz regularity, but their approach is based on the fact that the spaces they consider are limit of smooth manifolds, so that, very shortly said, they run the necessary computations in the smooth setting,  obtain estimates stable under convergence and then pass to the limit. As such, this technique is not applicable in the $\RCD(K,N)$ class.

There are various approaches to differential calculus on metric measure spaces, most notably by Cheeger \cite{Cheeger00} and Weaver \cite{Weaver01}, but these frameworks do not describe how to integrate by parts in a non-smooth setting. This topic has been investigated in  \cite{Gigli12} where it is has been shown how it leads to the notion of measure valued Laplacian and how to get the  natural Laplacian comparison estimates for the distance on $\CD(K,N)$ spaces.

In this direction,  a good example about how to implement the `strategy' outlined above is the Abresch-Gromoll inequality \cite{AbreschGromoll90}: in \cite{Gigli-Mosconi12}  it has been proved  how the original argument, which is based not on the smooth structure of the manifolds, but only on the Laplacian comparison estimates, the linearity of the Laplacian itself and the weak maximum principle, can be repeated verbatim on $\RCD(K,N)$ spaces leading to the same result.

For the splitting things are not so easy, essentially due to the fact that currently neither the Bochner identity nor the Hessian are available in the non-smooth worlds. Because of this, suitable modifications of the original technique need to be developed, see in particular Sections \ref{se:dist} and \ref{se:dual}.

\emph{I wish to thank Luigi Ambrosio for comments on a preliminary version of the paper}

\section*{Notation}
\addcontentsline{toc}{section}{\protect\numberline{} Notation}

In order to prove the splitting theorem we will need some intermediate constructions like the Busemann function, its gradient flow etc.  To simplify the exposition, we collect here all the objects that we will build  and references to where they are defined. 

\underline{These notations will be fixed throughout all the text}.

\bigskip

\begin{tabular}{lp{12cm}}
$(X,\sfd,\mm)$& our $\RCD(0,N)$ space containing a  line. In Section \ref{se:infhil}  we introduce  infinitesimal Hilbertianity (Definition \ref{def:infhil}), then in Section \ref{se:cd} we define the curvature-dimension bound (Definition \ref{def:cd}) and finally from Section \ref{se:basebus} on we assume the existence of a line.\\
$\b^+,\b^-$  & the two Busemann functions associated to the line in $X$. See the beginning of Section \ref{se:basebus}. \\
$\b$ & Busemann function associated to the line in $X$ and defined as $\b:=\b^+=-\b^-$ once it has been proved that $\b^++\b^-\equiv 0$. See Theorem \ref{thm:bharm} and equation \eqref{eq:defb}.\\
$\X_t$ & gradient flow of $\b$ defined $\mm$-a.e.. See Proposition \ref{prop:gfb} for its introduction and Theorem \ref{thm:gfpresmeas} for the proof of the measure preservation.\\
$\bar\X_t$ & continuous representative of $\X_t$. Introduced in Theorem \ref{thm:gfpresdist} where it is also proved that  provides a family of isometries.\\
$(X',\sfd')$ & quotient metric space obtained from $(X,\sfd)$ by identifying orbits under the action of the flow $\bar\X_t$. See Definition \ref{def:xp}.\\
$\pi$ & natural projection map from $X$ to $X'$. See Definition \ref{def:xp}.\\
$\iota$ & right inverse of $\pi$ identified by $\pi(\iota(x'))=x'$ and $\b(\iota(x'))=0$ for every $x'\in X'$. See Definition \ref{def:xp}.\\
$\mm'$ & natural `quotient' measure on $(X',\sfd')$. See Definition \ref{def:mmp}.\\
$\mau,\mad$& natural maps from $X'\times\R$ to $X$ and viceversa given by \linebreak $\mau(x',t):=\bar\X_{-t}(\iota(x'))$ and $\mad(x):=(\pi(x),\b(x))$. See Definition \ref{def:maumad}.\\
\end{tabular}

\bigskip

Notice that the proof of our main result is scattered along the text and the necessary intermediate constructions. The crucial and almost final step is in Theorem \ref{thm:pitagora}, where we prove that the maps $\mau,\mad$ are isomorphisms of the spaces  $(X,\sfd,\mm)$ and $(X'\times\R,\sfd'\times\sfd_{\rm Eucl},\mm'\times\mathcal L^1)$. The fact that  the quotient space $(X',\sfd',\mm')$ has non-negative Ricci curvature is proved in Corollary \ref{cor:xp}, while  the dimension reduction is  given by Theorem \ref{thm:dimm1}.

\section{Algebraic manipulation of basic differential objects}
\subsection{\underline{Things to know:}\  Sobolev spaces over metric measure spaces}\label{se:sob}
We will only work on proper metric spaces $(X, \sfd)$, i.e. such that closed balls are compact. By $C([0,1], X)$ we denote the space of continuous curves on $[0,1]$ with values in $X$, which we endow with the $\sup$ norm. For $t\in[0,1]$ the map  $\e_t:C([0,1],X)\to  X$ is the evaluation at time $t$ defined by
\[
\e_t(\gamma):=\gamma_t,
\]
Given a non-trivial closed interval $I\subset \R$,  a curve $\gamma:I\to X$ is said absolutely continuous provided there exists $f\in L^1(I)$ such that
\begin{equation}
\label{eq:defac}
\sfd(\gamma_t,\gamma_s)\leq \int_t^sf(r)\,\d r,\qquad\forall t,s\in I,\ t<s.
\end{equation}
It turns out (see e.g. \cite[Theorem 1.1.2]{AmbrosioGigliSavare08}) that if $\gamma$ is absolutely continuous the limit
\[
|\dot\gamma_t|:=\lim_{h\to 0}\frac{\sfd(\gamma_{t+h},\gamma_t)}{|h|},
\]
exists for a.e. $t\in I$, defines an $L^1$ function on $I$ called metric speed and this function is the minimal $f$ in the a.e. sense which can be chosen in the right hand side of \eqref{eq:defac}. In the following we will write the expression $\int_a^b|\dot\gamma_t|^2\,\d t$ even for curves which are not absolutely continuous on $[a,b]$, in this case its value is taken $+\infty$ by definition.

To define the notion of Sobolev function we need to add some structure to the metric space $(X,\sfd)$: a Radon non-negative measure $\mm$. The definition that we shall present is taken from \cite{AmbrosioGigliSavare11} (along the presentation given in \cite{Gigli12}), where also the proof of the equivalence with the notions introduced in \cite{Cheeger00} and \cite{Shanmugalingam00} is given. See also \cite{AmbrosioGigliSavare11-3}.
\begin{definition}[Test Plans]
Let $\ppi\in\prob{C([0,1], X)}$. We say that $\ppi$ is a test plan provided 
\[
(\e_t)_\sharp\ppi\leq C\mm,\qquad\forall t\in[0,1],
\]
for some constant $C>0$, and 
\[
\iint_0^1|\dot\gamma_t|^2\,\d t\ppi(\gamma)<\infty.
\]
\end{definition}
Notice that  according to the convention $\int_0^1|\dot\gamma_t|^2\,\d t=+\infty$ if $\gamma$ is not absolutely continuous, any test plan must be concentrated on absolutely continuous curves.
\begin{definition}[The Sobolev class $\s^2(X,\sfd, \mm)$]
The Sobolev class $\s^2(X,\sfd,\mm)$ (resp.\linebreak $\s^2_{\rm loc}(X,\sfd,\mm)$) is the space of all Borel functions $f:X\to\R$ such that there exists a non-negative $G\in L^2(X, \mm)$ (resp. $G\in L^2_{\rm loc}(X, \mm)$) for which it holds
\begin{equation}
\label{eq:defsob}
\int|f(\gamma_1)-f(\gamma_0)|\,\d\ppi(\gamma)\leq \iint_0^1G(\gamma_t)|\dot\gamma_t|\,\d t\,\d\ppi(\gamma),\qquad\forall \ppi\textrm{ test plan}.
\end{equation}
\end{definition}
It turns out that for $f\in \s^2(X,\sfd,\mm)$ there exists a minimal $G$ in the $\mm$-a.e. sense for which \eqref{eq:defsob} holds: we will denote it by $\weakgrad f$ and call it minimal weak upper gradient. Notice that in fact both the notation and the terminology are misleading, because being this object defined in duality with speed of curves, it is closer to the norm of a cotangent vector rather to a tangent one. Yet, from the next section on we are going to make the assumption that the space is `infinitesimally Hilbertian' which in a sense allows to identify  differential and gradients (see in particular the symmetry relation \eqref{eq:simm}), so that it is quite safe to denote by $\weakgrad f$ the minimal weak upper `gradient'.
The minimal weak upper gradient $\weakgrad f$ is a local object, in the sense that for $f\in\s^2_{\rm loc}(X,\sfd,\mm)$ we have
\begin{align}
\label{eq:nullgrad}
\weakgrad f&=0,\qquad&&\textrm{on }f^{-1}(N),\qquad\forall N\subset \R\textrm{, Borel with }\mathcal L^1(N)=0,\\
\label{eq:localgrad}
\weakgrad f&=\weakgrad g,\qquad&&\mm\textrm{-a.e.\ on }\{f=g\},\ \forall f,g\in\s^2_{\rm loc}(X,\sfd,\mm).
\end{align}
Also, for any  $\ppi$ test plan and $t<s\in[0,1]$ it holds
\begin{equation}
\label{eq:localplan}
|f(\gamma_s)-f(\gamma_t)|\leq \int_t^s\weakgrad f(\gamma_r)|\dot\gamma_r|\,\d r,\qquad\ppi\ae\ \gamma.
\end{equation}

In particular, the definition of Sobolev class can be directly localized to produce the notion of Sobolev function defined on an open set $\Omega\subset  X$:
\begin{definition}\label{def:parigi}
Let $\Omega\subset X$ be an open set. A Borel function $f:\Omega\to \R$ belongs to $\s^2_{\rm loc}(\Omega,\sfd, \mm)$ provided for any Lipschitz function $\nchi: X\to \R$ with $\supp(\nchi)\subset \Omega$ it holds $f\nchi\in \s^2_{\rm loc}(X, \sfd,\mm)$. In this case, the function $\weakgrad f:\Omega\to [0,\infty]$ is $\mm$-a.e. defined by
\[
\weakgrad{f}:=\weakgrad{(\nchi f)},\qquad \mm\ae \ {\rm on}\ \nchi\equiv1,
\] 
for any $\nchi$ as above. Notice that thanks to \eqref{eq:localgrad} this is a good definition. The space $\s^2(\Omega)\subset\s^2_{\rm loc}(\Omega)$ is the set of $f$'s such that $\weakgrad f\in L^2(\Omega, \mm)$.
\end{definition}

The basic calculus properties of Sobolev functions are collected below.  $\Omega\subset  X$ is open and all the (in)equalities are intended $\mm$-a.e. on $\Omega$.

\noindent\underline{Lower semicontinuity of minimal weak upper gradients}. Let $(f_n)\subset \s^2(\Omega, \sfd,\mm)$ and $f:\Omega\to\R$ be such that $f_n(x)\to f(x)$ as $n\to\infty$ for $\mm$-a.e. $x\in\Omega$. Assume that $(\weakgrad {f_n})$ converges to some $G\in L^2(\Omega, \mm)$ weakly in $L^2(\Omega, \mm)$.

Then 
\begin{equation}
\label{eq:lscwug}
f\in \s^2(\Omega)\qquad\text{ and }\qquad\weakgrad f\leq  G,\quad \mm\ae.
\end{equation}

\noindent\underline{Weak gradients and local Lipschitz constants}. For any $f:\Omega\to\R$ locally Lipschitz it holds
\begin{equation}
\label{eq:lipweak}
\weakgrad f\leq \lip^\pm(f)\leq \lip(f),
\end{equation}
where the functions  $\lip^\pm(f),\lip(f):\Omega\to\R^+$ denote the slopes and the  local Lipschitz constant defined by
\[
\lip^+(f)(x):=\lims_{y\to x}\frac{(f(y)-f(x))^+}{\sfd(y,x)},\qquad\lip^-(f)(x):=\lims_{y\to x}\frac{(f(y)-f(x))^-}{\sfd(y,x)},
\]
and $\lip(f):=\max\{\lip^-(f),\lip^+(f)\}$ at points $x\in\Omega$ which are not isolated, 0 otherwise.

\noindent\underline{Vector space structure}. $\s^2_{\rm loc}(\Omega, \sfd,\mm)$ is a vector space and 
\begin{equation}
\label{eq:vectorstru}
\weakgrad{(\alpha f+\beta g)}\leq |\alpha|\weakgrad f+|\beta|\weakgrad g,\qquad\textrm{for any $f,g\in\s^2_{\rm loc}(\Omega, \sfd,\mm)$, $\alpha,\beta\in\R$,}
\end{equation}
similarly for  $\s^2(\Omega,\sfd, \mm)$.

\noindent\underline{Algebra structure}.  $\s^2_{\rm loc}(\Omega, \sfd,\mm)\cap L^\infty_{\rm loc}(\Omega, \mm)$ is an algebra and
\begin{equation}
\label{eq:leibbase}
\weakgrad{(fg)}\leq |f|\weakgrad g+g\weakgrad f,\qquad\textrm{for any $f,g\in \s^2_{\rm loc}(\Omega, \sfd,\mm)\cap L^\infty_{\rm loc}(\Omega,\mm)$,}
\end{equation}
and analogously  for the space  $\s^2(\Omega,\sfd, \mm)\cap L^\infty(\Omega, \mm)$. Similarly, if $f\in \s^2_{\rm loc}(\Omega, \sfd,\mm)$ and $g:\Omega\to\R$ is locally Lipschitz, then $fg\in\s^2_{\rm loc}(\Omega,\sfd,\mm)$ and the bound \eqref{eq:leibbase} holds.

\noindent\underline{Chain rule}. Let  $f\in \s^2_{\rm loc}(\Omega,\sfd,\mm)$ and $\varphi:\R\to \R$ Lipschitz. Then $\varphi\circ f\in \s^2_{\rm loc}(\Omega, \sfd, \mm)$ and
\[
\weakgrad{(\varphi\circ f)}=|\varphi'|\circ f\weakgrad f,
\]
where $|\varphi'|\circ f$ is defined arbitrarily at points where $\varphi$ is not differentiable (observe that the identity \eqref{eq:nullgrad} ensures that on $f^{-1}(\mathcal N)$ both $\weakgrad{(\varphi\circ f)}$ and $\weakgrad f$ are 0 $ \mm$-a.e., $\mathcal N$ being the negligible set of points of non-differentiability of $\varphi$). In particular, if $f\in \s^2(\Omega, \sfd, \mm)$, then $\varphi\circ f\in \s^2(\Omega, \sfd, \mm)$ as well.

\medskip

Finally, we remark that from the definition of Sobolev class it is easy to produce the one of Sobolev space $W^{1,2}(\Omega, \sfd, \mm)$ for $\Omega\subset X$ open: it is sufficient to put 
\begin{equation}
\label{eq:w12}
W^{1,2}(\Omega, \sfd,\mm):=L^2(\Omega, \mm)\cap\s^2(\Omega, \sfd, \mm)
\end{equation}
the corresponding $W^{1,2}$-norm being given by
\begin{equation}
\label{eq:w12norm}
\|f\|_{W^{1,2}(\Omega)}^2:=\|f\|^2_{L^2(\Omega)}+\|\weakgrad f\|^2_{L^2(\Omega)}=\int_\Omega f^2+\weakgrad f^2\,\d \mm.
\end{equation}
It is obvious that $\|\cdot\|_{W^{1,2}(\Omega)}$ is a norm on $W^{1,2}(\Omega)$. The completeness of the space is then a consequence of the lower semicontinuity property \eqref{eq:lscwug}, see for instance the argument in \cite{Cheeger00}. Hence $W^{1,2}(\Omega, \sfd, \mm)$ is always a Banach space although  in general not an Hilbert space.

\medskip

\emph{To simplify the notation, in the following we will often write $W^{1,2}(X)$, $\s^2_{\rm loc}(X)$, $\s^2_{\rm loc}(\Omega)$ ecc. in place of  $W^{1,2}(X,\sfd,\mm)$, $\s^2_{\rm loc}(X,\sfd, \mm)$, $\s^2_{\rm loc}(\Omega, \sfd,\mm)$. Similarly, we will write $L^p( X)$, $L^p(\Omega)$, $L^p_{\rm loc}(\Omega)$ in place of $L^p(X,\mm)$, $L^p(\Omega, \mm)$, $L^p_{\rm loc}(\Omega, \mm)$.}

\medskip

 In \cite{AmbrosioGigliSavare11} the following approximation result has been proved, previously known statements required the measure to be doubling and the space to support a 1-2 weak local Poincar\'e inequality (see  e.g. the argument in Theorem 5.1 of  \cite{Bjorn-Bjorn11} which gives a Lusin's type approximation under this further assumptions):
\begin{theorem}[Density in energy of Lipschitz functions in $W^{1,2}(X)$]\label{thm:energylip}
Let $(X,\sfd,\mm)$ be a proper metric measure space.

Then Lipschitz functions are dense in energy in $W^{1,2}(X)$, i.e. for any $f\in W^{1,2}(X)$ there exists a sequence $(f_n)\subset W^{1,2}(X)$ of Lipschitz functions such that $f_n\to f$, $\weakgrad {f_n}\to\weakgrad f$  in $L^2(X)$.

Furthermore, these $f_n$'s can be chosen with compact support for every $n\in\N$ and to satisfy $\lip(f_n)\to\weakgrad f$ in $L^2(X)$ as $n\to\infty$.
\end{theorem}
\subsection{Infinitesimally Hilbertian spaces and the object $\la\nabla f,\nabla g\ra$}\label{se:infhil}
From this section on  we will focus on those metric measure spaces which, from the Sobolev calculus' point of view, resemble a Riemannian structure rather than a general Finsler one.  The definition as well as the foregoing discussion comes from \cite{Gigli12}, which in turn is based and extends the analysis done in \cite{AmbrosioGigliSavare11-2}.
\begin{definition}[Infinitesimally Hilbertian spaces]\label{def:infhil}
Let $(X,\sfd,\mm)$ be a proper metric measure space. We say that it is infinitesimally Hilbertian provided $W^{1,2}(X,\sfd,\mm)$ is an Hilbert space.
\end{definition}
We already noticed that $W^{1,2}(X)$ is always a Banach space, so to ask that it is Hilbert is equivalent to ask that the $W^{1,2}$-norm satisfies the parallelogram rule. From the definition \eqref{eq:w12norm} and the fact that the $L^2(X,\mm)$-norm certainly satisfies the parallelogram rule, we see that $(X,\sfd,\mm)$ is infinitesimally Hilbertian if and only if
\begin{equation}
\label{eq:infhil}
\|\weakgrad{(f+g)}\|_{L^2}^2+\|\weakgrad{(f-g)}\|_{L^2}^2=2\Big(\|\weakgrad{f}\|_{L^2}^2+\|\weakgrad{g}\|_{L^2}^2\Big),\qquad\forall f,g\in\s^2(X).
\end{equation}
Notice that thanks to the uniform convexity of $W^{1,2}(X)$, on infinitesimally Hilbertian spaces Theorem \ref{thm:energylip} immediately yields the following statement:
\begin{theorem}[Density in $W^{1,2}$-norm of Lipschitz functions]\label{thm:stronglip}
Let $(X,\sfd,\mm)$ be an infinitesimally Hilbertian space.

Then Lipschitz functions are dense in  $W^{1,2}(X)$, i.e. for any $f\in W^{1,2}(X)$ there exists a sequence $(f_n)\subset W^{1,2}(X)$ of Lipschitz functions such that $f_n\to f$, $\weakgrad {(f_n-f)}\to0$  as $n\to\infty$ in $L^2(X)$.

Furthermore, these $f_n$'s can be chosen with compact support for every $n\in\N$ and to satisfy $\lip(f_n)\to\weakgrad f$ in $L^2(X)$ as $n\to\infty$.
\end{theorem}

On infinitesimally Hilbertian spaces and for given Sobolev functions $f,g$  one can define a bilinear object $\la\nabla f, \nabla g\ra$ which plays the role of  the scalar product of their gradients. This can be done without really defining what the gradient of a Sobolev function actually is, as in metric measure spaces this notion requires more care (see e.g. \cite{Weaver01} and \cite{Gigli12}). Thus, the spirit of the definition is similar to the one that leads to the definition of the carr\'e du champ $\Gamma(f,g)$ in the context of Dirichlet forms. Actually, on infinitesimally Hilbertian spaces the map 
\[
W^{1,2}(X,\sfd,\mm)\ni f\qquad\mapsto \qquad\int_X\weakgrad f^2\,\d\mm, 
\]
is a regular and strongly local Dirichlet form on $L^2(X,\mm)$, so that the object $\la\nabla f,\nabla g\ra$ that we are going to define could actually be introduced just as the carr\'e du champ $\Gamma(f,g)$ associated to this Dirichlet form. Yet, we are going to use a  different definition and a different notation since our structure is richer than the one available when working with  abstract Dirichlet forms, because we have a metric measure space $(X,\sfd,\mm)$ satisfying the assumption \eqref{eq:infhil} and not only a topological space $(X,\tau)$ endowed with a measure $\mm$ and a Dirichlet form $\mathcal E$. One of the effects of this additional structure is that in our context it is already given the $\mm$-a.e. value of `the modulus $\weakgrad f$ of the gradient of $f$', while in the context of Dirichlet forms this has to be built. Also, it is worth to notice that the definition \ref{def:nablafnablag} given below makes sense even on spaces which are not infinitesimally Hilbertian and in this higher generality provides a reasonable definition of what is `the differential of $f$ applied to the gradient of $g$' (see \cite{Gigli12}). In this sense, the approach we propose is more general than the one available in the `linear' framework  of Dirichlet form and formula \eqref{eq:defnablanabla} can be seen as a sort of nonlinear variant of the polarization identity. 

\begin{definition}[The object $\la\nabla f,\nabla g\ra$]\label{def:nablafnablag}
Let $(X,\sfd,\mm)$ be an infinitesimally Hilbertian space, $\Omega\subseteq X$ an open set and $f,g\in \s^2_{\rm loc}(\Omega)$. The map $\la\nabla f,\nabla g\ra:\Omega\to\R$ is $\mm$-a.e. defined as
\begin{equation}
\label{eq:defnablanabla}
\la\nabla f,\nabla g\ra:=\inf_{\eps>0}\frac{\weakgrad{(g+\eps f)}^2-\weakgrad g^2}{2\eps},
\end{equation}
the infimum being intended as $\mm$-essential infimum.
\end{definition}
Notice that as a direct consequence of the locality stated in \eqref{eq:localgrad}, also the object $\la\nabla f,\nabla g\ra$ is local, i.e.:
\begin{equation}
\label{eq:local}
\la\nabla f,\nabla g\ra=\la\nabla \tilde f,\nabla \tilde g\ra,\qquad\mm\text{\rm-a.e.  on}  \ \{f=\tilde f\}\cap\{g=\tilde g\}\cap\Omega.
\end{equation}
In the following theorem we collect the main properties of $\la\nabla f,\nabla g\ra$, showing that the expected algebraic calculus rules hold.
\begin{theorem}\label{thm:calculus}
Let $(X,\sfd,\mm)$ be infinitesimally Hilbertian and $\Omega\subseteq X$ an open set. 

Then $W^{1,2}(\Omega)$ is an Hilbert space and the following hold.
\begin{itemize}
\item\underline{`Cauchy-Schwartz'.} For any $f,g\in\s^2_{\rm loc}(\Omega)$ it holds
\begin{align}
\label{eq:1}
\la\nabla f,\nabla f\ra&=\weakgrad f^2,\\
\label{eq:boundfg}
\big|\la\nabla f,\nabla g\ra\big|&\leq \weakgrad f\weakgrad g,
\end{align}
$\mm$-a.e. on $\Omega$.
\item\underline{Linearity in $f$.} For any  $f_1,f_2,g\in \s^2_{\rm loc}(\Omega)$ and $\alpha,\beta\in\R$ it holds
\begin{equation}
\label{eq:linf}
\la\nabla(\alpha f_1+\beta f_2),\nabla g\ra=\alpha\la\nabla f_1,\nabla g\ra+\beta \la\nabla f_2,\nabla g\ra,\qquad\mm\text{\rm -a.e. on }\Omega.
\end{equation}
\item\underline{Chain rule in $f$.}
Let $f\in\s^2_{\rm loc}(\Omega)$ and $\varphi:\R\to\R$  Lipschitz. Then for any $g\in\s^2_{\rm loc}(\Omega)$ it holds
\begin{equation}
\label{eq:chainf}
\la\nabla(\varphi\circ f),\nabla g\ra=\varphi'\circ f\la\nabla f,\nabla g\ra,\qquad\mm\text{\rm-a.e. on }\Omega,
\end{equation}
where the value of $\varphi'\circ f$ is taken arbitrary on those $x\in\Omega$ such that $\varphi$ is not differentiable at $f(x)$. 
\item\underline{Leibniz rule in $f$.}
For $f_1,f_2\in\s^2_{\rm loc}(\Omega)\cap L^{\infty}_{\rm loc}(\Omega)$ and $g\in\s^2_{\rm loc}(\Omega)$ the Leibniz rule
\begin{equation}
\label{eq:leibf}
\la\nabla (f_1f_2),\nabla g\ra=f_1\la\nabla f_2,\nabla g\ra+f_2\la\nabla f_1,\nabla g\ra,\qquad\mm\text{\rm-a.e. on }\Omega,
\end{equation}
holds. 
\item\underline{Symmetry.} For any $f,g\in\s^2_{\rm loc}(\Omega)$ it holds
\begin{equation}
\label{eq:simm}
\la\nabla f,\nabla g\ra=\la\nabla g,\nabla f\ra,\qquad\mm\text{\rm-a.e. on }\Omega.
\end{equation}
In particular, the object $\la\nabla f,\nabla g\ra$ is also linear in $g$ and satisfies chain and Leibniz rules in $g$ analogous to those valid for $f$.
\end{itemize}
\end{theorem}
\begin{proof} The fact that $W^{1,2}(\Omega)$ is Hilbert is a direct consequence of the stated algebraic properties. Such properties are expressed as $\mm$-a.e. equalities on $\Omega$, hence, by the very definition of $\s^2_{\rm loc}(\Omega)$ and the locality property \eqref{eq:local}, to conclude it is sufficient to deal with the case of $\Omega=X$ and functions in $\s^2(X,\sfd,\mm)$.

The identity \eqref{eq:1} is a direct consequence of the definition. Taking into account that $\weakgrad{(g+\eps f)}\leq \weakgrad g+\eps\weakgrad f$ for any $\eps>0$, we get
\begin{equation}
\label{eq:lato1}
\la\nabla f,\nabla g\ra\leq \weakgrad f\weakgrad g,\qquad\mm\textrm{-a.e..}
\end{equation}
From the  inequality  \eqref{eq:vectorstru} we get that the map $\s^2(X,\sfd,\mm) \ni f\mapsto\weakgrad{f}$ is $\mm$-a.e. convex, in the sense that
\[
\weakgrad{((1-\lambda)f+\lambda g)}\leq(1-\lambda)\weakgrad f+\lambda\weakgrad g,\qquad\mm\text{\rm-a.e.}\quad \forall f,g\in \s^2(X,\sfd,\mm),\ \lambda\in[0,1].
\]
It follows that $\R\ni\eps\mapsto\weakgrad{(g+\eps f)}$ is also $\mm$-a.e. convex and, being non-negative, also $\R\ni\eps\mapsto\weakgrad{(g+\eps f)}^2/2$ is $\mm$-a.e. convex in the sense of the above inequality. In particular, the $\inf_{\eps>0}$ in definition \eqref{eq:defnablanabla} can be substituted with $\lim_{\eps\downarrow0}$ in $L^1(X,\mm)$, and thus  we  easily get that for given $g\in \s^2(X,\sfd,\mm)$
\begin{equation}
\label{eq:1conv}
\textrm{the map }\s^2(X,\sfd,\mm)\ni f\quad\mapsto\quad\la\nabla f,\nabla g\ra\textrm{ is }\mm\textrm{-a.e. positively 1-homogeneous and convex},
\end{equation}
and that
\[
\frac{\weakgrad{(g+\eps f)}^2-\weakgrad g^2}{2\eps}\leq \frac{\weakgrad{(g+\eps' f)}^2-\weakgrad g^2}{2\eps'},\qquad\mm\text{\rm-a.e.}\quad \forall \eps,\eps'\in\R\setminus\{0\},\ \eps\leq \eps',
\]
so that we obtain $\mm$-a.e.:
\begin{equation}
\label{eq:nonuguali}
\la\nabla f,\nabla g\ra=\inf_{\eps>0}\frac{\weakgrad{(g+\eps f)}^2-\weakgrad g^2}{2\eps}\geq\sup_{\eps<0}\frac{\weakgrad{(g+\eps f)}^2-\weakgrad g^2}{2\eps}=-\la\nabla(- f),\nabla g\ra.
\end{equation}
Now plug  $\eps f$ in place of $f$ in \eqref{eq:infhil} to get 
\[
\int \frac{\weakgrad{(g+\eps f)}^2-\weakgrad g^2}{2\eps}\,\d\mm=-\int \frac{\weakgrad{(g-\eps f)}^2-\weakgrad g^2}{2\eps}\,\d\mm+\eps\int\weakgrad f^2\,\d\mm.
\]
Letting $\eps\downarrow0$ we obtain $\int \la\nabla f,\nabla g\ra\d\mm=-\int\la\nabla (-f),\nabla g\ra\d\mm $, which by \eqref{eq:nonuguali} forces 
\begin{equation}
\label{eq:uguali}
\la\nabla f,\nabla g\ra=-\la\nabla(-f),\nabla g\ra,
\end{equation}
and in particular, by \eqref{eq:lato1}, we deduce \eqref{eq:boundfg}. 

For given $g\in\s^2(X,\sfd,\mm)$, \eqref{eq:1conv} yields that $f\mapsto -\la\nabla (-f), \nabla g\ra $ is $\mm$-a.e. positively 1-homogeneous and concave, hence from \eqref{eq:uguali} we deduce the linearity in $f$ of $\la\nabla f,\nabla g\ra$, i.e. \eqref{eq:linf} is proved.

We now turn to the chain rule in \eqref{eq:chainf}.  Notice that the linearity in $f$ and the inequality \eqref{eq:boundfg} immediately yield
\begin{equation}
\label{eq:1lip}
\big|\la\nabla f,\nabla g\ra-\la\nabla \tilde f,\nabla g\ra\big|\leq \weakgrad{(f-\tilde f)}\weakgrad{g}.
\end{equation}
Moreover, thanks to \eqref{eq:linf}, \eqref{eq:chainf} is obvious if  $\varphi$ is linear, and since \eqref{eq:chainf} is unchanged if we add a constant to $\varphi$, it is also true if $\varphi$ is affine. Then, using the locality property \eqref{eq:local}  we also get \eqref{eq:chainf} for $\varphi$ piecewise affine (notice that the property \eqref{eq:nullgrad} ensures that letting $\mathcal N\subset\R$ be the negligible points of non-differentiability of $\varphi$, both $|\nabla(\varphi\circ f)|$ and $|\nabla f|$ are 0 $\mm$-a.e. on $f^{-1}(\mathcal N)$). To conclude in the general case, let $\varphi$ be an arbitrary Lipschitz function and find a sequence $(\varphi_n)$ of piecewise affine functions such that $\varphi_n'(z)\to\varphi'(z)$ as $n\to\infty$ for $\mathcal L^1$-a.e. $z\in\R$. Let $\mathcal N\subset\R$ be the union of the set of points of non-differentiability of $\varphi$ and the $\varphi_n$'s with the set of $z$ such that $\varphi_n'(z)\not\to\varphi'(z)$. Then $\mathcal N$ is a Borel negligible set, and thus \eqref{eq:nullgrad} gives
\[
\varphi'_n\circ f\la\nabla f,\nabla g\ra\quad\to\quad\varphi'\circ f\la\nabla f,\nabla g\ra,\qquad\mm\textrm{-a.e.},
\]
and similarly
\[
\big|\la\nabla (\varphi\circ f),\nabla g\ra-\la\nabla (\varphi_n\circ f),\nabla g\ra\big|\leq\weakgrad{((\varphi-\varphi_n)\circ f )}\weakgrad g= |\varphi'-\varphi_n'|\circ f\weakgrad f\weakgrad g\to 0,
\]
$\mm$-a.e.. The chain rule \eqref{eq:chainf} follows.

The Leibniz rule \eqref{eq:leibf} is a consequence of the chain rule  \eqref{eq:chainf} and the linearity \eqref{eq:linf}: indeed, up to adding a constant to both $f_1$ and $f_2$, we can assume that $\mm$-a.e. it holds $f_1,f_2\geq c$ for some $c>0$, then notice that from \eqref{eq:chainf} and \eqref{eq:linf} we get
\[
\begin{split}
\la\nabla (f_1f_2),\nabla g\ra&=f_1f_2\la\nabla(\log(f_1f_2)),\nabla g\ra=f_1f_2\la\nabla(\log f_1+\log f_2),\nabla g\ra\\
&=f_1f_2\big(\la\nabla(\log f_1),\nabla g\ra+\la\nabla(\log f_2),\nabla g\ra\big)\\
&=f_1f_2\left(\frac1{f_1}\la\nabla f_1,\nabla g\ra+\frac1{f_2}\la\nabla f_2,\nabla g\ra\right)=f_2\la\nabla f_1,\nabla g\ra+f_1\la\nabla f_2,\nabla g\ra.
\end{split}
\]
To conclude it is now sufficient to show the symmetry relation \eqref{eq:simm}. For this we shall need some auxiliary intermediate results. The first one concerns continuity in $g$ of the map $\s^2(X,\sfd,\mm)\ni g\mapsto\int \la\nabla f,\nabla g\ra\d\mm$. More precisely, we claim that
\begin{equation}
\label{eq:contg}
\begin{split}
&\text{\rm given a sequence $(g_n)\subset \s^2(X,\sfd,\mm)$ and $g\in\s^2(X,\sfd,\mm)$ such that}\\
&\text{\rm $\lim_{n\to\infty}\int\weakgrad{(g_n-g)}^2\,\d\mm=0$, for any $f\in\s^2(X,\sfd,\mm)$ it holds}\\
&\lim_{n\to\infty}\int\la\nabla f,\nabla g_n\ra\d\mm=\int\la\nabla f,\nabla g\ra\d\mm.
\end{split}\end{equation}
To see this, notice that for any $\eps\neq 0$ and under the same assumptions it holds
\[
\lim_{n\to\infty}\int \frac{\weakgrad{(g_n+\eps f)}^2-\weakgrad{g_n}^2}{\eps}\,\d\mm=\int \frac{\weakgrad{(g+\eps f)}^2-\weakgrad{g}^2}{\eps}\,\d\mm.
\]
Now recall that $\R^+\ni\eps\mapsto\frac{\weakgrad{(g_n+\eps f)}^2-\weakgrad{g_n}^2}{\eps}$ is $\mm$-a.e. increasing and converges to $\la\nabla f,\nabla g_n\ra$ $\mm$-a.e. as $\eps\downarrow 0$ to get 
\[
\lims_{n\to\infty}\int\la\nabla f,\nabla g_n\ra\,\d\mm\leq\lim_{n\to\infty}\int \frac{\weakgrad{(g_n+\eps f)}^2-\weakgrad{g_n}^2}{\eps}\,\d\mm=\int \frac{\weakgrad{(g+\eps f)}^2-\weakgrad{g}^2}{\eps}\,\d\mm,
\]
and eventually passing to the limit as $\eps\downarrow 0$ we deduce
\[
\lims_{n\to\infty}\int\la\nabla f,\nabla g_n\ra\d\mm\leq\int\la\nabla f,\nabla g\ra\d\mm.
\]
The $\limi$ inequality then follows replacing $f$ with $-f$ and using  linearity in $f$ expressed in \eqref{eq:linf}.

We shall use \eqref{eq:contg} to obtain an integrated chain rule for $g$, i.e.:
\begin{equation}
\label{eq:chaingint}
\int \varphi'\circ g\la\nabla f,\nabla g\ra\d\mm=\int\la\nabla f,\nabla(\varphi\circ g)\ra\d\mm.
\end{equation}
To get this, start observing that letting $\eps\downarrow 0$ in the trivial identity
\[
\frac{\weakgrad{(\alpha g+\eps f)}^2-\weakgrad{(\alpha g)}^2}{2\eps}=\alpha\frac{\weakgrad{(g+\frac{\eps}\alpha f)}^2-\weakgrad{ g}^2}{2\frac\eps\alpha},\qquad\alpha\neq 0,
\]
and recalling the linearity in $f$ \eqref{eq:linf}, we obtain 1-homogeneity in $g$, i.e.
\[
\la\nabla f,\nabla (\alpha g)\ra=\alpha\la\nabla f,\nabla g\ra,\qquad\forall\alpha\in\R.
\]
From the locality property \eqref{eq:local} we then get that for $\varphi:\R\to\R$ piecewise affine it holds
\begin{equation}
\label{eq:quasichain}
\la\nabla f,\nabla(\varphi\circ g)\ra=\varphi'\circ g\la\nabla f,\nabla g\ra,\qquad\mm\text{\rm-a.e.},
\end{equation}
where, as before, the value of $\varphi\circ g$ can be chosen arbitrary at those $x$ such that $\varphi$ is not differentiable in $g(x)$. To conclude we argue as in the proof of \eqref{eq:chainf} using \eqref{eq:contg} in place of \eqref{eq:1lip}. More precisely, given $\varphi:\R\to\R$ Lipschitz we find a sequence $(\varphi_n)$ of uniformly Lipschitz piecewise affine functions such that $\varphi'_n(z)\to\varphi'(z)$ for $\mathcal L^1$-a.e.  $z$.

From $\weakgrad{(\varphi\circ g-\varphi_n\circ g)}=|\varphi'-\varphi'_n|\circ g\weakgrad g\to 0$ $\mm$-a.e. and the fact that $\varphi,\varphi_n$, $n\in\N$, are uniformly Lipschitz we get $\lim_{n\to\infty}\int\weakgrad{(\varphi\circ g-\varphi_n\circ g)}^2\,\d\mm\to 0$. Thus from \eqref{eq:contg} and \eqref{eq:quasichain} we conclude
\[
\begin{split}
\int\la\nabla f,\nabla (\varphi\circ g)\ra&=\lim_{n\to\infty}\int\la\nabla f,\nabla(\varphi_n\circ g)\ra\d\mm\\
&=\lim_{n\to\infty}\int\varphi_n'\circ g\la\nabla f,\nabla g\ra\d\mm=\int\varphi'\circ g\la\nabla f,\nabla g\ra\d\mm,
\end{split}
\]
having used the dominated convergence theorem in the last step.

The last ingredient we need to prove the symmetry property \eqref{eq:simm} is its integrated version
\begin{equation}
\label{eq:simmint}
\int\la\nabla f,\nabla g\ra\d\mm=\int\la\nabla g,\nabla f\ra\d\mm.
\end{equation}
This easily follows by noticing that the assumption of infinitesimal Hilbertianity yields
\begin{equation}
\label{eq:camel}
\int\frac{\weakgrad{(g+\eps f)}^2-\weakgrad{g}^2}{\eps}-\eps\weakgrad f^2\,\d\mm=\int\frac{\weakgrad{(f+\eps g)}^2-\weakgrad{f}^2}{\eps}-\eps\weakgrad g^2\,\d\mm,
\end{equation}
and then letting $\eps\downarrow 0$.

Now notice that \eqref{eq:simm} is equivalent to the fact that for any $h\in L^\infty(X,\mm)$ it holds 
\begin{equation}
\label{eq:trenino}
\int h\la\nabla f,\nabla g\ra\d\mm=\int h\la\nabla g,\nabla f\ra\d\mm,\qquad\forall f,g\in\s^2(X,\sfd,\mm).
\end{equation}
Taking into account the weak$^*$-density of Lipschitz and bounded functions in $L^\infty(X,\mm)$, we easily see that it is sufficient to check \eqref{eq:trenino} for any $h$ Lipschitz and bounded. Also, with the same arguments that led from \eqref{eq:camel} to \eqref{eq:simmint} and a simple truncation argument,  \eqref{eq:trenino} will follow if we show that 
\begin{equation}
\label{eq:simm2}
\s^2(X,\sfd,\mm)\cap L^\infty(X,\mm)\ni f\qquad\mapsto\qquad \int h\weakgrad f^2\,\d\mm\in\R\qquad\textrm{is a quadratic form}.
\end{equation}
To this aim, notice that from \eqref{eq:leibf}, \eqref{eq:chaingint} and \eqref{eq:simmint} we get
\begin{equation}
\label{eq:leibsimm}
\begin{split}
\int h\weakgrad f^2\,\d\mm&=\int \la\nabla (fh),\nabla f\ra-f\la\nabla h,\nabla f\ra\d\mm\\
&=\int \la\nabla (fh),\nabla f\ra-\la\nabla h,\nabla \Big(\frac{f^2}2\Big)\ra\d\mm=\int \la\nabla (fh),\nabla f\ra-\la\nabla \Big(\frac{f^2}2\Big), \nabla h\ra\d\mm.
\end{split}
\end{equation}
By \eqref{eq:linf} and \eqref{eq:simmint} we know that both $f\mapsto \int \la\nabla (fh),\nabla\tilde f\ra\d\mm$ and $f\mapsto\int\la\nabla (\tilde fh),\nabla f\ra\d\mm$ are linear maps, hence $f\mapsto\int\la \nabla (fh),\nabla f\ra\d\mm$ is a quadratic form. Again by \eqref{eq:linf} we also get that $f\mapsto \int \la\nabla(\frac {f^2}2), \nabla h\ra\d\mm$ is a quadratic form. Hence \eqref{eq:leibsimm} yields \eqref{eq:simm2} and the conclusion.
\end{proof}
We remark that during the proof we showed  that $\la\nabla f, \nabla g\ra$ can be realized as limit rather than as infimum, i.e. it holds
\begin{equation}
\label{eq:limite}
\la\nabla f,\nabla g\ra=\lim_{\eps\to 0}\frac{\weakgrad{(g+\eps f)}^2-\weakgrad g^2}{2\eps},\qquad\forall f,g\in\s^2(\Omega),
\end{equation}
the limit being intended both in $L^1(\Omega)$ and in the essential-$\mm$-a.e. sense.

\subsection{Horizontal and vertical derivatives, i.e. first order differentiation formula}
The definition of $\la\nabla f,\nabla g\ra$ that we just provided has all the basic expected algebraic properties one wishes. Yet, it does not really answer the question `what is the derivative of $f$ along the direction $\nabla g$?' The way we defined it, this object is obtained by a `vertical' derivative, i.e. by a perturbation in the dependent variable, while the essence of derivation is to take `horizontal' derivatives, i.e. perturbations in the independent variable. Notice indeed that in a smooth Riemannian world, the value of $\la\nabla f,\nabla g\ra(x)$ (more precisely: of the differential of $f$ applied to the gradient of $g$) is  defined as $\lim_{t\downarrow0}\frac{f(\gamma_t)-f(\gamma_0)}t$, where $\gamma$ is any smooth curve with $\gamma_0=x$ and $\gamma_0'=\nabla g(x)$. It is therefore natural to ask whether a similar approach exists in the non-smooth setting and if it provides the same calculus as given by Theorem \ref{thm:calculus}. It turns out that the answer is yes, see Theorem \ref{thm:horver} below: this result, appeared first in \cite{AmbrosioGigliSavare11-2} and then generalized in \cite{Gigli12}, should be considered as the single most important contribution to differential calculus on metric measure spaces among those presented in such papers.

\bigskip

Obviously on a non-smooth structure it makes no sense to say that a curve $\gamma$ satisfies $\gamma'_0=\nabla g(x)$. Yet, we can implicitly give a meaning to this expression mimicking De Giorgi's definition of gradient flow in metric spaces (see \cite{AmbrosioGigliSavare08}) arguing as follows.  Let  $g\in\s^2_{\rm loc}(X)$ and $\ppi$ a test plan such that $\supp((\e_t)_\sharp\ppi)\subset\Omega$ for some  bounded open set $\Omega$ and all $t$'s sufficiently small. Using the fact that $g\in\s^2(\Omega)$, for sufficiently small $t$'s we can integrate  inequality \eqref{eq:localplan} and use Young's inequality to get
\begin{equation}
\label{eq:perrepr}
\begin{split}
\int g(\gamma_t)-g(\gamma_0)\,\d\ppi(\gamma)&\leq \frac12\iint_0^t\weakgrad g^2(\gamma_s)\,\d s\,\d\ppi(\gamma)+\frac12\iint_0^t|\dot\gamma_s|^2\,\d s\,\d\ppi(\gamma)\\
&=\frac12\int \weakgrad g^2\,\d\left(\int_0^t(\e_s)_\sharp\ppi\,\d s\right)+\frac12\iint_0^t|\dot\gamma_s|^2\,\d s\,\d\ppi(\gamma).
\end{split}
\end{equation}
From the fact that $\iint_0^t|\dot\gamma_s|^2\,\d s\,\d\ppi(\gamma)<\infty$ it is immediate to verify that $(\e_t)_\sharp\ppi\to (\e_0)_\sharp\ppi$ weakly in duality with $C_b(X)$. Taking also into account that $(\e_t)_\sharp\ppi\leq C\mm$ for every $t\in[0,1]$ and some $C>0$, dividing \eqref{eq:perrepr} by $t$ and letting $t\downarrow0$ we deduce
\begin{equation}
\label{eq:perrepr2}
\lims_{t\downarrow0}\int \frac{g(\gamma_t)-g(\gamma_0)}t\,\d\ppi(\gamma)\leq \frac12\int\weakgrad g^2\,\d(\e_0)_\sharp\ppi(\gamma)+\frac12\lims_{t\downarrow0}\frac1t\iint_0^t|\dot\gamma_s|^2\,\d s\,\d\ppi(\gamma).
\end{equation}
In a smooth Riemannian world, this inequality reads as 
\begin{equation}
\label{eq:liscio}
\lim_{t\downarrow0}\frac{g(\gamma_t)-g(\gamma_0)}{t}\leq \frac12|\nabla g|^2(\gamma_0)+\frac12|\gamma'_0|^2,
\end{equation}
for any smooth function $g$ and smooth curve $\gamma$ and we know that it holds $\gamma'_0=\nabla g(\gamma_0)$ if and only if the equality in \eqref{eq:liscio} holds. We are therefore lead to the following definition:
\begin{definition}[Plan representing gradients]\label{def:planrepr} Let $(X,\sfd,\mm)$ be an infinitesimally Hilbertian space,  $g\in\s^2_{\rm loc}(X)$ and $\ppi\in\prob{C([0,1],X)}$. We say that $\ppi$ represents the gradient of $g$ provided:
\begin{itemize}
\item[i)] there is $T>0$ such that $(\e_t)_\sharp\ppi\leq C\mm$ and $\supp((\e_t)_\sharp\ppi)\subset \Omega$ for some constant $C>0$ and bounded open set $\Omega$,
\item[ii)] $\iint_0^T|\dot\gamma_t|^2\,\d t\,\d\ppi(\gamma)<\infty$,
\item[iii)] the inequality
\begin{equation}
\label{eq:defppigrad}
\limi_{t\downarrow0}\int \frac{g(\gamma_t)-g(\gamma_0)}t\,\d\ppi(\gamma)\geq \frac12\int\weakgrad g^2\,\d(\e_0)_\sharp\ppi(\gamma)+\frac12\lims_{t\downarrow0}\frac1t\iint_0^t|\dot\gamma_s|^2\,\d s\,\d\ppi(\gamma),
\end{equation}
holds.
\end{itemize}
\end{definition}
Notice that plans representing gradients exist in high generality (see \cite{Gigli12}). The following simple and crucial result shows the link between differentiation of a Sobolev function $f$ along a plan representing $\nabla g$ and the object $\la\nabla f,\nabla g\ra$. 
\begin{theorem}[Horizontal and vertical derivatives]\label{thm:horver}
Let $(X,\sfd,\mm)$ be an infinitesimally Hilbertian metric measure space, $f,g\in\s^2_{\rm loc}(X)$ and $\ppi\in\prob{C([0,1],X)}$ be representing the gradient of $g$. Then
\[
\lim_{t\downarrow0}\int \frac{f(\gamma_t)-f(\gamma_0)}{t}\,\d\ppi(\gamma)=\int\la\nabla f,\nabla g\ra\d(\e_0)_\sharp\ppi.
\]
\end{theorem}
\begin{proof}
Write inequality \eqref{eq:perrepr2} for the function $g+\eps  f$ and subtract inequality \eqref{eq:defppigrad} to get
\[
\lims_{t\downarrow0}\eps\int\frac{f(\gamma_t)-f(\gamma_0)}{t}\,\d\ppi(\gamma)\leq \int \weakgrad{(g+\eps f)}^2-\weakgrad g^2\,\d(\e_0)_\sharp\ppi.
\] 
Divide by $\eps >0$ (resp. $\eps<0$), let $\eps\downarrow0$ (resp. $\eps\uparrow0$) and recall \eqref{eq:limite} to conclude.
\end{proof}

\subsection{Measure valued Laplacian}
Having understood the definition of $\la\nabla f,\nabla g\ra$, we can now integrate by parts and give the definition of measure valued Laplacian.

For $\Omega\subseteq X$ open, we will denote by  $\test\Omega$ the set of all Lipschitz functions compactly supported in $\Omega$.
\begin{definition}[Measure valued Laplacian]\label{def:distrlap} Let $(X,\sfd,\mm)$ be an infinitesimally Hilbertian space  and $\Omega\subseteq X$ open.
Let $g:\Omega\to\R$ be a locally Lipschitz function. We say that $g$ has a distributional Laplacian in $\Omega$, and write $g\in D(\bd,\Omega)$, provided there exists a Radon  measure $\mu$ on $\Omega$ such that
\begin{equation}
\label{eq:deflap}
-\int\la\nabla f,\nabla g\ra\d\mm=\int f\,\d\mu,\qquad\forall f\in\test\Omega.
\end{equation}
In this case we will say that $\mu$ (which is clearly unique) is the distributional Laplacian of $g$ and  indicate it by $\bd g\restr\Omega$. In the case $\Omega=X$ we write $D(\bd)$ and $\bd g$ in place of $D(\bd,\Omega)$ and $\bd g\restr\Omega$.
\end{definition} 
Notice that the integrand in the left hand side of \eqref{eq:deflap} is in $L^1(\Omega)$, because $g$, being locally Lipschitz, is Lipschitz on $\supp(f)$ and thus  inequalities \eqref{eq:boundfg} and \eqref{eq:lipweak} grant that the integrand is bounded. In this direction, the restriction to locally Lipschitz $g$'s is quite unnatural and indeed unnecessary (see \cite{Gigli12}), yet it is sufficient for our purposes so that we will be satisfied with it. 

The calculus rules for $\bd g$ are easily derived from those of $\la\nabla f,\nabla g\ra$ from basic algebraic manipulation. Start observing that  since the left hand side of \eqref{eq:deflap} is linear in $g$, the set $D(\bd,\Omega)$ is a vector space and the map
\[
D(\bd,\Omega)\ni g\qquad\mapsto\qquad\bd g\restr\Omega\,,
\]
is linear. We also have natural chain and Leibniz rules:
\begin{proposition}[Chain rule]\label{prop:chainlap}
Let $(X,\sfd,\mm)$ be an infinitesimally Hilbertian space,  $\Omega\subseteq X$  an open set and $g\in D(\bd,\Omega)$. Then for every  function $\varphi\in C^{2}(g(\Omega))$, the function $\varphi\circ g$ is in $D(\bd,\Omega)$ and it holds  
\begin{equation}
\label{eq:chainlap}
\bd(\varphi\circ g)\restr\Omega=\varphi'\circ g\bd g\restr\Omega+\varphi''\circ g\weakgrad g^2\mm\restr\Omega.
\end{equation}
\end{proposition}
\begin{proof} The right hand side of \eqref{eq:chainlap} defines a locally finite measure, so the statement makes sense. Now let $f\in\test\Omega$ and notice that being $\varphi'\circ g$ locally Lipschitz, we also have  $f\varphi'\circ g\in\test\Omega$. The conclusion comes from the calculus rules expressed in Theorem \ref{thm:calculus} noticing that:
\[
\begin{split}
-\int \la\nabla f,\nabla(\varphi\circ g)\ra\d\mm&=-\int \varphi'\circ g\la\nabla f,\nabla g\ra\d\mm\\
&=-\int\la\nabla(f\varphi'\circ g),\nabla g\ra-f\la\nabla(\varphi'\circ g),\nabla g\ra\d\mm\\
&=\int f\varphi'\circ g\,\d\bd g\restr\Omega+\int f\varphi''\circ g\weakgrad g^2\,\d\mm,
\end{split}
\]
which is the thesis.
\end{proof}
\begin{proposition}[Leibniz rule]
Let $(X,\sfd,\mm)$ be an infinitesimally Hilbertian space, $\Omega\subseteq X$  an open set  and $g_1,g_2\in D(\bd,\Omega)$. Then $g_1g_2\in D(\bd,\Omega)$ and
\begin{equation}
\label{eq:leiblap}
\bd(g_1g_2)\restr\Omega=g_1\bd g_2\restr\Omega+g_2\bd g_1\restr\Omega+2\la\nabla g_1,\nabla g_2\ra\mm\restr\Omega.
\end{equation}
\end{proposition}
\begin{proof}
The right hand side of \eqref{eq:leiblap} defines a locally finite measure, so the statement makes sense. For $f\in\test\Omega$ we have $fg_1,fg_2\in\test\Omega$, hence using the Leibniz rule \eqref{eq:leibf} and the symmetry \eqref{eq:simm}  we get
\[
\begin{split}
-\int\la\nabla f,\nabla(g_1g_2)\ra\d\mm&=-\int g_1\la\nabla f,\nabla g_2\ra+g_2\la\nabla f,\nabla g_1\ra\d\mm\\
&=-\int \la\nabla(fg_1),\nabla g_2\ra+\la \nabla(fg_2),\nabla g_1\ra-2f\la\nabla g_1,\nabla g_2\ra\d\mm\\
&=\int fg_1\,\d\bd g_2\restr\Omega+\int fg_2\,\d\bd g_1\restr\Omega+\int2f\la\nabla g_1,\nabla g_2\ra\d\mm,
\end{split}
\]
which is the thesis.
\end{proof}
We conclude with the following useful comparison property:
\begin{proposition}[Comparison]\label{prop:compar}
Let $(X,\sfd,\mm)$ be an infinitesimally Hilbertian space,  $\Omega\subseteq X$  an open set, $g:\Omega\to\R$ locally Lipschitz and assume that there exists a Radon measure $\mu$ on $\Omega$ such that
\[
-\int\la\nabla f,\nabla g\ra\d\mm\leq \int f\,\d\mu,\qquad\forall f\in\test\Omega,\ f\geq 0.
\]
Then $g\in D(\bd,\Omega)$ and $\bd g\restr\Omega\leq \mu$.
\end{proposition}
\begin{proof}
The map
\[
\test\Omega\ni f\qquad\mapsto\qquad L(f):= \int f\,\d\mu+\int\la\nabla f,\nabla g\ra\d\mm,
\]
is linear and satisfies $L(f)\geq 0$ for $f\geq 0$. To conclude we need to show that there exists a non-negative Radon measure $\tilde\mu$ on $\Omega$ such that $L(f)=\int f\,\d\tilde\mu$ for any $f\in \test\Omega$. 

To this aim, fix a compact set $K\subset\Omega$ and a  function $\nchi_K\in \test\Omega$ such that $0\leq\nchi_K\leq 1$ everywhere and $\nchi_K= 1$ on $K$. Let $V_K\subset \test\Omega$ be the set of Lipschitz  functions with support contained in $K$ and observe that for any  $f\in V_K$, the fact that $(\max |f|)\nchi_K-f$ is in $\test\Omega$ and non-negative yields
\[
\begin{split}
0\leq L\big((\max|f|)\nchi -f\big)=(\max|f|)L(\nchi)-L(f).
\end{split}
\]
Replacing $f$ with $-f$ we deduce
\[
|L(f)|\leq (\max|f|)\, |L(\nchi)|,
\]
i.e. $L:V_K\to\R$ is continuous w.r.t. the $\sup$ norm. Hence it can be extended to a (unique, by the density of Lipschitz functions in the uniform norm) linear bounded functional on the set $C_K\subset C(X)$ of continuous functions with support contained in $K$. Since $K$ was arbitrary, by the Riesz theorem we get that there exists a Radon  measure $\tilde\mu$ such that $L(f)=\int f\,\d\tilde\mu$ for any $f\in \test\Omega$.
It is obvious that  $\tilde\mu$ is non-negative, thus the thesis is achieved.
\end{proof}
\section{Analytic effects of the geometric assumptions}
\subsection{\underline{Things to know:}\ optimal transport and $\RCD(0,N)$ condition}\label{se:cd}
Let $(X,\sfd)$ be a proper geodesic space. By $\probt X$ we denote the space of Borel probability measures on $X$ with finite second moment and by $W_2$ the quadratic transportation distance defined on it. In this setting $W_2(\mu,\nu)$ can be defined as
\begin{equation}
\label{eq:defw2}
W_2^2(\mu,\nu):=\inf\iint_0^1|\dot\gamma_t|^2\,\d t\,\d\ppi(\gamma),
\end{equation}
the $\inf$ being taken among all plans $\ppi\in\prob{C([0,1],X)}$ such that $(\e_0)_\sharp\ppi=\mu$, $(\e_1)_\sharp\ppi=\nu$. It turns out that a minimum always exists and is concentrated on the set $\geo(X)\subset C([0,1],X)$ of constant speed minimizing geodesics on $X$, i.e. curves $\gamma$ such that $\sfd(\gamma_t,\gamma_s)=|t-s|\sfd(\gamma_0,\gamma_1)$ for every $t,s\in[0,1]$.  In the following, when speaking about geodesics  we will always refer to constant speed minimizing geodesics.

The set of minimizers for \eqref{eq:defw2} is denoted by $\gopt(\mu,\nu)$. For every $\ppi\in\gopt(\mu,\nu)$ the map $t\mapsto(\e_t)_\sharp\ppi$ is a $W_2$-geodesic connecting $\mu$ to $\nu$ and viceversa for any $(\mu_t)\subset \probt X$ geodesic with $\mu_0=\mu$ and $\mu_1=\nu$ there is $\ppi\in\gopt(\mu,\nu)$ (not necessarily unique) such that $\mu_t=(\e_t)_\sharp\ppi$ for every $t\in[0,1]$. Any such $\ppi$ is said to be a lifting of $(\mu_t)$, or to induce $(\mu_t)$.

A function $\varphi:X\to\R\cup\{-\infty\}$ not identically $-\infty$ is said $c$-concave provided there is $\psi:X\to\R\cup\{-\infty\}$ such that
\[
\varphi(x)=\inf_{y\in X}\frac{\sfd^2(x,y)}{2}-\psi(y).
\]
Given a $c$-concave function $\varphi$, its $c$-transform $\varphi^c:X\to\R\cup\{-\infty\}$ is defined by
\[
\varphi^c(y):=\inf_{x\in X}\frac{\sfd^2(x,y)}{2}-\varphi(x).
\] 
It turns out that $\varphi$ is $c$-concave if and only if  $\varphi^{cc}=\varphi$. The $c$-superdifferential $\partial^c\varphi$ of a $c$-concave function $\varphi$ is the subset of $X^2$ of those couples $(x,y)$ such that $\varphi(x)+\varphi^c(y)=\frac{\sfd^2(x,y)}2$, or equivalently the set of $(x,y)$'s such that
\[
\varphi(z)-\varphi(x)\leq \frac{\sfd^2(z,y)}{2}-\frac{\sfd^2(x,y)}{2},\qquad\forall z\in X.
\]
For $x\in X$, the set $\partial^c\varphi(x)\subset X$ is the set of those $y$'s such that $(x,y)\in\partial^c\varphi(y)$.

It can be proved that a  $\ppi\in\geo(X)$ belongs to $\gopt((\e_0)_\sharp\ppi,(\e_1)_\sharp\ppi)$ if and only if there is a $c$-concave function $\varphi$ such that $\supp((\e_0,\e_1)_\sharp\ppi)\subset\partial^c\varphi$. Any such $\varphi$ is called Kantorovich potential from $(\e_0)_\sharp\ppi$ to $(\e_1)_\sharp\ppi$ and is said to induce $\ppi$. It is then easy to check that for any Kantorovich potential $\varphi$ from $\mu$ to $\nu$, every  $\ppi\in\gopt(\mu,\nu)$ and every $t\in[0,1]$,  the function $t\varphi$ is a Kantorovich potential from $\mu$ to $(\e_t)_\sharp\ppi$, i.e. $t\varphi$ is $c$-concave and it holds
\begin{equation}
\label{eq:tkant}
\gamma\in\geo(X),\quad\gamma_1\in\partial^c\varphi(\gamma_0)\qquad\Rightarrow\qquad\gamma_t\in\partial^c(t\varphi)(\gamma_0),\qquad\forall t\in[0,1].
\end{equation}

Notice that Kantorovich potentials can be chosen to satisfy the following property, slightly stronger than $c$-concavity:
\[
\varphi(x)=\inf_{y\in\supp(\nu)}\frac{\sfd^2(x,y)}{2}-\varphi^c(y),
\]
which shows in particular that if $\supp(\nu)$ is bounded, then $\varphi$ can be chosen to be locally Lipschitz.

\bigskip

Let $\mm$ be a non-negative Radon measure on our proper geodesic metric space $(X,\sfd)$. For $N\in[1,\infty)$ we define the functional $\u_N:\probt X\to[-\infty,0]$ as
\[
\u_N(\mu):=-\int\rho^{1-\frac1N}\,\d\mm,\qquad\mu=\rho\mm+\mu^s,\ \mu^s\perp\mm,
\]
if $N>1$ and
\[
\u_1(\mu):=\mm(\{\rho>0\}),\qquad\mu=\rho\mm+\mu^s,\ \mu^s\perp\mm.
\]
Notice that if $\mu$ is concentrated on a bounded set, then $\u_N(\mu)>-\infty$ and for every $B\subset X $ Borel and bounded the restriction of $\u_N$ to the measures concentrated on $B$ is lower semicontinuous w.r.t. weak convergence.

In the limiting case $N=\infty$ we consider the relative entropy functional $\u_\infty$ defined on the space of measures with bounded support given by
\[
\u_\infty(\mu):=\left\{\begin{array}{ll}
\displaystyle{\int\rho\log\rho\,\d\mm},&\qquad\text{ if }\mu=\rho\mm,\\
+\infty,&\qquad\text{ if $\mu$ is not absolutely continuous w.r.t. $\mm$}.
\end{array}
\right.
\]

\begin{definition}[$\CD(0,N)$ and $\RCD(0,N)$ conditions]\label{def:cd}
Let $N\in[1,\infty]$. A proper geodesic metric measure space $(X,\sfd,\mm)$ is said a $\CD(0,N)$ space provided for any couple of measures $\mu_0,\mu_1\in\probt X$ with bounded support there exists a geodesic $(\mu_t)\subset\probt X$ connecting them such that
\begin{equation}
\label{eq:cd}
\u_{N'}(\mu_t)\leq(1-t)\u_{N'}(\mu_0)+t\u_{N'}(\mu_1),\qquad\forall t\in[0,1],
\end{equation}
for every $N'\in [N,\infty]$.

A $\CD(0,N)$ space which is also infinitesimally Hilbertian will be called $\RCD(0,N)$ space.
\end{definition}
Notice that the definition given in this way (i.e. with the measures $\mu_0,\mu_1$ with bounded support instead of bounded and contained in $\supp(\mm)$ as in \cite{Sturm06II}) forces the support of $\mm$ to be the whole $X$. This is a bit dangerous only when discussing stability issues in the infinite dimensional case, but in fact irrelevant for our discussion.

The restriction to proper geodesic spaces when dealing with the $\CD(0,\infty)$ condition  is not natural (see e.g. \cite{Sturm06I}, \cite{Villani09}, \cite{AmbrosioGigliSavare11}) but for our purposes it does not really matter, given that our space is $\CD(0,N)$. In this direction, notice that choosing $\mu_0=\delta_{x_0}$ and $\mu_1=\mm(B_R(x_0))^{-1}\mm\restr{B_R(x_0)}$, a direct application of inequality \eqref{eq:cd} and of Jensen's inequality yields the sharp Bishop-Gromov volume comparison estimate (\cite{Lott-Villani09}, \cite{Sturm06II}), valid on general $\CD(0,N)$ spaces:
\begin{equation}
\label{eq:BG}
\frac{\mm(B_r(x_0))}{\mm(B_R(x_0))}\geq\frac{r^N}{R^N},\qquad\forall x_0\in X,\ 0\leq r\leq R,
\end{equation}
which in particular yields that $\mm$ is doubling and henceforth gives an estimate on the total boundedness of bounded sets (see e.g. the last part of the proof of Theorem \ref{thm:dimm1}), so that we have a precise quantification of `how compact' bounded sets are.

\bigskip

An important and non-trivial fact about $\RCD(0,\infty)$ spaces is the following generalization of the celebrated Brenier-McCann theorem proved in \cite{RajalaSturm12} and \cite{Gigli12a}:
\begin{theorem}[Optimal maps in $\RCD(0,\infty)$ spaces]\label{thm:maprcd}
Let $(X,\sfd,\mm)$ be a $\RCD(0,\infty)$ space and $\mu,\nu\in\probt X$ two measures absolutely continuous w.r.t. $\mm$. Then there exists a unique $\ppi\in\gopt(\mu,\nu)$.  Moreover, such plan is induced by a map and concentrated on a set of non-branching geodesics, i.e. for every $t\in[0,1]$ there exists a Borel map ${\rm inv_\sppi\text{-}e}_t:X\to\geo(X)$ such that
\begin{equation}
\label{eq:invet}
\ppi=({\rm inv_\sppi\text{-}e}_t)_\sharp(\e_t)_\sharp\ppi.
\end{equation}
\end{theorem}

\subsection{Improved geodesic regularity in the case $N<\infty$}
\begin{quote}
{From now on the space $(X,\sfd,\mm)$ will always be assumed to be a $\RCD(0,N)$ space.}
\end{quote}
Here we show how finite dimensionality can improve the result of Theorem \ref{thm:maprcd} by weakening the assumptions in `just one of $\mu,\nu$ is absolutely continuous', rather then asking for both of them to be so. The discussion is taken from \cite{GigliRajalaSturm13}.
\begin{proposition}\label{prop:varphi}
Let   $\mu_i=\rho_i\mm\in\probt X$, $i=0,1$,  two given measures,  $\ppi\in\gopt(\mu_0,\mu_1)$ the unique optimal geodesic plan from $\mu_0$ to $\mu_1$ given by Theorem \ref{thm:maprcd}  and put $\mu_t:=(\e_t)_\sharp\ppi$. Then $\mu_t\ll\mm$ for every $t\in[0,1]$ and writing $\mu_t=\rho_t\mm$  we have
\begin{equation}
\label{eq:boundpoint}
\rho_t(\gamma_t)^{-\frac1N}\geq (1-t)\rho_0(\gamma_0)^{-\frac1N}+t\rho_1(\gamma_1)^{-\frac1N},\qquad\ppi\ae \ \gamma.
\end{equation}
\end{proposition}
\begin{proof} We start proving that $\mu_t\ll\mm$ for every $t\in[0,1]$.  Fix $\bar x\in X$ and for $M>0$ let $G_M\subset \geo(X)$ be defined by 
\[
G_M:=\Big\{\gamma\in \geo(X)\ :\ \rho_0(\gamma_0),\rho_1(\gamma_1),\sfd(\gamma_0,\bar x),\sfd(\gamma_1,\bar x)\leq M\Big\}.
\]
For $M$ large enough we have $\ppi(G_M)>0$, thus the plan $\ppi_M:=c_M\ppi\restr{G_M}$ is well defined,  $c_M:=\ppi(G_M)^{-1}$ being the normalization constant. Put $\mu_0^M:=(\e_0)_\sharp\ppi_M$, $\mu_1^M:=(\e_1)_\sharp\ppi_M$ and notice that $\mu_0^M,\mu_1^M\ll\mm$ and that by construction and since  optimality is stable by restriction we get  $\ppi_M\in\gopt( \mu_0^M, \mu_1^M)$. Hence the uniqueness part of Theorem \ref{thm:maprcd}  yields that  $\ppi_M$ is the only optimal plan from $\mu_0^M$ to $\mu_1^M$. Being $(X,\sfd,\mm)$ a $\CD(0,N)$ space it is also a  $\CD(0,\infty)$ space and thus fact that $\u_\infty(\mu_0^M),\u_\infty(\mu_1^M)<\infty$ (because both have bounded densities and support) gives $\u_\infty((\e_t)_\sharp\ppi_M)<\infty$ for every $t\in[0,1]$. In particular, $(\e_t)_\sharp\ppi_M\ll\mm$ for every $t\in[0,1]$. Since $c_M^{-1}(\e_t)_\sharp\ppi_M\uparrow (\e_t)_\sharp\ppi=\mu_t$ as $M\to\infty$, we deduce $\mu_t\ll\mm$ for every $t\in[0,1]$.

We turn to \eqref{eq:boundpoint}. Notice that to prove it  is equivalent to prove that for any bounded Borel set $G\subset\geo(X)$ it holds
\begin{equation}
\label{eq:intcd}
\begin{split}
-\int_G\rho_{t}^{-\frac1N}(\gamma_{t})\,\d\ppi(\gamma)\leq& -\int_G(1-t)\rho_0(\gamma_0)^{-\frac1N}+t\rho_1(\gamma_1)^{-\frac1N}\,\d\ppi(\gamma).
\end{split}
\end{equation}
Fix such $G\subset\geo(X)$, assume without loss of generality that $\ppi(G)>0$ and define $\ppi_G:=\ppi(G)^{-1}\ppi\restr G$.  Notice that since $G$ is bounded, $(\e_t)_\sharp\ppi_G$ has bounded support for every $t\in[0,1]$. Let ${\rm inv_\sppi\text{-}e}_t:X\to\geo(X)$ be the maps given by Theorem \ref{thm:maprcd} and notice that the identity \eqref{eq:invet} ensures  that $(\e_t)_\sharp\ppi_G=\ppi(G)^{-1}\nchi_G\circ {\rm inv_\sppi\text{-}e}_t\,(\e_t)_\sharp\ppi$.  In other words,  letting $\rho_{G,t}\mm=(\e_t)_\sharp\ppi_G$ we have $\rho_{G,t}(\gamma_t)=\ppi(G)^{-1}\rho_t(\gamma_t)$ for $\ppi$-a.e. $\gamma\in G$  and therefore
\begin{equation}
\label{eq:nizza1}
-\int_G\rho_{t}^{-\frac1N}(\gamma_{t})\,\d\ppi(\gamma)=\ppi(G)^{-\frac 1N}\,\u_N((\e_t)_\sharp\ppi_G),\qquad\forall t\in[0,1].
\end{equation}
By construction, $\ppi_G$ is optimal from $\rho_{G,0}\mm$ to $\rho_{G,1}\mm$ and by the uniqueness part of Theorem \ref{thm:maprcd} we know that it is the only optimal plan, hence \eqref{eq:intcd} follows from  the $\CD(0,N)$ condition and \eqref{eq:nizza1}.
\end{proof}

\begin{lemma}\label{le:ac}
Let $\mu,\nu\in\probt X$ such that  $\mu\leq C\mm$ for some $C>0$. Then there exists a geodesic $(\mu_t)$ from $\mu$ to $\nu$ such that $\mu_t\leq\frac{C}{(1-t)^N}\mm$ for every $t\in[0,1)$.
\end{lemma}
\begin{proof}
Let $(\nu_n)\subset\probt X$ be a sequence of absolutely continuous measures weakly converging to $\nu$ and with uniformly bounded supports and $\ppi_n\in\gopt(\mu,\nu_n)$ the unique optimal plan given by Theorem \ref{thm:maprcd}. By Proposition \ref{prop:varphi} we know that $(\e_t)_\sharp\ppi_n\ll\mm$ and that denoting by $\rho_{n,t}$ its density we have
\[
\rho_{n,t}(\gamma_t)\leq \rho_{n,0}(\gamma_0)(1-t)^{-N},\qquad\ppi_n\ae\ \gamma,
\]
having dropped the term involving $\rho_{n,1}$ in the bound \eqref{eq:boundpoint}. By the assumption $\mu\leq C\mm$ we thus deduce
\[
(\e_t)_\sharp\ppi_n\leq \frac{C}{(1-t)^N}\mm,\qquad\forall t\in[0,1),\ n\in\N.
\]
This bound is independent on $n\in\N$, hence with a simple compactness argument based on the fact that $(X,\sfd)$ is proper we get the conclusion by letting $n\to\infty$.
\end{proof}

\begin{theorem}[Exponentiation and optimal maps]\label{thm:exp} Let $\varphi:X\to\R$ a locally Lipschitz $c$-concave function. Then for $\mm$-a.e. $x\in X$ there exists a unique geodesic $T(x)\in\geo(X)$ with $T(x)_0=x$ and $T(x)_1\in\partial^c\varphi(x)$. For any $t\in[0,1)$ the map $T_t:X\to X$ sending $x$ to $T(x)_t$ is Borel and satisfies
\begin{equation}
\label{eq:acexp}
(T_t)_\sharp\mm\ll\mm.
\end{equation}
In particular, for every $\mu,\nu\in\probt X$ with $\mu\ll\mm$, there exists a unique geodesic $(\mu_t)$ connecting them, a unique lifting  $\ppi\in\gopt(\mu,\nu)$ of it and this plan is induced by a map and concentrated on a set of non-branching geodesics.
\end{theorem}
\begin{proof} We start with existence. Let $x\in X$ and $(y_n)\subset X$ a sequence such that $\varphi(x)=\lim_{n\to\infty}\frac{\sfd^2(x,y_n)}2-\varphi^c(y_n)$. Assume by contradiction that $\lims_{n\to\infty}\sfd(x,y_n)=\infty$, let $\gamma^n:[0,\sfd(x,y_n)]\to X$ be a unit speed geodesic connecting $x$ to $y_n$ and put $z_n:=\gamma^n_1$. Then the sequence $(z_n)\subset X$ is bounded and the inequality
\[
\varphi(z_n)-\varphi(x)\leq \frac{\sfd^2(z_n,y_n)}{2}-\frac{\sfd^2(x,y_n)}2=-2\sfd(x,y_n)+1,
\]
shows that $\limi_{n\to\infty}\varphi(z_n)=-\infty$, contradicting the fact that $\varphi$ is locally Lipschitz. Hence $(y_n)\subset X$ must be bounded and a simple compactness-continuity argument shows that any limit point $y$ belongs to $\partial^c\varphi(x)$. Since $(X,\sfd)$ is geodesic and $x\in X$ was arbitrary, this is sufficient to get existence of geodesics as in the statement.

For uniqueness we argue by contradiction as well. For $x\in\Omega$ let  $G(x)\subset\geo(X)$ be the set of $\gamma$'s such that $\gamma_0=x$ and $\gamma_1\in\partial^c\varphi(x)$ and assume that there is a compact set $E_1\subset \Omega$ such that $\mm(E_1)>0$ and $\#G(x)\geq 2$ for every $x\in E_1$. 

For some $a>0$ there is a compact set $E_2\subset E_1$  with $\mm(E_2)>0$ such that ${\rm diam}G(x)\geq a$ for every $x\in E_2$. Pick such $a$ and $E_2$. For $t\in[0,1]$ put  $G_t(x):=\{\gamma_t:\gamma\in G(x)\}\subset X$ and consider the set $\mathcal B\subset E_2\times [0,1]$ of $(x,t)$'s such that   such that ${\rm diam}G_t(x)\geq\frac a2$. It is easy to check that $\mathcal B$ is compact and the continuity of geodesics grants that for any $x\in E_2$ the set of $t$'s such that $(x,t)\in \mathcal B$  has positive $\mathcal L^1$-measure. By Fubini's theorem, there is $t_0\in[0,1]$ such that the compact set $E_3\subset E_2$ of $x$'s such that ${\rm diam}G_{t_0}(x)\geq\frac a2$ has positive $\mm$-measure. Notice that necessarily $t_0>0$. With a Borel selection argument we can find a Borel map $T:E_3\to X$ such that $T(x)\in G_{t_0}(x)$ for every $x\in E_3$. Let $x_0\in X$ be such that  $T_\sharp(\mm\restr {E_3})(B_{\frac a6}(x_0))>0$ and put  $E_4:=T^{-1}(B_{\frac a6}(x_0))$, so that $\mm(E_4)>0$. By construction, the map $E_4\ni x\mapsto G_{t_0}(x)\setminus B_{\frac{a}3}(x_0)$ is Borel and has non-empty values, thus again with a Borel selection argument we can find Borel map $S:E_4\to X$ such that $S(x)\in G_{t_0}(x)\setminus B_{\frac{a}3}(x_0) $ for every $x\in A$.

Let $\mu:=\mm(E_4)^{-1}\mm\restr {E_4}$, $\nu_1:=T_\sharp\mu$ and $\nu_2:=S_\sharp\mu$. By construction $\nu_1$ and $\nu_2$ have disjoint support, and in particular $\nu_1\neq \nu_2$. Furthermore, recalling property \eqref{eq:tkant}, the function $t_0\varphi$ is a Kantorovich potential both from $\mu$ to $\nu_1$ and from $\mu$ to $\nu_2$. Apply Lemma \ref{le:ac} to both $(\mu,\nu_1)$ and $(\mu,\nu_2)$ to find geodesics $(\mu^i_t)$, $i=1,2$, from $\mu$ to $\nu_1,\nu_2$ respectively such that $\mu^i_t\ll\mm$ for every $t\in[0,1)$, $i=1,2$. By construction, for $t$ sufficiently close to 1 we have $\mu^1_t\neq \mu^2_t$. Fix such $t$, let $\ppi^i\in\gopt(\mu,\mu^i_t)$, $i=1,2$ and notice that $\ppi^1\neq \ppi^2$ and that $\supp((\e_0,\e_1)_\sharp\ppi^i)\subset\partial^c(tt_0\varphi)$, $i=1,2$. 

Thus for the plan $\ppi:=\frac12(\ppi^1+\ppi^2)$ it also holds $\supp((\e_0,\e_1)_\sharp\ppi)\subset \partial^c(tt_0\varphi)$, so that $\ppi$ is optimal as well. Moreover it satisfies $(\e_0)_\sharp\ppi,(\e_1)_\sharp\ppi\ll\mm$ and, by construction, is not induced by a map. This contradicts Theorem \ref{thm:maprcd}, concluding the proof of the first part of the statement.

For the last part, notice that if the optimal geodesic plan is not unique or not induced by a map, there must be $\ppi\in\gopt(\mu,\nu)$ which is not induced by a map. With a restriction argument we can then assume that $\mu:=(\e_0)_\sharp\ppi$, and $\nu:=(\e_1)_\sharp\ppi$ have bounded support, with $\mu\ll\mm$. But in this case there is a locally Lipschitz Kantorovich potential from $\mu$ to $\nu$ and the first part of the statement gives the conclusion. This argument shows not only uniqueness of  $\ppi$, but also that of the geodesic $(\mu_t)$.

Finally, the property \eqref{eq:acexp} is now a simple consequence of the uniqueness we just proved and Lemma \ref{le:ac}
\end{proof}
We conclude with the following result which puts in relation optimal plans and Sobolev calculus. Notice that it is in fact a restatement of the metric Brenier theorem proved in \cite{AmbrosioGigliSavare11}.
\begin{corollary}\label{cor:metrbre}
Let $\mu,\nu\in\probt X$ with bounded support, assume that $\mu\leq C\mm$ for some $C>0$, let $\ppi\in\gopt(\mu,\nu)$ be the optimal geodesic plan given by Theorem \ref{thm:exp} and let $\varphi$ be a locally Lipschitz Kantorovich potential from $\mu$ to $\nu$. 

Then   $\ppi$ represents the gradient of $-\varphi$  in the sense of Definition \ref{def:planrepr}. 
\end{corollary}
\begin{proof}
It is trivial that  $\cup_{t\in[0,1]}\supp(\mu_t)$ is bounded, so the existence of $\Omega$ as in $(i)$ of Definition \ref{def:planrepr}  follows. Lemma \ref{le:ac} and the uniqueness granted by Theorem \ref{thm:exp} ensure that $(\e_t)_\sharp\ppi\leq \frac{C}{(1-t)^N}\mm$ for every $t\in[0,1)$ and so property $(i)$ in Definition \ref{def:planrepr} holds. Given that $\iint_0^1|\dot\gamma_t|^2\,\d t\,\d\ppi(\gamma)=W_2^2(\mu,\nu)<\infty$, property $(ii)$ holds as well, so we need only to check $(iii)$. By construction, we have $\gamma_1\in\partial^c\varphi(\gamma_0)$ for $\ppi$-a.e. $\gamma$, therefore for $\ppi$-a.e. $\gamma$ and every $z\in X$ we have
\[
\varphi(z)-\varphi(\gamma_0)\leq\frac{\sfd^2(z,\gamma_1)}{2}-\frac{\sfd^2(\gamma_0,\gamma_1)}{2}\leq \frac{\sfd(\gamma_0,z)}{2}(\sfd(z,\gamma_1)+\sfd(\gamma_0,\gamma_1)).
\]
Dividing by $\sfd(\gamma_0,z)$ and letting $z\to\gamma_0$ we deduce $\lip^+(\varphi)(\gamma_0)\leq \sfd(\gamma_0,\gamma_1)$, while choosing $z=\gamma_t$ after little manipulation we get
\[
\limi_{t\downarrow0}\int\!\!\frac{\varphi(\gamma_0)-\varphi(\gamma_t)}{t}\,\d\ppi(\gamma)\!\geq\!\!\int\!\!\sfd^2(\gamma_0,\gamma_1)\,\d\ppi(\gamma)\geq \frac12\!\int\!\!\big(\lip^+(\varphi)\big)^2\,\d(\e_0)_\sharp\ppi+\frac12\!\int\!\!\sfd^2(\gamma_0,\gamma_1)\,\d\ppi(\gamma).
\]
Since $\ppi$ is concentrated on $\geo(X)$ we have $\int\sfd^2(\gamma_0,\gamma_1)\,\d\ppi(\gamma)=\iint_0^1|\dot\gamma_t|^2\,\d t\,\d\ppi(\gamma)$, hence  recalling the bound \eqref{eq:lipweak} we conclude.
\end{proof}

\subsection{Laplacian comparison estimates}
In this section we prove the sharp Laplacian comparison estimate for the distance on $\RCD(0,N)$ spaces.

The idea of the proof, which relies only on the curvature-dimension condition and not, as in the smooth case, on Jacobi fields calculus or on the Bochner inequality, is the following. Fix a $c$-concave function $\varphi$, a measure $\mu=\rho\mm$ and consider the geodesic $t\mapsto\mu_t:=(T_t)_\sharp\mm$, $T_t$ being given by Theorem \ref{thm:exp}. Then combine the inequality 
\[
\frac{\u_N(\mu_t)-\u_N(\mu_0)}{t}\leq\u_N(\mu_1)-\u_N(\mu_0),
\]
which follows directly from \eqref{eq:cd}, with the bound
\[
\limi_{t\downarrow0}\frac{\u_N(\mu_t)-\u_N(\mu_0)}{t}\geq -\frac1N\int\la\nabla \rho^{1-\frac1N},\nabla\varphi\ra\d\mm,
\]
which follows from the first order differentiation formula, to obtain
\[
-\u_N(\mu_0)\geq -\frac1N\int\la\nabla \rho^{1-\frac1N},\nabla\varphi\ra\d\mm,
\]
having recalled that $\u_N(\mu_1)\leq 0$. Given that  $\u_N(\mu_0)=-\int\rho^{1-\frac1N}\,\d\mm$ and using the fact that $\rho$ was chosen independently on $\varphi$, we get the conclusion from Proposition \ref{prop:compar}.

We turn to the details.
\begin{proposition}[Lower bound on the derivative of $\u_N$]\label{prop:lowerbound}
Let $\Omega\subset X$ be a bounded open set and  $\ppi\in\prob{\geo(X)}$  an optimal geodesic plan such that:
\begin{itemize}
\item for every $t\in[0,1]$ the measure $\mu_t:=(\e_t)_\sharp\ppi$ is concentrated on $\Omega$,
\item the measure $\mu_0$ is absolutely continuous w.r.t. $\mm$ and for its density $\rho$ we have that $\rho\restr\Omega:\Omega\to\R$ is Lipschitz and bounded from below by a positive constant.
\end{itemize} 
Then we have
\begin{equation}
\label{eq:upperbound}
\limi_{t\downarrow0}\frac{\u_N(\mu_t)-\u_N(\mu_0)}t\geq-\frac1N\int_{\Omega} \la\nabla( \rho^{1-\frac1N}), \nabla \varphi\ra\d\mm,
\end{equation}
where $\varphi:X\to\R$ is any locally Lipschitz Kantorovich potential inducing $\ppi$. 
\end{proposition}
\begin{proof}
Notice that since $\rho,\rho^{-1}$ are  Lipschitz and bounded on $\Omega$, the function $ \rho^{1-\frac1N}$ is Lipschitz and bounded on $\Omega$ as well, in particular the right hand side of \eqref{eq:upperbound} is well defined and the statement makes sense. For every $\nu\in\probt{X}$ concentrated on $\Omega$ and absolutely continuous w.r.t. $\mm$, the convexity of $u_N(z)=-z^{1-\frac1N}$ gives $\u_N(\nu)-\u_N(\mu_0)\geq \int_{\Omega} u_N'(\rho)(\frac{\d\nu}{\d\mm}-\rho)\,\d\mm$. Then a simple approximation argument based on the continuity of $\rho$ gives
\[
\u_N(\nu)-\u_N(\mu_0)\geq \int_{\Omega} u_N'(\rho)\,\d\nu-\int_{\Omega} u_N'(\rho)\,\d\mu,\qquad\forall\nu\in\probt X\ \text{concentrated on }\Omega.
\]
Plugging $\nu:=\mu_t$, dividing by $t$ and letting $t\downarrow0$ we get
\begin{equation}
\label{eq:ober}
\limi_{t\downarrow 0}\frac{\u_N(\mu_t)-\u_N(\mu_0)}t\geq\limi_{t\downarrow0}\int\frac{u_N'(\rho)\circ\e_t-u_N'(\rho)\circ\e_0}t\,\d\ppi.
\end{equation}
Now recall that by Corollary \ref{cor:metrbre} the plan $\ppi$ represents $\nabla(-\varphi)$ and that by the assumptions on $\rho$ we have $u_N'\circ\rho\in\s^2(\Omega)$. Thus by the first order differentiation formula given in Theorem \ref{thm:horver} we can compute the right hand side of \eqref{eq:ober} and get
\[
\limi_{t\downarrow 0}\frac{\u_N((\e_t)_\sharp\ppi)-\u_N((\e_0)_\sharp\ppi)}t\geq-\int_\Omega\la\nabla( u_N'\circ\rho),\nabla \varphi\ra\rho\,\d\mm.
\]
To conclude, notice that $u_N'(z)=(-1+\frac1N)z^{-\frac1N}$ and apply twice the chain rule \eqref{eq:chainf}:
\[
\begin{split}
\Big(\frac1N-1\Big) \int_{\Omega} \la\nabla ( \rho^{-\frac1N}), \nabla \varphi\ra \rho\,\d\mm&=\Big(\frac1N-\frac1{N^2}\Big)\int_{\Omega} \rho^{-\frac1N} \la\nabla \rho, \nabla \varphi\ra\d\mm\\
&=\frac1N\int_{\Omega} \la\nabla ( \rho^{1-\frac1N}), \nabla \varphi \ra\d\mm.
\end{split}
\]
\end{proof}

\begin{lemma}\label{le:tecn}
Let $\varphi$ be a locally Lipschitz Kantorovich potential and $\Omega\subset X$ an open bounded set. Then there exists another open bounded set $\tilde\Omega$ and another locally Lipschitz Kantorovich potential $\tilde\varphi$ such that the following holds:
\begin{itemize}
\item[i)] $\tilde\varphi=\varphi$ on $\Omega$,
\item[ii)] for every $x\in\Omega$ and $y\in\partial^c\varphi(x)$ it holds $y\in\partial^c\tilde\varphi(x)$,
\item[iii)] for every $x\in\tilde\Omega$ the set $\partial^c\tilde\varphi(x)$ is non-empty and for every geodesic $\gamma$ such that $\gamma_0=x$ and $\gamma_1\in\partial^c\tilde\varphi(x)$ it holds $\gamma_t\in\tilde\Omega$ for every $t\in[0,1]$. 
\end{itemize}
\end{lemma}
\begin{proof}
Arguing as in the beginning of the proof of Theorem \ref{thm:exp} we see that  the set 
\[
B:=\{y\in X\ :\ y\in\partial^c\varphi(x)\ \text{ for some $x\in\Omega$}\},
\]
is bounded. Define $\tilde\varphi:X\to \R$ by
\[
\tilde\varphi(x):=\inf_{y\in B}\frac{\sfd^2(x,y)}{2}-\varphi^c(y),
\]
and notice that by construction $\tilde\varphi$ is a Kantorovich potential satisfying $(i)$ and $(ii)$ of the statement. Let $s:=\sup_\Omega\varphi$ and define 
\[
\tilde\Omega:=\{\tilde\varphi<s+1\}.
\]
Obviously $\tilde\Omega$ is open, bounded and contains $\Omega$. Now let $x\in\tilde\Omega$ and $y\in\partial^c\tilde\varphi(x)$. The inequality
\[
\tilde\varphi(z)-\tilde\varphi(x)\leq \frac{\sfd^2(z,y)}{2}-\frac{\sfd^2(x,y)}2,
\]
shows that if $\sfd(z,y)\leq \sfd(x,y)$, then $\tilde\varphi(z)\leq\tilde\varphi(x)$ and thus  $z\in\tilde\Omega$. This applies in particular to the choice $z=\gamma_t$, where $\gamma\in\geo(X)$ is a geodesic from $x$ to $y$, hence $(iii)$ is fulfilled as well.
\end{proof}
\begin{proposition}[Key inequality]\label{thm:enin}
Let $\mu,\nu\in\probt X$ be two measures with bounded support, $\varphi$ a locally Lipschitz Kantorovich potential from $\mu$ to $\nu$ and assume that $\mu\ll\mm$ with density $\rho$ such that $\rho^{1-\frac1N}$ is Lipschitz. 

Then
\[
\u_N(\nu)-\u_N(\mu)\geq -\frac1N\int \la\nabla(\rho^{1-\frac1N}),\nabla\varphi\ra\,\d\mm.
\]
\end{proposition}
\begin{proof} Let $\Omega\subset X$ be an open bounded set containing $\supp(\mu)$ and use Lemma \ref{le:tecn} above to find $\tilde\varphi$ and $\tilde \Omega$ fulfilling $(i),(ii),(iii)$ of the statement. For $\eps>0$ define $\rho_\eps:X\to\R^+$ as $0$ on $X\setminus\tilde\Omega$ and as $c_\eps(\eps+\rho^{1-\frac1N})^{\frac N{N-1}}$ on $\tilde\Omega$, $c_\eps\uparrow1$ being chosen so that $\rho_\eps$ is a probability density. Let $T_t:X\to X$, $t\in[0,1]$, be the optimal maps induced by $\tilde\varphi$ as in Theorem \ref{thm:exp} and put $\mu_\eps:=\rho_\eps\mm$ and $\mu_{t,\eps}:=(T_t)_\sharp\mu_\eps$. Notice that by $(iii)$ of Lemma \ref{le:tecn} we know that  $\mu_{t,\eps}$ is concentrated on $\tilde\Omega$ for every $\eps>0$, $t\in[0,1]$ and that by $(i),(ii)$ of Lemma \ref{le:tecn} and the uniqueness given by Theorem \ref{thm:exp} we have $(T_1)_\sharp\mu=\nu$.

By construction we know that $\mu_\eps\to\mu$ as $\eps\downarrow0$ in the total variation distance which in particular implies that $\mu_{1,\eps}\to\nu$ as $\eps\downarrow0$ in the total variation distance as well. Using the sublinearity of $u_N(z)=-z^{1-\frac1N}$ and the fact that all the considered measures are concentrated on the bounded set $\tilde\Omega$, it is then immediate to see that
\begin{equation}
\label{eq:apprun}
\u_N(\mu_{1,\eps})-\u_N(\mu_\eps)\quad\to\quad\u_N(\nu)-\u_N(\mu),\qquad as\ \eps\downarrow0.
\end{equation}
For given $\eps>0$, the assumptions of Proposition \ref{prop:lowerbound} are fulfilled with $\mu_{t,\eps}:=(T_t)_\sharp\mu_\eps$ in place of $\mu_t$ and $\tilde\Omega$ in place of $\Omega$. Thus recalling the definition of $\rho_\eps$ we have
\begin{equation}
\label{eq:treno1}
\limi_{t\downarrow0}\frac{\u_N(\mu_{t,\eps})-\u_N(\mu_\eps)}t\geq-\frac{c_\eps}N\int_{\tilde\Omega} \la\nabla( \rho^{1-\frac1N}), \nabla \tilde\varphi\ra\d\mm=-\frac{c_\eps}N\int_{\Omega} \la\nabla( \rho^{1-\frac1N}), \nabla \varphi\ra\d\mm,
\end{equation}
where in the equality we used the fact that $\rho$ is concentrated on $\Omega$, the locality of the object $\la\nabla f,\nabla g\ra$ and the fact that $\tilde\varphi=\varphi$ on $\Omega$.

Now observe that the curve $t\mapsto\mu_{t,\eps}$ is a geodesic from $\mu_\eps$ to $\nu_\eps$ and that by Theorem \ref{thm:exp} it is the only one. Hence the $\CD(0,N)$ condition \eqref{eq:cd} yields
\[
\u_N(\mu_{t,\eps})\leq(1-t)\u_N(\mu_\eps)+t\u_N(\nu_\eps),\qquad\forall t\in[0,1],
\]
and thus
\[
\frac{\u_N(\mu_{t,\eps})-\u_N(\mu_\eps)}t\leq \u_N(\nu_\eps)-\u_N(\mu_\eps),\qquad\forall t\in(0,1].
\]
This bound, \eqref{eq:treno1} and \eqref{eq:apprun} yield the thesis.
\end{proof}

\begin{theorem}[Laplacian comparison]\label{cor:laplcomp}
Let $\varphi:X\to\R$ be a locally Lipschitz $c$-concave function.

Then $\varphi\in D(\bd,X)$ and
\begin{equation}
\label{eq:lapsq}
\bd\varphi\leq N\mm.
\end{equation}
\end{theorem}
\begin{proof}
By Proposition \ref{prop:compar} it is sufficient to show that 
\begin{equation}
\label{eq:percomp}
-\int_X\la\nabla f,\nabla \varphi\ra\d\mm\leq N\int_X f\,\d\mm,\qquad\forall f\in\test X,\ f\geq 0.
\end{equation}
Thus, fix a non-negative $f\in\test X$, $f$ not identically 0 and let $\Omega$ be an open bounded set containing $\supp(f)$. Define $\rho:=cf^{\frac{N}{N-1}}$, $c:=(\int f^{\frac N{N-1}})^{-1}$ being the normalization constant, let $T=T_1$ be the optimal map induced by $\varphi$ given by Theorem \ref{thm:exp} and put $\nu:=T_\sharp(\rho\mm)$. Then by Proposition \ref{thm:enin} we get
\[
\u_N(\nu)-\u_N(\rho\mm)\geq -\frac1N\int \la\nabla(\rho^{1-\frac1N}),\nabla\varphi\ra\,\d\mm.
\]
Now notice that $\u_N(\nu)\leq 0$ and recall the definition of $\rho$ to get \eqref{eq:percomp} and the conclusion.
\end{proof}

\subsection{\underline{Things to know:}\ strong maximum principle}
In order to prove that the Busemann function is harmonic, we need some form of the strong maximum principle. The following statement has been proved in \cite{Bjorn-Bjorn11}, notice that it does not require any notion of distributional Laplacian, being based on the variational formulation of super-harmonicity. The simple link between such formulation and the measure valued Laplacian has been established in \cite{Gigli12}, \cite{Gigli-Mondino12}, see the proof of Theorem \ref{thm:bharm}.
\begin{theorem}\label{thm:strongmax}
Let $(\tilde X,\tilde \sfd,\tilde \mm)$ be a metric measure space supporting a 1-2 weak local Poincar\'e inequality with $\tilde \mm$ doubling and let $g\in C(\tilde X)\cap\s^2_{\rm loc}(\tilde X)$  be with   the following property: for any non-positive $f\in\test {\tilde X}$ it holds 
\[
\int_\Omega |\nabla g|^2\,\d\tilde \mm\leq \int_\Omega|\nabla (g+f)|^2\,\d\tilde \mm,
\]
where $\Omega\subset\tilde X$ is any bounded open set containing $\supp(f)$. Assume that $g$ has a maximum. Then $g$ is constant.
\end{theorem}
We shall not discuss the meaning of 1-2 weak local Poincar\'e inequality (see for instance \cite{Bjorn-Bjorn11} and the discussion therein). For our purposes it is sufficient to know that our $\RCD(0,N)$ space $(X,\sfd,\mm)$ fulfills the assumptions of the above theorem (see \cite{Lott-Villani07}, \cite{Rajala12} and \cite{Rajala12-2}).

\subsection{The Busemann function is harmonic and $c$-concave}\label{se:basebus}
\begin{quote}
From now on the space $(X,\sfd,\mm)$ will always be assumed to be a $\RCD(0,N)$ space
and it will be assumed that there is a line $\bar\gamma:\R\to X$, i.e. a curve satisfying $$\sfd(\bar\gamma_t,\bar\gamma_s)=|t-s|,\qquad\forall t,s\in\R.$$
\end{quote}
This completes our set of assumptions on $X$ to get the splitting theorem. It is a classical and easy to prove fact that in presence of the line $\bar\gamma$ the two functions $\b^\pm:X\to\R$, called Busemann functions, are well defined by:
\[
\b^+(x):=\lim_{t\to+\infty}t-\sfd(x,\bar\gamma_t),\qquad\qquad\qquad\b^-(x):=\lim_{t\to+\infty}t-\sfd(x,\bar\gamma_{-t}).
\]
Indeed, the triangle inequality gives that the limits exist and are real valued for any $x\in X$.

In this section we first prove, following the original arguments of Cheeger-Gromoll \cite{Cheeger-Gromoll-splitting}, that it holds $\b^++\b^-\equiv 0$ and that these functions are harmonic, i.e. $\bd\b^\pm\equiv 0$. Then we show the technically useful fact that  for any $t\in\R$  the functions $t\b^\pm$ are $c$-concave. In particular, this property is what links the geometric condition of existence of a line with the theory of optimal transport on which the definition of the curvature-dimension condition is based.

We start with the following statement, which is a simple consequence of the Laplacian comparison estimates for the distance.
\begin{proposition}[Subharmonicity of the Busemann function] With the same notation as above, we have $\b^\pm\in D(\bd)$ and $\bd\b^\pm\geq 0$.
\end{proposition}
\begin{proof}
We shall prove the result for $\b^+$ only, the proof for $\b^-$ being similar. According to Proposition \ref{prop:compar} it is sufficient to show that
\[
-\int\la\nabla f,\nabla\b^+\ra\,\d\mm\geq 0,\qquad\forall f\in \test X,\ f\geq 0.
\]
Fix such $f$, let $\Omega\subset X$ be a bounded open set such that $\supp(f)\subset\Omega$ and notice that the functions $\b_t(x):=t-\sfd(x,\bar\gamma_t)$ are 1-Lipschitz and uniformly converge to $\b^+$ on $\Omega$ as $t\to\infty$. For $t$ big enough we have $\sfd(\bar\gamma_t,\Omega)>0$ and therefore applying the chain rule \eqref{eq:chainlap} to the Kantorovich potential  $g:=\frac12\sfd^2(\cdot,\bar\gamma_t)$ and the function $\psi(z):=\sqrt{2z}$ and taking into account the comparison estimate \eqref{eq:lapsq} we deduce that for $t$ big enough it holds
\[
\sfd(\cdot,\bar\gamma_t)\in D(\bd,\Omega),\qquad\bd\sfd(\cdot,\bar\gamma_t)\leq\frac{N}{\sfd(\cdot,\bar\gamma_t)}\mm,
\]
having also used the trivial bound  $|\nabla \sfd(\cdot,\bar\gamma_t)|\geq 0$ $\mm$-a.e.. It directly follows that for $t\gg1$ we have $\b_t\in D(\bd,\Omega)$ with $\bd\b^+\restr\Omega\geq\frac{N}{\sfd(\cdot,\bar\gamma_t)}\mm\restr\Omega\geq \frac{N}{\sfd(\bar\gamma_t,\Omega)}\mm\restr\Omega$ and therefore
\[
-\int_X\la\nabla f,\nabla\b_t\ra\,\d\mm=\int_\Omega f\,\d\bd\b_t\restr\Omega\geq \frac{N}{\sfd(\bar\gamma_t,\Omega)}\int_\Omega f\,\d\mm\to 0,\qquad\textrm{ as }t\to+\infty.
\]
To conclude it is  therefore sufficient to show that
\[
\lim_{t\to+\infty}\int_X\la\nabla f,\nabla\b_t\ra\,\d\mm=\int_X\la\nabla f,\nabla\b^+\ra\,\d\mm.
\]
To see this, notice that $\{\b_t\}_{t\geq 0}$ is a bounded family in $W^{1,2}(\Omega)$ and therefore, since $W^{1,2}(\Omega)$ is Hilbert by Theorem \ref{thm:calculus}, weakly relatively compact  in $W^{1,2}(\Omega)$. The uniform convergence of $(\b_t)$ to $\b^+$ when $t\to +\infty$ grants in particular the convergence in $L^2(\Omega)$ and therefore $(\b_t)$ weakly converges to $\b^+$ as $t\to+\infty$ in $W^{1,2}(\Omega)$. Conclude observing that the inequality
\[
\int_X\la\nabla f,\nabla g\ra\,\d\mm\leq \int_\Omega\weakgrad f\weakgrad g\,\d\mm\leq\sqrt{\int_\Omega\weakgrad f^2\,\d\mm}\sqrt{\int_\Omega\weakgrad g^2\,\d\mm}\leq \|f\|_{W^{1,2}(\Omega)}\|g\|_{W^{1,2}(\Omega)},
\]
shows that the linear map $W^{1,2}(\Omega)\ni g\mapsto \int_X\la\nabla f,\nabla g\ra\,\d\mm$ is continuous.
\end{proof}
We now use the strong maximum principle to deduce that $\b^++\b^-\equiv 0$ and that  $\bd\b^\pm=0$.
\begin{theorem}[Harmonicity of the Busemann function]\label{thm:bharm}
We have  $\b^++\b^-\equiv 0$ and $\bd\b^+=\bd\b^-=0$.
\end{theorem}
\begin{proof} Put $g:=\b^++\b^-$ and notice that by the linearity of the Laplacian we have $g\in D(\bd)$ with $\bd g\geq 0$. It is obvious that $g$ is Lipschitz, that $g\leq 0$ (by the triangle inequality) and that $g(\bar\gamma_t)=0$ for any $t\in\R$. Thus according to the strong maximum principle (Theorem \ref{thm:strongmax}) to conclude it is sufficient to show that for any non-positive $f\in \test X$ it holds
\[
\int_\Omega|\nabla g|^2\,\d\mm\leq \int_\Omega|\nabla(g+f)|^2\,\d\mm,
\]
where $\Omega\subset X$ is any bounded open set containing $\supp(f)$. This is an obvious consequence of the convexity of $\eps\mapsto\int_\Omega|\nabla(g+\eps f)|^2\,\d\mm$ and the inequality $\bd g\geq 0$:
\[
\begin{split}
\int_\Omega|\nabla(g+f)|^2\,\d\mm-\int_\Omega|\nabla g|^2\,\d\mm&\geq\lim_{\eps\downarrow0}\int_\Omega\frac{|\nabla(g+\eps f)|^2-|\nabla g|^2}{\eps}\,\d\mm\\
&=2\int_\Omega\la\nabla g,\nabla f\ra\,\d\mm=-2\int f\,\d\bd g\geq 0,
\end{split}
\]
and the proof is completed.
\end{proof}
From now on, to simplify the notation we shall consider the Busemann function $\b:X\to\R$ defined as
\begin{equation}
\label{eq:defb}
\b:=\b^+=-\b^-.
\end{equation}

\begin{theorem}[Multiples of $\b$ are Kantorovich potentials]\label{thm:basemetric}
For every $a\in\R$ the function $a\b$ is $c$-concave and fulfills
\begin{equation}
\label{eq:abcconc}
\begin{split}
(a\b)^c&=-a\b-\frac{a^2}2,\\
(-a\b)^c&=a\b-\frac{a^2}2.
\end{split}
\end{equation}
In particular, $(x,y)\in\partial^c(a\b)$ if and only if $(y,x)\in \partial^c(-a\b)$.

\end{theorem}
\begin{proof} Fix $a\in\R$ and notice that since $a\b$ is $|a|$-Lipschitz we have
\[
a\b(x)-a\b(y)\leq|a|\sfd(x,y)\leq \frac{\sfd^2(x,y)}2+\frac{a^2}2,\qquad\forall x,y\in X,
\]
which yields $\frac{\sfd^2(x,y)}{2}-a\b(x)\geq -a\b(y)-\frac{a^2}2$ for any $x,y\in X$, and thus
\[
(a\b)^c(y)\geq -a\b(y)-\frac{a^2}2,\qquad\forall y\in X.
\]
To prove the opposite inequality, fix $y\in X$ and assume for the moment $a\geq 0$. Let  $\gamma^{t,y}:[0,\sfd(y,\bar\gamma_t)]\to X$ be a unit speed geodesic connecting $y$ to $\bar\gamma_t$ and notice that since $(X,\sfd)$ is proper, for some sequence $t_n\uparrow +\infty$ the sequence $n\mapsto \gamma^{t_n,y}_a$ converges to some point $ y_a\in X$ which clearly has distance $a$ from $y$. 

Letting $n\to\infty$  in
\[
\begin{split}
t_n-\sfd(y_a,\bar\gamma_{t_n})\geq t_n-\sfd(\gamma^{t_n,y}_a,\bar\gamma_{t_n})-\sfd(y_a,\gamma^{t_n,y}_a)= t_n-\sfd(y,\bar\gamma_{t_n})+a-\sfd(y_a,\gamma^{t_n,y}_a),
\end{split}
\]
and recalling that $\b=\b^+=\lim_{n\to\infty}t_n-\sfd(\cdot,\bar\gamma_{t_n})$ we deduce 
\begin{equation}
\label{eq:perdopo}
\b(y_a)\geq \b(y)+a.
\end{equation}
Choosing $ y_a$ as competitor in the definition of $(a\b)^c(y)$ we obtain
\[
(a\b)^c(y)=\inf_x\frac{\sfd^2(x,y)}{2}-a\b(x)\leq\frac{\sfd^2( y_a,y)}{2}-a\b( y_a)\stackrel{\eqref{eq:perdopo}}\leq -a\b(y)-\frac{a^2}2,
\]
as desired. The case $a\leq0$ is handled analogously by letting $y_a$ be any limit of $\gamma^{-t,y}_{|a|}$ as $t\to+\infty$ and using the fact that $\b=-\b^-=\lim_{t\to+\infty}\sfd(\cdot,\bar\gamma_{-t})-t$.

This proves the first identity in \eqref{eq:abcconc}. The second follows from the first choosing $-a$ in place of $a$. Finally, the $c$-concavity of $a\b$ is obtained by direct algebraic manipulation:
\[
(a\b)^{cc}=\left(-a\b-\frac{a^2}2\right)^c=(-a\b)^c+\frac{a^2}2=a\b.
\]
The last assertion follows from the fact that $(x,y)\in\partial^c(a\b)$ if and only if $(y,x)\in\partial^c(a\b)^c$ and identities \eqref{eq:abcconc}.
\end{proof}

\subsection{The gradient flow of $\b$ preserves the measure}

\begin{proposition}\label{prop:gfb}
There exists a Borel map $\R\times X\ni (t,x)\mapsto \X_t(x)\in X$ such that for $\mm$-a.e. $x\in X$ the curve $t\mapsto \X_t(x)$ is continuous and fulfills  $\X_t(x)\in\partial^c(t\b)(x)$. Such curve is unique up to $\mm$-a.e. equality. Furthermore we have
\begin{align}
\label{eq:almostmeaspres}
\mm&\ll(\X_t)_\sharp\mm\ll\mm,&&\hspace{-1cm}\forall t\in\R,\\
\label{eq:group}
\X_{t+s}(x)&=\X_t(\X_s(x)),&&\hspace{-1cm}\mm\ae\ x\in\X,\ t,s\in\R,
\end{align}
and for $\mm$-a.e. $x\in X$ the curve $\R\ni t\mapsto\X_t(x)$ is a unit speed geodesic, i.e. a line.
\end{proposition}
\begin{proof} By Theorem \ref{thm:exp} and the fact that $t\b$ is a Kantorovich potential for every $t\in\R$ we deduce that there is a Borel negligible set $\mathcal N\subset X$ such that for $x\in X\setminus\mathcal N$ and $t_0\in\Q$ the set $\partial^c(t\b)(x)$ is a singleton and there is a unique geodesic $[0,1]\ni t\mapsto T_t(t_0,x)\in X$ such that $T_0(t_0,x)=x$ and  $T_1(t_0,x)\in\partial^c(t_0\b)(x)$. By the property \eqref{eq:tkant} we have that
\[
T_t(t_0,x)=T_1(tt_0,x),\qquad\forall t,t_0\in\Q,\ x\in X\setminus \mathcal N.
\]
It follows that for any $t\in\R$ and $x\in X\setminus\mathcal N$ the definition
\[
\X_t(x):=T_{\frac t{t_0}}(t_0,x),\qquad \forall t_0\in\Q\text{ such that }\tfrac t{t_0}\in[0,1],
\] 
is well posed and defines a curve which is a geodesic when restricted to $[0,+\infty)$ and $(-\infty,0]$. The uniqueness of such $\X$ follows by the construction and a simple continuity argument gives $\X_t(x)\in\partial^c(t\b)(x)$ for every $t\in\R$ and $x\in X\setminus\mathcal N$. Notice also that by the property \eqref{eq:acexp} we deduce that $(\X_t)_\sharp\mm\ll\mm$ for every $t\in\R$.

For the group property \eqref{eq:group}, start assuming that $t,s\geq 0 $ and pick $x\in X\setminus(\mathcal N\cup\X_s^{-1}(\mathcal N)\cup\X_{t+s}^{-1}(\mathcal N))$ (notice that $\mathcal N\cup\X_s^{-1}(\mathcal N)\cup\X_{t+s}^{-1}(\mathcal N)$ is Borel and $\mm$-negligible) and observe that from  $\X_s(x)\in\partial^c(s\b)(x)$ and  the relations \eqref{eq:abcconc} we get
\[
s\b(x)-s\b(\X_t(x))=\frac{\sfd^2(x,\X_s(x))}{2}+\frac{s^2}2,
\]
which, due to the fact that $\b$ is 1-Lipschitz, forces 
\begin{equation}
\label{eq:treno2}
\sfd(x,\X_s(x))=s,\qquad\text{ and }\qquad \b(x)-\b(\X_s(x))=s.
\end{equation}
Similarly, from $x\notin \X_s^{-1}(\mathcal N)$ we have $\X_t(\X_s(x))\in\partial^c(t\b)(\X_s(x))$ which forces 
\begin{equation}
\label{eq:treno3}
\sfd\big(\X_t(\X_s(x)),\X_s(x)\big)=t,\qquad\text{ and }\qquad\b(\X_s(x))-\b\big(\X_t(\X_s(x))\big)=t.
\end{equation}
From \eqref{eq:treno2} and \eqref{eq:treno3} we get $\b(x)-\b(\X_t(\X_s(x)))=t+s$ and $\sfd(x,\X_t(\X_s(x)))\leq t+s$ and thus recalling the relations \eqref{eq:abcconc} again, we get
\[
\frac{\sfd^2(x,\X_t(\X_s(x)))}2\leq (t+s)\b(x)-\big((t+s)\b\big)^c\big(\X_t(\X_s(x))\big),
\]
which means  $\X_t(\X_s(x))\in\partial^c((t+s)\b)(x)$. Given that  $x\notin \X_{t+s}^{-1}(\mathcal N)$, this forces $\X_t(\X_s(x))=\X_{t+s}(x)$, as desired.

To get the full group property it is now sufficient to show that for $t\in \Q$ and $x\in X\setminus(\mathcal N\cup\X_t^{-1}(\mathcal N)\cup\X_{-t}^{-1}(\mathcal N)) $ it holds 
\begin{equation}
\label{eq:treno4}
\X_{-t}(\X_t(x))=x,\qquad\text{ and }\qquad\X_t(\X_{-t}(x))=x.
\end{equation}
To check the first notice that we have $\X_t(x)\in\partial^c(t\b)(x)$ and thus by the last assertion in Theorem \ref{thm:basemetric} that $x\in\partial^c(-t\b)(\X_t(x))$. Since $x\notin \X_t^{-1}(\mathcal N)$ we know that $\partial^c(-t\b)(\X_t(x))$ contains only the point $\X_{-t}(\X_t(x))$, we deduce that the first equality in \eqref{eq:treno4} indeed holds. The second is proved analogously.

To prove that $\R\ni t\mapsto\X_t(x)$ is a geodesic for $\mm$-a.e. $x\in X$ it sufficient to prove that $[-T,T]\ni t\mapsto \X_t(x)$ is a geodesic for $\mm$-a.e. $x\in X$ and every $T>0$. This follows from the group property, which grants that $\X_t(x)=\X_{t+T}(\X_{-T}(x))$ for $\mm$-a.e. $x\in X$, and the fact that $[0,2T]\ni t\mapsto \X_t(x)$ is a geodesic, as pointed out in the first part  of the proof.

Finally, the first in \eqref{eq:almostmeaspres} follows from the second one and the group property.
\end{proof}
We shall refer to the map $(t,x)\mapsto\X_t(x)$ as the gradient flow of $\b$ although in fact we characterized it by the property $\X_t(x)\in\partial^c(t\b)(x)$. It is indeed easy to see that in the smooth setting this is really the gradient flow of $\b$ in the sense that it satisfies $\partial_t\X_t=-\nabla\b(\X_t)$. In our context, this property is expressed by the derivation rule   \eqref{eq:derfiga} given below and the group law \eqref{eq:group}.
\begin{theorem}[The gradient flow of $\b$ preserves the measure]\label{thm:gfpresmeas} The map $\R\times X\ni (t,x)\mapsto \X_t(x)\in X$ given by Proposition \ref{prop:gfb} satisfies
\begin{equation}
\label{eq:gfb1}
\begin{split}
(\X_t)_\sharp\mm&=\mm,\qquad\forall t\in\R.\\
\end{split}
\end{equation}
\end{theorem}
\begin{proof} Pick $t\in\R$, $\mu\in\probt X$ absolutely continuous w.r.t. $\mm$ and with bounded support and notice that since $\X_t(x)\in\partial^c(t\b)$ for $\mu$-a.e. $x$, $t\b$ is a Kantorovich potential from $\mu$ to $(\X_t)_\sharp\mu$. It is trivial that $(\X_t)_\sharp\mu$ has bounded support, hence by Proposition \ref{thm:enin} and the fact that $\bd\b=0$ we deduce $\u_N((\X_t)_\sharp\mu)\geq\u_N(\mu)$. Proposition \ref{prop:gfb} grants that $(\X_t)_\sharp\mu\ll\mm$ and $(\X_{-t})_\sharp(\X_t)_\sharp\mu=\mu$, hence the same argument applied to the couple $((\X_t)_\sharp\mu,\mu)$ in place of $(\mu,(\X_t)_\sharp\mu)$ and with $-t$ in place of $t$ yields  the reverse inequality and thus that
\begin{equation}
\label{eq:gfnoncambia}
\u_N(\mu)=\u_N((\X_t)_\sharp\mu),\qquad\forall t\in\R,\ \mu\in\probt X,\ \mu\ll\mm. 
\end{equation}
From this identity the conclusion follows easily. Indeed, recalling the first in \eqref{eq:almostmeaspres}, for $t\in\R$ we define the map $|\d\X_t|:X\to\R^+$ $\mm$-a.e. by
\[
|\d\X_t|:=\frac{\d\mm}{\d(\X_t)_\sharp\mm}\circ\X_{-t}.
\]
Then for $\mu=\rho\mm$ the equalities
\[
\int f\,\d(\X_t)_\sharp\mu=\int f\circ \X_t\rho\,\d\mm=\int f\rho\circ\X_{-t}\,\d(\X_t)_\sharp\mm=\int f\frac{\rho}{|\d\X_t|}\circ \X_{-t}\,\d\mm,
\]
valid for any Borel $f:X\to\R^+$ show that $\frac{\d(\X_t)_\sharp\mu}{\d\mm}=\frac{\rho}{|\d\X_t|}\circ \X_{-t}$ and in particular
\[
\begin{split}
\u_N((\X_t)_\sharp\mu)=-\int \left(\frac{\rho}{|\d\X_t|}\right)^{-\frac1N}\circ \X_{-t}\,\d (\X_t)_\sharp\mu=-\int \left(\frac{\rho}{|\d\X_t|}\right)^{-\frac1N}\,\d \mu=-\int\frac{\rho^{1-\frac1N}}{|\d\X_t|^{-\frac1N}}\,\d\mm.
\end{split}
\]
Taking into account \eqref{eq:gfnoncambia} and the arbitrariness of $\mu=\rho\mm$, the latter identity forces $|\d\X_t|=1$ $\mm$-a.e., which is the thesis.
\end{proof}

The measure preservation property just proved   has the following important consequences about the behavior of Sobolev functions along the flow:
\begin{proposition}\label{prop:nizza} For $f\in\s^2(X)$ we have
\begin{equation}
\label{eq:lungogf2}
\int|f(\X_t(x))-f(x)|^2\,\d\mm(x)\leq t^2\int\weakgrad f^2(x)\,\d\mm(x),\qquad\forall t\in\R,
\end{equation}
and
\begin{equation}
\label{eq:derfiga}
\lim_{t\to 0}\frac{f\circ\X_t-f}{t}=-\la\nabla f, \nabla\b\ra,\qquad\textrm{ weakly in }L^2(X).
\end{equation}
\end{proposition}
\begin{proof} Let $f\in\s^2(X)$. We claim that for every $t\in\R$  it holds
\begin{equation}
\label{eq:claimrighe}
|f(\X_t(x))-f(x)|\leq \int_0^t\weakgrad f(\X_s(x))\,\d s,\qquad\mm\ae\ x\in X,
\end{equation}
with the obvious interpretation of the right hand side for $t<0$. Indeed, fix $t_0\in \R$, and let $T:X\to C([0,1],X)$ be $\mm$-a.e. defined by $(T(x))_t:=\X_{tt_0}(x)$, let $\tilde\mm\in\prob X$ be such that $\tilde\mm\leq \mm$ and $\mm\ll\tilde\mm$ and put   $\ppi:=T_\sharp\tilde\mm\in \prob{C([0,1],X)}$. Then by Proposition \ref{prop:gfb}, $\ppi$ is concentrated on geodesics of speed $|t_0|$  and $(\e_t)_\sharp\ppi=(\X_{tt_0})_\sharp\tilde\mm\leq (\X_{tt_0})_\sharp\mm=\mm$ for every $t\in[0,1]$. Thus $\ppi$ is a test plan and  inequality \eqref{eq:localplan} yields
\[
|f(\gamma_1)-f(\gamma_0)|\leq \int_0^1\weakgrad f(\gamma_s)|\dot\gamma_s|\,\d s=|t_0|\int_0^1\weakgrad f(\gamma_s)\,\d s,\qquad\ppi\ae \ \gamma,
\]
which by definition of $\ppi$ is equivalent to the claim \eqref{eq:claimrighe}. Now   square and integrate \eqref{eq:claimrighe} to get
\[
\begin{split}
\int | f(\X_t(x))-f(x)|^2\,\d\mm(x)&\leq \int\left(\int_0^t\weakgrad f(\X_s(x))\,\d s\right)^2\,\d\mm(x)\leq t\iint_0^t\weakgrad f^2(\X_s(x))\,\d s\,\d\mm(x)\\
&=t \int\weakgrad f^2(x)\,\d\left(\int_0^t (\X_s)_\sharp\mm\,\d s\right)(x)=t^2\int\weakgrad f^2(x)\,\d \mm(x),
\end{split}
\]
which is \eqref{eq:lungogf2}. Finally, observe that \eqref{eq:lungogf2} grants that the $L^2$-norm of $\frac{f\circ\X_t-f}{t}$ is uniformly bounded, thus with a trivial density argument to conclude is sufficient to show that for any non-negative $g\in L^1\cap L^\infty(X)$ with bounded support it holds
\[
\lim_{t\downarrow 0}\int\frac{f\circ\X_t-f}{t}g\,\d\mm=-\int \la\nabla f,\nabla\b\ra g\,\d\mm,
\]
the proof of the limiting property as $t\uparrow0$ being analogous. Pick such $g$, assume $g$ is not identically 0 (otherwise there is nothing to prove), define $\mu:=(\int g\,\d\mm)^{-1}g\mm\in\prob X$ and $\ppi:=S_\sharp\mu\in \prob{(C([0,1],X))}$, where $S:X\to C([0,1],X)$ is  given by $(S(x))_t:=\X_t(x)$. By construction, for some bounded open set $\Omega$ it holds $\supp((\e_t)_\sharp\ppi)\subset\Omega$ for any $t\in[0,1]$ and thus  Proposition \ref{prop:gfb} and Corollary \ref{cor:metrbre} grant that $\ppi$ represents the gradient of $-\b$. By the first order differentiation formula given by Theorem \ref{thm:horver} we deduce
\[
\lim_{t\downarrow0}\int \frac{f\circ\X_t-f}{t}g\,\d\mm=\lim_{t\downarrow 0}\int\frac{f(\gamma_t)-f(\gamma_0)}{t}\,\d\ppi(\gamma)=-\int \la\nabla f,\nabla \b\ra g\,\d\mm,
\]
which is the thesis.
\end{proof}
Notice that the bound \eqref{eq:claimrighe} applied to (a cut-off of) $\b$ gives, taking into account that $\b$ is 1-Lipschitz:
\begin{equation}
\label{eq:nablab}
|\nabla\b|=1,\qquad\mm\ae.
\end{equation}
In fact, this could also be deduced by the fine results of Cheeger \cite{Cheeger00}, but we pointed out this argument in order to give an  exposition independent on Cheeger's analysis.

 \subsection{\underline{Things to know:}\ heat flow and Bakry-\'Emery contraction estimate}\label{se:baseheat}
In the following we will need to work with the heat flow on our $\RCD(0,N)$ space $(X,\sfd,\mm)$: on one side as regularizing flow in a setting where standard convolution techniques are unavailable (see the proof of Theorem \ref{prop:crucial}), and on the other as tool to get the  hands on - under minimal regularity assumptions - the Bochner inequality (see \eqref{eq:BE} below and its consequences in Proposition \ref{prop:eulerb}).

Start noticing that being $(X,\sfd,\mm)$ infinitesimally Hilbertian, the map 
\[
L^2(X)\ni f\mapsto \mathcal E(f):=\frac12\int\weakgrad f^2\,\d\mm,
\]
set to $+\infty$ if $f\notin W^{1,2}(X)$ is a Dirichlet form. By polarization, it defines a bilinear map $W^{1,2}\ni f,g\mapsto \mathcal E(f,g)\in\R$ so that $\mathcal E(f,f)=\mathcal E(f)$ and from Theorem \ref{thm:calculus} and its proof it is immediate to see that
\[
\mathcal E(f,g)=\frac12\int\la\nabla f,\nabla g\ra\d\mm.
\]
We can then consider the  evolution semigroup associated to $\mathcal E$ in $L^2(X)$ or, which is equivalent, its gradient flow in $L^2(X)$. This means that we define 
$D(\Delta)\subset W^{1,2}(X)$ and $\Delta:D(\Delta)\to L^2(X)$ by declaring that $f\in D(\Delta)$ with $\Delta f=h$ provided  for every $g\in W^{1,2}(X)$ it holds
\[
\int gh\,\d\mm=-\int\la\nabla f,\nabla g\ra\d\mm.
\]
Notice that in fact this definition is nothing but a particular case of the one of measure valued Laplacian given in Definition \ref{def:distrlap}. Indeed, it is immediate to verify that 
\[
f\in D(\Delta)
\]
is equivalent to 
\[
\textrm{$f\in W^{1,2}(X,\sfd,\mm)\cap D(\bd)$ and $\bd f=h\mm$ for some $h\in L^2(X,\mm)$,}
\]
and that if these holds we also have $h=\Delta f$: one implication is obvious, and the other one follows from the approximation result in Theorem \ref{thm:energylip}. Yet, to single out the definition of $\Delta$ is useful because it allows us to directly use the regularization properties of the heat flow classical in the context of linear semigroups, see in particular the proof of Theorem \ref{prop:crucial}.

Then the heat flow $\h_t:L^2(X)\to L^2(X)$, $t\geq 0$ is the unique family of maps such that for any $f\in L^2(X)$ the curve $[0,\infty)\ni t\mapsto \h_t(f)\in L^2(X)$ is continuous, locally absolutely continuous on $(0,+\infty)$, fulfills $\h_0(f)=f$, $\h_t(f)\in D(\Delta)$ for $t>0$ and solves
\[
\frac{\d}{\d t}\h_t(f)=\Delta\h_t( f),\qquad \mathcal L^1\ae \ t>0.
\]
Notice that by direct computation we have  $\frac{\d}{\d t}\|\h_t(f)\|^2_{L^2}=-4\mathcal E(\h_t(f))$, and using the fact that $\mathcal E$ is decreasing along the flow, after little algebraic manipulation we get the simple yet useful bound:
\begin{equation}
\label{eq:l2w12}
\|\h_t(f)\|_{W^{1,2}}\leq \frac1{\sqrt{2t}}\|f\|_{L^2}.
\end{equation}

The fact that the measure $\mm$ is doubling (see \eqref{eq:BG}) and $(X,\sfd,\mm)$ supports a 1-2 weak local Poincar\'e inequality (see \cite{Lott-Villani07}, \cite{Rajala12} and \cite{Rajala12-2}) already grant important properties of this flow. In particular, from the general results obtained by Sturm in \cite{Sturm96I}, \cite{Sturm96III}, we get the existence of a {\bf mass preserving heat kernel} satisfying {\bf Gaussian estimates}, i.e. there is a map $(0,+\infty)\times X^2\ni (t,x,y)\mapsto \rho_t[x](y)=\rho_t[y](x)\in\R$ such that $\int\rho_t[x]\,\d\mm=1$ and
\[
0<\rho_t[x](y)\leq\frac{ \mathcal C}{\mm(B_{\sqrt t}(x))}\,e^{-\dfrac{\sfd^2(x,y)}{5t}},
\]
for some  constant $\mathcal C$ depending only on $(X,\sfd,\mm)$ (in particular thanks to the polynomial volume growth \eqref{eq:BG} this grants $\rho_t[x]\in L^2(X)$ for every $t>0$, $x\in X$) and
\begin{equation}
\label{eq:heatrepr}
\h_t(f)(x)=\int f\rho_t[x]\,\d\mm,
\end{equation}
for every $f\in L^2(X)$ and $\mm$-a.e. $x\in X$. Very shortly and roughly said, the Gaussian bounds  are a consequence of a   generalization to non-smooth spaces of De Giorgi-Moser-Nash type arguments for regularity theory for parabolic equations, see \cite{Sturm96III} and references therein for more details. The mass preservation follows instead from the volume growth estimate along techniques that in the smooth setting are due to Grigoryan \cite{Grigoryan99}, see also the recent generalization to non-linear heat flow in Finsler-type geometries given in \cite{AmbrosioGigliSavare11}.

Later on we will want to evaluate the heat flow starting from the Busemann function $\b$, which certainly is not in $L^2(X)$. Yet, this is not a big issue, because the Gaussian estimates and the polynomial volume growth allow to extend the domain of the definition of the heat flow far beyond the space $L^2(X,\mm)$. We will be satisfied in considering as Domain of the Heat flow the (non maximal) space $\dom(X)=\dom(X,\sfd,\mm,\bar x)$ defined by
\[
\dom(X):=\Big\{f:X\to\R\ \textrm{ Borel : }\int |f|(x)e^{-\sfd(x,\bar x)}\,\d\mm(x)<\infty\Big\},
\]
where $\bar x\in X$ is a point that we shall consider as fixed from now on. It is immediate to check that $\|f\|_\dom:=\int |f|e^{-\sfd(\cdot,\bar x)}\,\d\mm$ is a norm on $\dom(X)$, that $(\dom(X),\|\cdot\|_\dom)$ is a Banach space, that the right hand side of formula \eqref{eq:heatrepr} makes sense for general $f\in\dom(X)$ and that the bound
\begin{equation}
\label{eq:hdom}
\|\h_t(f)\|_\dom\leq \mathcal C(t)\|f\|_\dom,
\end{equation}
holds for some constants $\mathcal C(t)$ depending only on $t$ and the space $(X,\sfd,\mm)$. We omit the simple details.

The Riemannian curvature dimension condition $\RCD(0,N)$ ensures further regularizing properties of the heat flow, in particular we have the {\bf Bakry-\'Emery} contraction estimate
\begin{equation}
\label{eq:BE}
\weakgrad{\h_t(f)}^2\leq \h_t(\weakgrad f^2),\qquad\mm\ae,\ \forall t\geq 0,
\end{equation}
valid for any $f\in W^{1,2}(X)$. We recall that in the smooth Riemannian case this inequality is equivalent to the dimension free Bochner inequality
\begin{equation}
\label{eq:Bochner}
\Delta\frac{|\nabla f|^2}2\geq \la\nabla\Delta f,\nabla f\ra,
\end{equation}
indeed to get \eqref{eq:Bochner} from \eqref{eq:BE} to  just differentiate at time $t=0$, while for the other way around differentiate in $s$ the map $\h_{s}(|\nabla(\h_{t-s}f)|^2)$, use \eqref{eq:Bochner} and integrate from $s=0$ to $s=t$. In the non-smooth setting, \eqref{eq:BE} has been proved at first in \cite{Gigli-Kuwada-Ohta10} in the context of finite dimensional Alexandrov spaces with curvature bounded from below with a technique which, as shown in \cite{AmbrosioGigliSavare11-2}, generalizes to $\RCD(0,\infty)$ spaces (see also \cite{AmbrosioGigliSavare12}, \cite{Erbar-Kuwada-Sturm13}, \cite{AmbrosioMondinoSavare13} for more recent progresses). The argument of the proof uses in a crucial way the identification of the gradient flow of the Dirichlet energy in $L^2$ with the one of the relative entropy in $(\probt X,W_2)$ (\cite{Gigli-Kuwada-Ohta10},  \cite{AmbrosioGigliSavare11}) together with a very general duality argument due to Kuwada \cite{Kuwada10}.


\subsection{The gradient flow of $\b$ preserves the Dirichlet energy}\label{se:dist}

In this section we prove that the right composition with $\X_t$ preserves the Dirichlet energy $\mathcal E$.

Notice that being $\b$ Lipschitz, it certainly belongs to $\dom(X)$, so that $\h_t(\b)$ is well defined. Then the fact that $\bd\b=0$ strongly suggests that $\b$ is invariant under the heat flow, i.e.:
\begin{equation}
\label{eq:bstabile}
\h_t(\b)(x)=\b(x),\qquad\mm\ae\ x\in X.
\end{equation}
This is indeed the case, the proof being based on the consistency of the notion of Laplacian $\Delta$  in $L^2$ with that of distributional Laplacian $\bd$ pointed out at the beginning of section \ref{se:baseheat} and an approximation argument. We omit the uninspiring  technical details.

This invariance property and the  Bakry-\'Emery condition \eqref{eq:BE} are the ingredient needed to obtain the following crucial Euler's equation of $\b$:
\begin{proposition}[Euler's equation of $\b$]\label{prop:eulerb}
For any $f\in W^{1,2}(X)$ it holds
\begin{equation}
\label{eq:euler1}
\h_t(\la\nabla\b,\nabla f\ra)=\la\nabla\b,\nabla\h_t(f)\ra,\qquad\mm\ae,
\end{equation}
and for every $f\in D(\Delta)$ with $\Delta f\in W^{1,2}(X)$ and every $g\in D(\Delta)$ it holds
\begin{equation}
\label{eq:euler2}
\int \Delta g\la\nabla\b,\nabla f\ra\d\mm=\int g\la\nabla\b,\nabla\Delta f\ra \d\mm,
\end{equation}
\end{proposition}
\begin{proof} 
Pick $f\in W^{1,2}(X)$, $\eps\in\R$, put $\b_\eps:=\b+\eps f$ and observe that $\b_\eps\in  \dom\cap \s^2_{\rm loc}(X)$.  Our first task is to write the Bakry-\'Emery contraction estimate \eqref{eq:BE} for $\b_\eps$ Let $(B_n)$ be an increasing sequence of bounded sets covering $X$ and for every $n\in\N$ let $\nchi_n:X \to[0,1]$ be a 1-Lipschitz function with compact support identically 1 on $B_n$. Obviously, $\nchi_n\b_\eps\in W^{1,2}(X)$ so that  \eqref{eq:BE} yields
\begin{equation}
\label{eq:beperlim}
|\nabla (\h_t(\nchi_n\b_\eps))|^2\leq \h_t(|\nabla (\nchi_n\b_\eps)|^2),\qquad\mm\ae,\qquad\forall t\geq 0.
\end{equation}
Since $\nchi_n$ is 1-Lipschitz with values in $[0,1]$ we have  $|\nabla (\nchi_n\b_\eps)|^2\leq 2|\nabla \b_\eps|^2+2|\b_\eps|^2$ and it is easy to see that the right hand side belongs to $\dom(X)$. Given that trivially $|\nabla(\nchi_n\b_\eps)|\to|\nabla \b_\eps|$ $\mm$-a.e. as $n\to\infty$, by the dominate convergence theorem we deduce $\||\nabla (\nchi_n\b_\eps)|^2-|\nabla \b_\eps|^2\|_\dom\to 0$ as $n\to\infty$. Hence  inequality \eqref{eq:hdom} grants that $\h_t(|\nabla (\nchi_n\b_\eps)|^2)\to\h_t(|\nabla \b_\eps|^2)$ in $\dom(X)$ and thus, up to pass to a non-relabeled subsequence we get that 
\begin{equation}
\label{eq:sonno}
\h_t(|\nabla(\nchi_n\b_\eps)|^2)\to \h_t(|\nabla\b_\eps|^2)\quad\mm\ae\textrm{  as $n\to\infty$ for any $t\geq 0$}.
\end{equation}
A similar argument gives that  $\h_t(\nchi_n\b_\eps)\to\h_t(\b_\eps)$ in $\dom(X)$ and $\mm$-a.e. as $n\to\infty$ so that taking into account  the lower semicontinuity of minimal weak upper gradients \eqref{eq:lscwug}, the limiting property  \eqref{eq:sonno} and letting $n\to\infty$ in \eqref{eq:beperlim} we deduce
\begin{equation}
\label{eq:bona}
\weakgrad{\h_t(\b_\eps)}^2\leq \h_t(\weakgrad {\b_\eps}^2),\qquad\mm\ae,\qquad\forall t\geq 0,
\end{equation}
as desired. Expanding both sides of this inequality using the linearity of the heat flow we get
\[
\begin{split}
|\nabla \h_t(\b+\eps f)|^2&=|\nabla\h_t(\b)|^2+2\eps\la\nabla(\h_t(\b)),\nabla \h_t(f) \ra+\eps^2|\nabla \h_t(f)|^2,\\
\h_t(|\nabla \b+\eps f|^2)&=\h_t(|\nabla\b|^2)+2\eps\h_t(\la\nabla\b,\nabla f \ra)+\eps^2\h_t(|\nabla f|^2),
\end{split}
\]
hence using \eqref{eq:bstabile}, the fact that $|\nabla\b|\equiv 1$ (recall \eqref{eq:nablab}) and the mass preservation which grants $\h_t(1)\equiv 1$, from \eqref{eq:bona} we obtain \eqref{eq:euler1}. Then \eqref{eq:euler2} follows multiplying \eqref{eq:euler1} by $g\in D(\Delta)$, integrating and differentiating at $t=0$.
\end{proof}
From these Euler's equations we can now deduce that the right composition with $\X_t$ preserves the Dirichlet energy. 
We shall need the identity
\begin{equation}
\label{eq:divb}
\int \la \nabla g,\nabla\b\ra f\,\d\mm=-\int \la \nabla f,\nabla\b\ra g\,\d\mm,\qquad\forall f,g\in W^{1,2}(X),
\end{equation}
which can be proved by first choosing sequences $(f_n),(g_n)\subset\test X$   converging to $f,g$ respectively in $W^{1,2}(X)$ (Theorem \ref{thm:stronglip}), then noticing that $\bd \b=0$ yields $\int\la \nabla (f_ng_n),\nabla\b\ra\,\d\mm=0$ and thus
\begin{equation}
\label{eq:divlip}
\int \la \nabla g_n,\nabla\b\ra f_n\,\d\mm=-\int \la \nabla f_n,\nabla\b\ra g_n\,\d\mm,\qquad\forall n\in\N,
\end{equation}
then observing that $\la\nabla f_n,\nabla\b\ra$ (resp. $\la\nabla g_n,\nabla\b\ra$) converge to $\la\nabla f,\nabla\b\ra $ (resp. to $\la\nabla g,\nabla\b\ra$) in $L^2(X)$ as $n\to\infty$  and finally passing to the limit in \eqref{eq:divlip}.

\begin{theorem}[Right compositions with $\X_t$ preserve the Dirichlet energy]\label{prop:crucial} 
For any $f\in L^2(X)$ and $t\in\R$ we have
\begin{equation}
\label{eq:dirpres}
\mathcal E(f\circ\X_t)=\mathcal E(f).
\end{equation}
\end{theorem}
\begin{proof} We claim that \eqref{eq:dirpres} holds for $f\in W^{1,2}(X)$. This will be sufficient to conclude by applying such claim also to $\X_{-t}$ and recalling the group property \eqref{eq:group}.

Fix such $f$ and recall inequality \eqref{eq:lungogf2} to get
\[
\int |f\circ\X_s-f\circ \X_t|^2\,\d\mm=\int |f\circ\X_{s-t}-f|^2\,\d\mm\leq |s-t|^2\int |\nabla f|^2\,\d\mm,
\]
which shows that the map $\R\ni t\mapsto f_t:=f\circ\X_t\in L^2(X)$ is Lipschitz with Lipschitz constant bounded by $\||\nabla f|\|_{L^2}$. Now notice that inequality \eqref{eq:l2w12} grants that 
\begin{equation}
\label{eq:elip}
\textrm{the map \quad$\R\ni t\quad\mapsto\quad \h_\eps(f_t)\in W^{1,2}(X)$\quad is Lipschitz for every $\eps>0$,}
\end{equation}
its Lipschitz constant being bounded by $\frac{1}{\sqrt{2\eps}}\||\nabla f|\|_{L^2}$. 

In particular, the map $t\mapsto \frac1{2}\int |\nabla \h_\eps(f_t)|^2\,\d\mm$ is Lipschitz; our aim is  to show that it is constant. Start from
\[
\begin{split}
\int |\nabla \h_\eps(f_{t+h})|^2-| \nabla\h_\eps(f_t)|^2\,\d\mm&= \int 2\la\nabla \h_\eps(f_t),\nabla {\h_\eps(f_{t+h}-f_t)}{}\ra+|\nabla(\h_\eps(f_{t+h}-f_t))|^2\,\d\mm,
\end{split}
\]
and notice that  \eqref{eq:l2w12} yields the bound  $\int |\nabla(\h_\eps(f_{t+h}-f_t))|^2\,\d\mm\leq \frac1{2\eps}\||\nabla f|\|^2_{L^2} |h|^2$ and thus  for any $t\in\R$ it holds
\[
\lim_{h\to 0}\int\frac{ |\nabla \h_\eps(f_{t+h})|^2-| \nabla\h_\eps(f_t)|^2}{2h}\,\d\mm=\lim_{h\to 0}\int\la \nabla \h_\eps(f_t),\nabla\frac{\h_\eps(f_{t+h})-\h_\eps(f_t)}h\ra\,\d\mm.
\]
We compute the limit in the right-hand-side of this expression:
\begin{equation}
\label{eq:cuneo2}
\begin{split}
\lim_{h\to 0}\int \la\nabla \h_\eps(f_t),\nabla\frac{\h_\eps(f_{t+h})-\h_\eps(f_t)}h\ra\,\d\mm&=-\lim_{h\to 0}\int \Delta \h_\eps( f_t) \frac{\h_\eps\big(f_{t+h}-f_t\big)}h\,\d\mm\\
&=-\lim_{h\to 0}\int \Delta \h_{2\eps}(f_t) \frac{ f_t\circ\X_{h}-f_t}h\,\d\mm\\
&=-\lim_{h\to 0}\int \frac{\big(\Delta \h_{2\eps}(f_t)\big)\circ\X_{-h}-\Delta \h_{2\eps}(f_t)}h   f_t \,\d\mm\\
&=-\int\la\nabla\big(\Delta \h_{2\eps}(f_t)\big),\nabla\b\ra\, f_t\,\d\mm,
\end{split}
\end{equation}
having used the measure preservation property in the third equality and the differentiation formula \eqref{eq:derfiga} in the last one. We claim that
\begin{equation}
\label{eq:claim0}
\int\la\nabla\big(\Delta \h_{2\eps}(g)\big),\nabla\b\ra\, g\,\d\mm=0,\qquad\forall g\in L^2(X).
\end{equation}
Notice that the map $L^2(X)\ni g\mapsto \Delta \h_{2\eps}(g)\in W^{1,2}(X)$ is continuous, thus from the fact that $\b$ is Lipschitz we get
\[
L^2(X)\ni g\quad\mapsto\quad \la\nabla\big(\Delta \h_{2\eps}(g)\big),\nabla\b\ra\in L^2(X)\qquad\textrm{ is continuous}.
\]
Hence it is sufficient to check \eqref{eq:claim0}  for $g\in D(\Delta)$ such that $\Delta g\in W^{1,2}(X)$, because - by regularization with the heat flow - the set of such $g$'s is dense in $L^2(X)$. With this choice of $g$, recalling the integration by parts formula \eqref{eq:divb} and the Euler equation \eqref{eq:euler2} we have
\begin{equation}
\label{eq:pezzoclaim1}
\begin{split}
\int\la\nabla\big(\Delta \h_{2\eps}(g)\big),\nabla\b\ra\, g\,\d\mm=-\int \Delta \h_{2\eps}(g)\la\nabla g,\nabla \b\ra\,\d\mm=-\int \h_{2\eps}(g)\la\nabla\Delta g,\nabla \b\ra\,\d\mm.
\end{split}
\end{equation}
On the other hand, the Euler equation \eqref{eq:euler1} applied with $\Delta g$ in place of $f$ yields
\[
\int\la\nabla\big(\Delta \h_{2\eps}(g)\big),\nabla\b\ra\, g\,\d\mm=\int\la\nabla\big( \h_{2\eps}(\Delta g)\big),\nabla\b\ra\, g\,\d\mm=\int \h_{2\eps}(\la\nabla\Delta g,\nabla\b\ra) g\,\d\mm,
\]
which together with \eqref{eq:pezzoclaim1} gives \eqref{eq:claim0}. According to \eqref{eq:elip} and \eqref{eq:cuneo2} we thus obtained that
\[
\R\ni t\qquad\mapsto\qquad\frac12\int|\nabla\h_\eps(f_t)|^2\,\d\mm,\qquad\textrm{is constant for every }\eps>0.
\]
Letting $\eps\downarrow0 $ and recalling that from the very definition of heat flow we have $\mathcal E(g)=\lim_{\eps\downarrow0}\frac12\int|\nabla\h_\eps(g)|^2\,\d\mm$ for any $g\in L^2(X)$,  we deduce that $t\mapsto \mathcal E(f_t)$ is constant, as desired.
\end{proof}

\section{Geometric consequences and conclusion}
\subsection{Isometries by duality with Sobolev functions}\label{se:dual}
We just proved that the right composition with $\X_t$ preserves the Dirichlet energy. In order to translate this Sobolev information into a metric one we shall make use of the following result, coming from \cite{AmbrosioGigliSavare11-2}. Notice that we simplified the statement below by asking the measure to be doubling, but this is actually unnecessary.
\begin{proposition}\label{prop:sobtolip} Let $(\tilde X,\tilde\sfd,\tilde\mm)$ be an $\RCD(0,\infty)$ space and such that $\tilde\mm$ is doubling. Then every $f\in W^{1,2}(\tilde X)$ with $|\nabla f|\leq 1$ $\tilde\mm$-a.e.  has a 1-Lipschitz representative, i.e. there exists $\tilde f:X\to\R$ 1-Lipschitz such that $f=\tilde f$ $\tilde\mm$-a.e..
\end{proposition}
\begin{sketch}
Let $x,y\in \tilde X$, $\eps>0$, put $\mu_x^\eps:=\frac1{\tilde \mm(B_\eps(x))}\tilde \mm\restr{B_\eps(x)}$, $\mu_y^\eps=\frac1{\tilde \mm(B_\eps(y))}\tilde \mm\restr{B_\eps(y)}$ and let $\ppi\in\gopt(\mu_x^\eps,\mu^\eps_y)$ be the unique optimal geodesic plan given by Theorem \ref{thm:maprcd}. Arguing as in the proof of Proposition \ref{prop:varphi} and Lemma \ref{le:ac}, we see that since both $\mu^\eps_x$ and $\mu^\eps_y$ have  bounded supports and densities, it holds $(\e_t)_\sharp\ppi\leq C\tilde \mm$ for every $t\in[0,1]$ where $C:=\max\{\frac1{\tilde \mm(B_\eps(x))},\frac1{\tilde \mm(B_\eps(y))}\}$. Thus the plan $\ppi$ is a test plan and for $f$ as in the assumptions we get
\[
\begin{split}
\left|\int f\,\d\mu^\eps_x-\int f\,\d\mu^\eps_y\right|&\leq\int|f(\gamma_1)-f(\gamma_0)|\,\d\ppi(\gamma)\leq \iint_0^1|\dot\gamma_t|\,\d t\,\d\ppi(\gamma)\leq \sqrt{\iint_0^1|\dot\gamma_t|^2\,\d t\,\d\ppi(\gamma)}.
\end{split}
\]
By construction the rightmost side is equal to $W_2(\mu^\eps_x,\mu^\eps_y)$, which converges to $\tilde\sfd(x,y)$ as $\eps\downarrow0$. Now use the fact that $\tilde \mm$ is doubling to deduce that $\tilde \mm$-a.e. $x$ is a Lebesgue point for $f$ (see for instance \cite{Heinonen01}), so that $\int f\,\d\mu^\eps_x\to f(x)$ for $\tilde \mm$-a.e. $x$.  The conclusion follows by considering $x,y$ Lebesgue points and letting $\eps\downarrow0$ in the above inequality.
\end{sketch}
It is worth noticing that the same conclusion of the above proposition fails if $(\tilde X,\tilde\sfd,\tilde\mm)$ is only assumed to support a weak local 1-1 Poincar\'e inequality with $\tilde\mm$ being doubling. Indeed, these assumptions are invariant under a bi-Lipschitz change of metric but it can be shown that any proper space fulfilling the thesis of Proposition \ref{prop:sobtolip} must be a geodesic space. The argument is the following. Define an $\eps$-chain connecting $x$ to $y$ as a finite sequence $\{x_i\}_{i=0,\ldots,n}$, $n\in\N$, such that $x_0=x$, $x_n=y$ and $\tilde\sfd(x_i,x_{i+1})\leq\eps$ for every $i$, then consider the function $f_\eps(y):=\inf\sum_i\sfd(x_i,x_{i+1})$, the $\inf$ being taken among all $\eps$-chains connecting $x$ to $y$ and notice that $f_\eps$ is locally 1-Lipschitz and thus, if the thesis of  the above proposition holds, globally 1-Lipschitz. Then let $\eps\downarrow0$ and use the assumption that the space is proper to find a geodesic connecting $x$ to $y$ as limit of minimizing $\eps$-chains. Notice the analogy of this argument with the one providing Semmes' Lemma as given in \cite{KleinerMackay11}.
\begin{theorem}\label{thm:dual}
Let $(X_1,\sfd_1,\mm_1)$, $(X_2,\sfd_2,\mm_2)$ be two $\RCD(0,\infty)$ spaces such that $\mm_1$ and $\mm_2$ are doubling and $T:X_1\to X_2$ an invertible map such that
\begin{equation}
\label{eq:assT}
\begin{split}
T_\sharp\mm_1&=\mm_2,\\
\mathcal E_1(f\circ T)&=\mathcal E_2(f),\qquad\forall f\in L^2(X_2),
\end{split}
\end{equation}
where $\mathcal E_i$ is the natural Dirichlet energy on the space $X_i$, $i=1,2$. Then $T$ is, up to a redefinition on a $\mm_1$-negligible set, an isometry from $(X_1,\sfd_1)$ to $(X_2,\sfd_2)$.
\end{theorem}
\begin{proof}
It is sufficient to prove that $T$ has a 1-Lipschitz representative, as then the same arguments can be carried out for the inverse.

Notice that from the assumptions \eqref{eq:assT} it directly follows that for $f\in W^{1,2}(X_2)$ we have $f\circ T\in W^{1,2}(X_1)$. We further claim that it holds
\begin{equation}
\label{eq:claimloc}
|\nabla(f\circ T)|=|\nabla f|\circ T,\qquad\mm_1\ae ,\qquad\forall f\in W^{1,2}\cap L^\infty(X_2).
\end{equation}
Fix such $f$ and let $g:X_2\to\R$ be Lipschitz with bounded support. Then $f^2,gf\in W^{1,2}(X_2)$ and from the Leibniz and chain rules we get
\begin{equation}
\label{eq:loc1}
\begin{split}
\int_{X_2}\!\!\! g|\nabla f|^2\,\d\mm_2&=\int_{X_2}\!\!\! \la\nabla (gf),\nabla f\ra-f\la\nabla g,\nabla f\ra\,\d\mm_2=\int_{X_2} \!\!\!\la\nabla (gf),\nabla f\ra-\la\nabla g,\nabla (f^2/2)\ra\,\d\mm_2.
\end{split}
\end{equation}
Now notice that the second in \eqref{eq:assT} yields, by polarization, the identity 
\[
\int_{X_2}\la\nabla f_1,\nabla f_2\ra\,\d\mm_2=\int_{X_1}\la\nabla (f_1\circ T),\nabla (f_2\circ T)\ra\,\d\mm_1,\qquad\forall f_1,f_2\in W^{1,2}(X_2).
\]
Using this equality in \eqref{eq:loc1}, putting for brevity $\tilde f:=f\circ T$, $\tilde g:=g\circ T$ and using again the Leibniz and chain rules we get
\begin{equation}
\label{eq:loc2}
\begin{split}
\int_{X_2} g|\nabla f|^2\,\d\mm_2&=\int_{X_1} \la\nabla (\tilde g\tilde f),\nabla \tilde f\ra-\la\nabla\tilde g,\nabla (\tilde f^2/2)\ra\,\d\mm_1=\int_{X_1}\tilde g|\nabla\tilde f|^2\,\d\mm_1.
\end{split}
\end{equation}
By the first in \eqref{eq:assT} we also have $\int_{X_2} g|\nabla f|^2\,\d\mm_2=\int_{X_1}\tilde g|\nabla f|^2\circ T\,\d\mm_1$, hence \eqref{eq:loc2}, the arbitrariness of $g$, the fact that $T$ is invertible and that $|\nabla f|^2\in L^1(X_2)$ give
\[
\int_{X_2} h|\nabla f|^2\circ T\,\d\mm_1=\int h|\nabla (f\circ T)|^2\,\d\mm_1,\qquad\forall h\in L^\infty(X_1),
\]
which implies our claim \eqref{eq:claimloc}.

Now let $\{x_n\}_{n\in\N}$ be a countable dense subset of $X_2$ and for $k,n\in\N$ consider the functions $f_{k,n}:=\max\{0,\min\{\sfd(\cdot,x_n),k-\sfd(\cdot,x_n)\}\}$. These are 1-Lipschitz and satisfy
\[
\sfd_2(x,y)=\sup_{k,n}|f_{k,n}(x)-f_{k,n}(y)|,\qquad\forall x,y\in X_2.
\]
Given that they also have bounded support, we have $f_{k,n}\in W^{1,2}\cap L^\infty(X_2) $ for any $k,n\in\N$ and the 1-Lipschitz property grants $|\nabla f_{k,n}|\leq 1$ $\mm_2$-a.e.. 

By \eqref{eq:claimloc} we deduce that $f_{k,n}\circ T$ also belongs to $W^{1,2}(X_1)$ with $|\nabla(f_{k,n}\circ T)|\leq 1$ $\mm_1$-a.e.. Now we use Proposition \ref{prop:sobtolip} to deduce that for every $k,n\in\N$ there exists a Borel $\mm_1$-negligible set $\mathcal N_{k,n}$ such that the restriction of $f_{k,n}\circ T$ to $X_1\setminus \mathcal N_{k,n}$ is 1-Lipschitz. Hence for any $x,y\in X\setminus \cup_{k,n}\mathcal N_{k,n}$ we have
\[
\sfd_1(x,y)\geq \sup_{k,n}|(f_{k,n}\circ T)(x)-(f_{k,n}\circ T)(y)|= \sup_{k,n}|f_{k,n}(T(x))-f_{k,n}(T(y))|=\sfd_2(T(x),T(y)).
\]
Given that $ \cup_{k,n}\mathcal N_{k,n}$ is Borel and negligible we conclude that $T$ has a 1-Lipschitz representative, as desired.
\end{proof}

\subsection{The  gradient flow of $\b$ preserves the distance}
The duality statement proved in the previous section and Theorem \ref{prop:crucial} quickly gives that there is a unique continuous representative of the gradient flow $\X_t$ of $\b$ which is a family of isometries.
\begin{theorem}[The gradient flow of $\b$ preserves the distance]\label{thm:gfpresdist} 
The following holds:
\begin{itemize}
\item[i)] There exists a unique continuous map $\bar\X:X\times\R\to X$ coinciding $\mm\times\mathcal L^1$-a.e. with $\X$ .
\item[ii)] For every $t_0\in\R$ and $x_0\in X$ the maps $X\ni x\mapsto \bar\X_{t_0}(x)$ and  $\R\ni t\mapsto\bar\X_{t}(x_0)$ are isometries of $X$ into itself and of $\R$ into $X$ respectively.
\item[iii)] It holds $\bar\X_t(\bar \X_s(x))=\bar\X_{t+s}(x)$, for any $x\in X$ and $t,s\in\R.$
\end{itemize}
\end{theorem}
\begin{proof}$\ $\\
\noindent{${\mathbf{(i),(ii)}}$} Uniqueness is obvious. By Theorems \ref{thm:gfpresmeas} and \ref{prop:crucial} we know that the assumptions of Theorem \eqref{thm:dual} are fulfilled with $X_1=X_2=X$ and $T=\X_t$  (recall \eqref{eq:BG} to get that $\mm$ is doubling). Hence by  Theorem \ref{thm:dual} we  get the existence of an isometry $\bar\X_t$ of $(\supp(\mm),\sfd)$ into itself $\mm$-a.e. coinciding with $\X_t$. Since $t\mapsto \X_t(x)$ is a line for $\mm$-a.e. $x$, it is  immediate to verify that  $\sfd(\bar\X_t(x),\bar\X_s(x))=|t-s|$ for every $x\in X$ and $t,s\in\R$, which gives the continuity of $\bar\X$ jointly in $t,x$.

\noindent{${\mathbf{(iii)}}$} Direct consequence of the group property \eqref{eq:group}, the measure preservation property \eqref{eq:gfb1} and what we just proved.
\end{proof}

\subsection{The quotient space isometrically embeds}
We are now ready to introduce the quotient metric space:
\begin{definition}[The quotient metric space]\label{def:xp}
We define $X':=X/\sim$ where $x\sim y$ if $\bar\X_t(x)=y$ for some $t\in\R$ and denote by $\pi:X\to X'$ the natural projection. We endow $X'$ with the distance $\sfd'$ given by $\sfd'(\pi(x),\pi(y)):=\inf_{t\in \R}\sfd(\bar\X_t(x),y)$.

Also, we denote by $\iota:X'\to X$ the right inverse of $\pi$ given by $\iota(x')=x$ provided $\pi(x)=x'$ and $\b(x)=0$. 
\end{definition}
From the fact that $(\bar\X_t)$ is a one-parameter group of isometries it is immediate to see that the definition of $\sfd'$  is well posed, i.e. that $\sfd'(\pi(x),\pi(y))$ depends only on $\pi(x),\pi(y)$. Also, it is easy to see that  $(X',\sfd')$ is a complete, separable and geodesic metric space, and that the topology induced by $\sfd'$ is the quotient topology. We omit the simple proof of these facts.

What is a priori non trivial, and the focus of this section, is that $\iota$ is an isometric embedding or, which is the same, that the minimum of the function $t\mapsto \frac{\sfd^2(x,\bar\X_t(y))}2$ is attained at that $t_0$ such that $\b(x)=\b(\bar \X_{t_0}(y))$.  The lack of smoothness of the space prevents a direct proof of the fact that such map is $C^1$, thus creating problems when trying to write down the Euler equation of the minimum. To overcome this difficulty, we first lift analysis from points to probability measures with bounded densities in order to get the $C^1$ regularity expressed by Proposition \ref{prop:c1} below, and then come back to points in the space with a limiting argument.

\begin{proposition}[A result about $C^1$ regularity]\label{prop:c1}
Let $x_0\in X$, $\mu\in\probt X$ be with bounded support and such that $\mu\leq C\mm$ for some $C>0$ and put $\mu_t:=(\bar\X_t)_\sharp\mu$. Then the map $\R\ni t\mapsto \frac 12\int \sfd^2(\cdot,x_0)\,\d\mu_t$ is $C^1$ and its derivative is given by
\begin{equation}
\label{eq:derd2}
\frac{\d}{\d t}\frac12\int \sfd^2(\cdot,x_0)\,\d\mu_t=-\frac12\int\la\nabla(\sfd^2(\cdot,x_0)),\nabla\b\ra\,\d\mu_t.
\end{equation}
\end{proposition}
\begin{proof}
It is obvious that $\R\ni t\mapsto \frac 12\int \sfd^2(\cdot,x_0)\,\d\mu_t$ is locally Lipschitz. For given $t\in\R$ we know by Proposition \ref{prop:gfb} that $\b$ is a Kantorovich potential inducing the geodesic $[0,1]\ni s\mapsto\mu_{t+s}=(\bar\X_s)_\sharp\mu_t$, hence by the differentiation formula \eqref{eq:derfiga} (applied to $f:=\sfd^2(\cdot,x_0)\nchi\in\s^2(X)$, where $\nchi$ is a Lipschitz compactly supported function identically 1 on $\cup_{t\in[0,1]}\supp(\mu_t)$)   and the identity $\int\sfd^2(\cdot,x_0)\,\d\mu_{t+h}=\int\sfd^2(\cdot,x_0)\circ\bar\X_{h}\,\d\mu_{t}$ we deduce that for any $t\in\R$ it holds
\[
\lim_{h\downarrow0}\frac{\int\sfd^2(\cdot,x_0)\,\d\mu_{t+h}-\int\sfd^2(\cdot,x_0)\,\d\mu_{t}}{2h}=-\frac12\int\la\nabla(\sfd^2(\cdot,x_0)),\nabla\b\ra\,\d\mu_t.
\]
To conclude it is therefore sufficient to show that the right hand side of \eqref{eq:derd2} is continuous. But this is obvious, because $\la\nabla(\sfd^2(\cdot,x_0)),\nabla\b\ra\in L^1_{\rm loc}(X)$ and the curve $t\mapsto\mu_t$ is weakly continuous in duality with $C_b(X)$, made of measures with uniformly bounded densities (by the measure preservation property \eqref{eq:gfb1}) and, locally in $t$, the supports of $\mu_t$ are contained in a bounded set.
\end{proof}

\begin{corollary}\label{cor:min}

Let $\mu\in\probt X$ be  with bounded support and such that $\mu\leq C\mm$ for some $C>0$ and put $\mu_t:=(\bar\X_t)_\sharp\mu$.

Then for every $x_0\in X$ the map $t\mapsto \int \sfd^2(\cdot,x_0)\,\d \mu_t$ has a unique minimum and such minimum is the only $t\in\R$ for which $\int\b\,\d\mu_{t}=\b(x_0)$.
\end{corollary}
\begin{proof} It is clear that the map $t\mapsto \int \sfd^2(\cdot,x_0)\,\d \mu_t=W_2^2(\mu_t,\delta_{x_0})$ has at least a minimum $t_0$. Fix it, let $\ppi\in\gopt(\mu_{t_0},\delta_{x_0})$ be the unique optimal geodesic plan (Theorem \ref{thm:exp}) and put $\nu_s:=(\e_s)_\sharp\ppi$. We claim that for each $s\in[0,1]$ the map $t\mapsto W_2^2(\delta_{x_0},(\bar\X_t)_\sharp\nu_s)$ has a minimum for $t=0$. Indeed, if by absurdum for some $t\in\R$ it holds $W_2(\delta_{x_0},(\bar\X_t)_\sharp\nu_s)<W_2(\delta_{x_0},\nu_s)$, the fact that $\bar\X_t:X\to X$ is an isometry would give
\[
\begin{split}
W_2(\delta_{x_0},(\bar\X_t)_\sharp\mu_{t_0})&\leq W_2(\delta_{x_0},(\bar\X_t)_\sharp\nu_s)+W_2((\bar\X_t)_\sharp\nu_s,(\bar\X_t)_\sharp\mu_{t_0})\\
&<W_2(\delta_{x_0},\nu_s)+W_2(\nu_s,\mu_{t_0})=W_2(\delta_{x_0},\mu_{t_0}),
\end{split}
\]
thus contradicting the minimality of $t_0$.

Put $\varphi:=\frac{\sfd^2(\cdot,x_0)}{2}$ and notice that $\frac12W_2^2(\nu,\delta_{x_0})=\int\varphi\,\d\nu$ for every $\nu\in\probt X$. Hence Proposition \ref{prop:c1} and the minimality of $\nu_s$ gives
\begin{equation}
\label{eq:usandoeulero}
0=\frac{\d}{\d t}\frac12W_2^2(\nu,(\X_t)_\sharp\nu_s)\restr{t=0}=-\int \la\nabla\varphi,\nabla\b \ra\,\d\nu_s,\qquad\forall s\in[0,1].
\end{equation}
Now notice that $s\mapsto \int\b\,\d\nu_s$  is Lipschitz and compute its left derivative. Given that, trivially, $\varphi$ is a Kantorovich potential for the geodesic $[0,1]\ni r\mapsto \nu_{s(1-r)}$, Corollary \ref{cor:metrbre} and the first order differentiation formula in Theorem \ref{thm:horver} ensure that for any $s\in(0,1]$ it holds:
\[
\lim_{h\downarrow0}\frac{\int\b\,\d\nu_{s-h}-\int\b\,\d\nu_s}{h}=\lim_{h\downarrow0}\frac{\int\b\,\d\nu_{s(1-h)}-\int\b\,\d\nu_s}{sh}=-\frac1s\int\la\nabla\varphi,\nabla\b\ra\,\d\nu_s\stackrel{\eqref{eq:usandoeulero}}=0.
\] 
Hence  $s\mapsto \int\b\,\d\nu_s$ is constant, i.e. for any minimum $t_0$ of $t\mapsto \int \sfd^2(\cdot,x_0)\,\d \mu_t$ it holds $\int \b\,\d\mu_{t_0}=\b(x_0)$. It is now obvious that such $t_0$ must be unique, hence the proof is completed.
\end{proof}

Corollary \ref{cor:min} allows us to prove the main result of this section:
\begin{theorem}[The quotient space isometrically embeds into the original one]\label{thm:embed}
$\iota$ is an isometric embedding of $(X',\sfd')$ into $(X,\sfd)$.
\end{theorem}
\begin{proof}
Let $x',y'\in X'$ and $x:=\iota(x')$, $y:=\iota(y')$. By definition of $\sfd'$ and $\iota$ it certainly holds $\sfd'(x',y')\leq \sfd(x,y)$. To prove the converse inequality amounts to prove that the minimum of the function $f(t):=\sfd(x,\X_t(y))$ is attained at $t=0$. For $\eps>0$ let $\mu_\eps\in\probt X$ be given by $\mu_\eps:=\mm(B_\eps(y))^{-1}\mm\restr{B_\eps(y)}$ and define $f_\eps(t):=W_2(\delta_x,(\bar\X_t)_\sharp\mu_\eps)$. Notice that $f_\eps$ is 1-Lipschitz and that it holds
\[
\begin{split}
|f_\eps(t)-f(t)|&=|W_2(\delta_x,(\bar\X_t)_\sharp\mu_\eps)-W_2(\delta_x,(\bar\X_t)_\sharp\delta_{y})|\leq W_2((\bar\X_t)_\sharp\mu_\eps,(\bar\X_t)_\sharp\delta_{y})=W_2(\mu_{\eps},\delta_y)\leq\eps.
\end{split}
\]

By definition, we have  $|\int \b\,\d\mu_\eps|\leq \eps$, thus letting $t_\eps$ be the minimizer of $f_\eps$, Corollary \ref{cor:min} and the trivial identity  $\int\b\,\d(\bar\X_t)_\sharp\mu_\eps=\int\b\,\d\mu_\eps-t$ valid for any $t\in\R$  yield $|t_\eps|=|\int\b\,\d\mu_{\eps}|\leq \eps$.

Thus for any $t\in\R$ we have
\[
f(0)\leq \eps+f_\eps(0)\leq \eps+f_\eps(t_\eps)+|t_\eps|\leq 2\eps +f_\eps(t)\leq 3\eps+f(t)
\] 
so that letting $\eps\downarrow0$ we conclude $f(0)\leq f(t)$ for any $t\in\R$, as desired.
\end{proof}

\subsection{The quotient measure $\mm'$ and basic properties of $(X',\sfd',\mm')$}
Theorem \ref{thm:embed} has a number of simple consequences about the structure of $X'$. We start defining the natural maps from $X'\times\R$ to $X$ and viceversa.
\begin{definition}[From $X'\times\R$ to $X$ and viceversa]\label{def:maumad} The maps $\mau:X'\times\R\to X$ and $\mad:X\to X'\times\R$ are defined by
\[
\begin{split}
\mau(x',t)&:=\bar\X_{-t}(\iota(x')),\\
\mad(x)&:=(\pi(x),\b(x)).
\end{split}
\]
\end{definition}

\begin{proposition}[$\mau$ and $\mad$ are homeomorphisms]\label{prop:homeo} The maps $\mau,\mad$ are homeomorphisms each one inverse of the other. Furthermore it holds
\begin{equation}
\label{eq:bilip}
\begin{split}
\frac1{\sqrt 2}\sqrt{\sfd'(x_1',x_2')^2+|t_1-t_2|^2}\leq \sfd\big(\mau(x'_1,t_1),\mau(x'_2,t_2)\big)&\leq\sqrt 2\sqrt{\sfd'(x_1',x_2')^2+|t_1-t_2|^2},\\
\end{split}
\end{equation}
for any $x_1',x_2'\in X'$, $t_1,t_2\in\R$.
\end{proposition}
\begin{proof} It is clear that $\mau\circ\mad={\rm Id}_{X}$ and $\mad\circ\mau={\rm Id}_{X'\times\R}$, thus we only need to prove \eqref{eq:bilip}. 

For the first inequality notice that since both  $\pi:(X,\sfd)\to (X',\sfd')$ and $\b:(X,\sfd)\to (\R,\sfd_{\rm Eucl})$ are 1-Lipschitz, it holds
\[
\begin{split}
\sfd\big(\mau(x'_1,t_1),\mau(x'_2,t_2)\big)^2\geq\max\{\sfd'(x_1',x_2')^2,|t_1-t_2|^2\}\geq\frac12\big(\sfd'(x_1',x_2')^2+|t_1-t_2|^2\big).
\end{split}
\]
The second follows from:
\[
\begin{split}
\sfd\big(\mau(x'_1,t_1),\mau(x'_2,t_2)\big)&=\sfd\big(\X_{-t_1}(\iota(x'_1)),\X_{-t_2}(\iota(x'_2))\big)= \sfd\big(\X_{t_2-t_1}(\iota(x'_1)),\iota(x'_2)\big)\\
&\leq  \sfd\big(\X_{t_2-t_1}(\iota(x'_1)),\iota(x'_1)\big)+ \sfd\big(\iota(x'_1),\iota(x'_2)\big)=|t_2-t_1|+\sfd'(x'_1,x'_2)\\
&\leq \sqrt 2\sqrt{\sfd'(x_1',x_2')^2+|t_1-t_2|^2}.
\end{split}
\]
\end{proof}
We can now introduce the natural measure on $X'$ as follows:
\begin{definition}[The measure $\mm'$]\label{def:mmp} We define the measure $\mm'$ on $(X',\sfd')$ as:
\[
\mm'(E):=\mm\big(\pi^{-1}(E)\cap \b^{-1}([0,1])\big),\qquad\forall E\subset X'\ \textrm{Borel}.
\]
\end{definition}
Notice that the definition is well posed because from Proposition \ref{prop:homeo} we know that for $E\subset X'$ Borel the set $\pi^{-1}(E)\subset X$ is also Borel. Also, the definition is made in such a way that the identity
\begin{equation}
\label{eq:rettangoli}
\mad_\sharp\mm(E\times I)=\mm'(E)\mathcal L^1(I),
\end{equation}
holds for every $E\subset X'$ Borel and every interval $I$ of the form $I=[a,a+1)$, $a\in\R$. Then a simple dichotomy argument based on the measure preservation property of $\bar\X_t$ shows that \eqref{eq:rettangoli} also holds for $I$ of the form $[a,a+\frac{1}{2^n})$, $a\in\R$, $n\in\N$. Thus, by density, it holds for any interval $I\subset \R$ and since the class of sets of the form $E\times I$, with $E\subset X'$ Borel and $I\subset \R$ interval, is closed under finite intersection and generates the $\sigma$-algebra of $X'\times\R$, by general results of measure theory  (see e.g. Corollary 1.6.3 in \cite{Cohn80}) we deduce that 
\begin{equation}
\label{eq:mesprod}
\mad_\sharp\mm=\mm'\times\mathcal L^1\qquad\qquad\text{and}\qquad\qquad\mau_\sharp(\mm'\times\mathcal L^1)=\mm.
\end{equation}

The metric information given by Theorem \ref{thm:embed}  and the measure theoretic one  which we just proved  grant natural relations   between Sobolev functions on $X$ and   $X'$. To emphasize the fact that the minimal weak upper gradients depend on the space and to help keeping track of spaces themselves, we write $|\nabla f|_X$ (resp. $|\nabla f|_{X'}$) for functions $f\in \s^2_{\rm loc}(X)$ (resp. in $\s^2_{\rm loc}(X')$). 
\begin{proposition}\label{prop:sezioni1} The following holds.
\begin{itemize}
\item[i)]  Let $f\in \s^2_{\rm loc}(X)$ and for $t\in\R$ let $f^{(t)}:X'\to\R$ be given by $f^{(t)}(x'):=f(\mau(x',t))$. Then for $\mathcal L^1$-a.e. $t$ it holds $f^{(t)}\in \s^2_{\rm loc}(X')$ and
\[
|\nabla f^{(t)}|_{X'}(x')\leq |\nabla f|_X(\mau(x',t)),\qquad \mm'\times\mathcal L^1\ae\ (x',t)\in X'\times\R.
\]
\item[ii)] Let $g\in \s^2_{\rm loc}(X')$ and define $f:X\to\R$ by $f(x):=g\circ\pi$. Then $f\in \s^2_{\rm loc}(X)$ and
\[
|\nabla f|_{X}(x)=|\nabla g|_{X'}(\pi(x)),\qquad\mm\ae\ x\in X.
\]
\end{itemize}
\end{proposition}
\begin{sketch} Denote by $\lip_X(f)$ (resp. $\lip_{X'}(g)$) the local Lipschitz constant in the space $(X,\sfd)$ (resp. $(X',\sfd')$) of a real valued function $f$ on $X$ (resp. $g$ on $X'$).

For point $(i)$ observe that we have the  simple inequality
\[
\begin{split}
\lip_X(f)(x)&=\lims_{y\to x}\frac{|f(x)-f(y)|}{\sfd(x,y)}\geq\!\!\! \lims_{y\to x\atop \b(y)=\b(x)}\!\!\!\frac{|f^{(\b(x))}(\pi(x))-f^{(\b(x))}(\pi(y))|}{\sfd'(\pi(x),\pi(y))}=\lip_{X'}(f^{(\b(x))})(\pi(x)),
\end{split}
\]
then approximate a generic $f\in W^{1,2}(X)$ with Lipschitz functions as in Theorem \ref{thm:energylip}, apply the inequality above to the approximating sequence and observe that by construction the leftmost side converges to $|\nabla f|_X$ in $L^2(X)$, while the measure preservation property \eqref{eq:mesprod} and the semicontinuity property \eqref{eq:lscwug} ensure that any weak limit of the rightmost side bounds $\mm'$-a.e. from above  $|\nabla f^{(t)}|_{X'}$ for $\mathcal L^1$-a.e. $t$, where $t=\b(x)$. The case of general $f\in \s^2_{\rm loc}$ is the obtained with a cut-off argument using the locality of minimal weak upper gradients.

Similarly, point $(ii)$ follows from point $(i)$ and from the relaxation of the inequality
\[
\begin{split}
\lip_X( g\circ \pi)(x)=\lims_{y\to x}\frac{|g( \pi(y))-g(\pi(x))|}{\sfd(x,y)}\leq \lims_{y\to x}\frac{|g(\pi(y))-g(\pi(x))|}{\sfd'(\pi(x),\pi(y))}=
\lip_{X'}(g)(\pi(x)).
\end{split}
\]
\end{sketch}
It is now easy to prove the following:
\begin{corollary}\label{cor:xp} $(X',\sfd',\mm')$ is an $\RCD(0,N)$ space.
\end{corollary}
\begin{sketch}$\ $\\
\noindent{\bf Infinitesimal Hilbertianity} Let $f',g'\in \s^2_{\rm loc}(X')$ and define $f,g:X\to\R$ as $f(x):=f'(\pi(x))$, $g(x):=g'(\pi(x))$. By Proposition \ref{prop:sezioni1} above we know that $f,g\in \s^2_{\rm loc}(X)$, hence, since $(X,\sfd,\mm)$ is infinitesimally Hilbertian, we have
\[
|\nabla(f+g)|^2_X+|\nabla (f-g)|^2_X=2\big(|\nabla f|_X^2+|\nabla g|_X^2\big),\qquad\mm\ae.
\]
Then noticing that $(f\pm g)(x)=(f'\pm g')(\pi(x))$, using the measure preservation property \eqref{eq:mesprod} and Fubini's theorem we deduce
\[
|\nabla(f'+g')|^2_{X'}+|\nabla (f'-g')|^2_{X'}=2\big(|\nabla f'|_{X'}^2+|\nabla g'|_{X'}^2\big),\qquad\mm'\ae,
\]
which, by the arbitrariness of $f',g'\in \s^2_{\rm loc}(X)$, yields the claim.

\noindent{\bf Curvature Dimension condition} Define $\mathcal I:\probt{X'}\to\probt{X}$ by putting
\[
\mathcal I(\mu'):=\mau_\sharp(\mu'\times\mathcal L^1\restr{[0,1]}),\qquad\forall\mu'\in\probt{X'}.
\]
Recalling that $\sfd'(x',y')=\sfd(\mau(x,t),\mau(y,t))\leq \sfd(\mau(x',t),\mau(y',s))$ for any $x',y'\in X'$ and $t,s\in\R$,  it is easy to see that $\mathcal I$ is an isometry of $(\probt{X'},W_2)$ with its image in $(\probt X,W_2)$. Denoting by  $\u_{N'}(\cdot|\mm)$ and $\u_{N'}(\cdot|\mm')$ the R\'enyi entropies functional on $\prob X$, $\prob{X'}$ respectively, it is also immediate to check that $\u_{N'}(\mathcal I(\mu')|\mm)=\u_{N'}(\mu'|\mm')$ for any $\mu'\in\probt{X'}$. Furthermore, by the uniqueness part of Theorem \ref{thm:exp} we also get that the only geodesic connecting absolutely continuous measures in $\mathcal I(\probt {X'})$ completely lies in $\mathcal I(\probt {X'})$.

The conclusion then follows by reading the $\CD(0,N)$-inequality on $X'$ as an inequality on $X$ via the map $\mathcal I$ and then recalling that the latter is a $\CD(0,N)$ space by assumption.
\end{sketch}

\subsection{\underline{Things to know:}\ Sobolev spaces and Ricci bounds over product spaces}
It is a simple exercise to check that  the standard definition of Sobolev space $W^{1,2}(\R)$ coincides with the one given by the formula \eqref{eq:w12} in the metric measure space $(\R,\sfd_{\rm Eucl},\mathcal L^1)$, $\sfd_{\rm Eucl}$ being the Euclidean distance, and that for $f\in W^{1,2}(\R)$ its minimal weak upper gradient coincides with the modulus of its distributional derivative.  To keep consistency of the notation we shall denote this object by $|\nabla f|_\R$.

We endow the set $X'\times\R$ with the product measure $\mm'\times\mathcal L^1$ and the product distance $\sfd'\times\sfd_{\rm Eucl}$ defined by
\[
\sfd'\times\sfd_{\rm Eucl}\big((x',t),(y',s)\big):=\sqrt{\sfd'(x',y')^2+|t-s|^2}
\]
Our next goal is to show that $(X'\times\R,\sfd'\times\sfd_{\rm Eucl},\mm'\times\mathcal L^1)$ is isomorphic to $(X,\sfd,\mm)$. To this aim, it is of course necessary to know how the structures of $X'$ and $\R$ reflect in the one of $X'\times\R$.

We shall use  the following result, proved in \cite{AmbrosioGigliSavare11-2}, which we restate to match the current setting.
\begin{theorem}\label{thm:xpr}
The space  $(X'\times\R,\sfd'\times\sfd_{\rm Eucl},\mm'\times\mathcal L^1)$  is $\RCD(0,\infty)$. Furthermore, the following holds:
\begin{itemize}
\item[i)] Let $f\in \s^2_{\rm loc}(X'\times\R)$ and for $t\in\R$ denote by $f^{(t)}:X'\to\R$ the function $f^{(t)}(x'):=f(x',t)$ and similarly for $x'\in X'$ let $f^{(x')}:\R\to\R$ be given by   $f^{(x')}(t):=f(x',t)$. Then:
\begin{itemize}
\item for $\mathcal L^1$-a.e. $t$ we have $f^{(t)}\in \s^2_{\rm loc}(X')$,
\item for $\mm'$-a.e. $x'$ we have $f^{(x')}\in \s^2_{\rm loc}(\R)$,
\item the formula
\begin{equation}
\label{eq:prodgrad}
|\nabla f|^2_{X'\times\R}(x',t)=|\nabla f^{(t)}|_{X'}^2(x')+|\nabla f^{(x')}|_{\R}^2(t),
\end{equation}
holds $\mm'\times\mathcal L^1$-a.e..
\end{itemize}
\item[ii)] Let $g\in \s^2_{\rm loc}(X')$ and define $f:X'\times\R\to\R$ by $f(x',t):=g(x')$. Then $f\in\s^2_{\rm loc}(X'\times\R)$ and $|\nabla f|_{X'\times\R}(x',t)=|\nabla g|_{X'}(x')$ for $\mm'\times\mathcal L^1$-a.e. $(x',t)$.
\item[iii)] Let $h\in \s^2_{\rm loc}(\R)$ and define $f:X'\times\R\to\R$ by $f(x',t):=h(t)$. Then $f\in\s^2_{\rm loc}(X'\times\R)$ and $|\nabla f|_{X'\times\R}(x',t)=|\nabla h|_{\R}(t)$ for $\mm'\times\mathcal L^1$-a.e. $(x',t)$.
\end{itemize}
\end{theorem}
We remark that the proof of the curvature bound is quite simple to obtain once Theorem \ref{thm:maprcd} is at disposal, following the original argument given in \cite{Sturm06I}. On the other hand the structure of minimal weak upper gradients in the product space provided by formula \eqref{eq:prodgrad} (which is the one granting that the product space is infinitesimally Hilbertian) seems surprisingly difficult to obtain and currently relies on some fine regularizing  properties of the heat flow.

\subsection{The space splits}\label{se:pit}

Aim of this section is to prove that $(X,\sfd)$ and $(X'\times\R,\sfd'\times\sfd_{\rm Eucl})$ are isometric and we will prove this with a duality argument based on Theorem \ref{thm:dual}. Our goal is therefore to put in relation the Sobolev norm in $X$ with the one in $X'\times\R$.  We start with the following statement, analogous to  Proposition  \ref{prop:sezioni1}:
\begin{proposition}\label{prop:sezioni2} The following holds.
\begin{itemize}
\item[i)] Let $f\in \s^2_{\rm loc}(X)$ and for $x'\in X'$ let $f^{(x')}:\R\to\R$ be given by $f^{(x')}(t):=f(\mau(x',t))$. Then  for $\mm'$-a.e. $x'$ it holds $f^{(x')}\in \s^2_{\rm loc}(\R)$ and
\[
|\nabla f^{(x')}|_{\R}(t)\leq |\nabla f|_X(\mau(x',t)),\qquad \mm'\times\mathcal L^1\ae\ (x',t)\in X'\times\R.
\]
\item[ii)] Let $h\in \s^2_{\rm loc}(\R)$ and define $f:X\to\R$ by $f(x):=h\circ\b$. Then $f\in \s^2_{\rm loc}(X)$ and
\[
|\nabla f|_{X}(x)=|\nabla h|_{\R}(\b(x)),\qquad\mm\ae \ x\in X.
\]
\end{itemize}
\end{proposition}
\begin{sketch}
The same arguments used in the proof of Proposition \ref{prop:sezioni1} can be applied also in this case recalling that the following are true:
\begin{itemize}
\item[-] for any $t,s\in \R$ it holds $|t-s|=\min_{x\in \b^{-1}(t),\ y\in \b^{-1}(s)}\sfd(x,y)$,
\item[-] for any $x\in X$ the map $t\mapsto \bar\X_t(x)$ provides an isometric embedding of $\R$ in $X$,
\item[-] it holds $\mau_\sharp(\mm'\times\mathcal L^1)=\mm$ and $\mad_\sharp\mm=\mm'\times\mathcal L^1$.
\end{itemize}
We omit the details.
\end{sketch} 
Now we introduce the following class of functions:
\[
\begin{split}
\mathcal G&:=\Big\{g:X'\times\R\to\R\ :\  g(x',t)=\tilde g(x')\textrm{ for some }  \tilde g\in \s^2\cap L^\infty(X') \Big\},\\
\mathcal H&:=\Big\{h:X'\times\R\to\R\ :\  h(x',t)=\tilde h(t)\textrm{ for some }  \tilde h\in \s^2\cap L^\infty(\R) \Big\}.
\end{split}
\]
Notice that both $\mathcal G$ and $\mathcal H$ are algebras, i.e. are closed w.r.t. linear combinations and products.

Using Theorem \ref{thm:xpr} and Proposition \ref{prop:sezioni1} we get
\begin{equation}
\label{eq:gradg}
g\in \mathcal G\qquad\Rightarrow\qquad\left\{\begin{array}{l} g\in \s^2_{\rm loc}(X'\times\R),\ g\circ\mad\in \s^2_{\rm loc}(X)\textrm{ and }\\
\\
|\nabla g|_{X'\times\R}\circ\mad=|\nabla (g\circ\mad)|_X\quad \mm\ae.
\end{array}\right.
\end{equation}
Similarly,  Theorem  \ref{thm:xpr} and Proposition \ref{prop:sezioni2} give
\begin{equation}
\label{eq:gradh}
h\in \mathcal H\qquad\Rightarrow\qquad\left\{\begin{array}{l} h\in \s^2_{\rm loc}(X'\times\R),\ h\circ\mad\in \s^2_{\rm loc}(X)\textrm{ and }\\
\\
|\nabla h|_{X'\times\R}\circ\mad=|\nabla (h\circ\mad)|_X\quad \mm\ae.
\end{array}\right.
\end{equation}
Now we introduce the algebra of functions $\mathcal A$ as:
\[
\mathcal A:=\textrm{ algebra generated by }\mathcal G\cup\mathcal H.
\]
Notice that $\mathcal A\subset \s^2_{\rm loc}(X'\times\R)$. $\mathcal A$ has two crucial properties which will allow us to prove that the Dirichlet energy of a function $f$ on $X'\times\R$ is the same as the energy of $f\circ\mad$ in $X$:
\begin{itemize}
\item[i)] such invariance property is easy to establish for functions in $\mathcal A$ once we realize that  functions in $\mathcal G$ and $\mathcal H$ have  `orthogonal gradients' in $W^{1,2}(X'\times\R)$ (by formula \eqref{eq:prodgrad}) and - after a right composition with $\mad$ - also in $W^{1,2}(X)$ (by the differentiation formula \eqref{eq:derfiga}).
\item[ii)] $\mathcal A\cap W^{1,2}(X'\times\R)$ is dense in $W^{1,2}(X'\times\R)$ and similarly $\mathcal A\circ\mad\cap W^{1,2}(X)$ is dense in $W^{1,2}(X)$. The case of $X'\times\R$ follows by a simple approximation arguments, then the one of $X$ makes use of the measure preservation property \eqref{eq:mesprod} and fact that the distances on $X'\times\R$ and $X$ are, after a composition with $\mad$, equivalent (recall Proposition \ref{prop:homeo}). 
\end{itemize}
We shall denote by $\mathcal E_X:L^2(X)\to[0,+\infty]$ the Dirichlet energy on $(X,\sfd,\mm)$ and by $\mathcal E_{X'\times\R}$ the one on $(X'\times\R,\sfd'\times\sfd_{\rm Eucl},\mm'\times\mathcal L^1)$.
\begin{proposition}\label{prop:astable} With the same notation as above, we have
\[
\mathcal E_{X'\times\R}(f)=\mathcal E_X(f\circ\mad),\qquad\forall f\in\mathcal A\cap L^2(X'\times\R).
\]
\end{proposition}
\begin{proof} A generic element $f$ of $\mathcal A$ can be written as $f=\sum_{i\in I}g_ih_i$ for some finite set $I$ of indexes and functions $g_i\in \mathcal G$, $h_i\in \mathcal H$, $i\in I$. The fact that $f\circ\mad\in\s^2_{\rm loc}(X)$ is a direct consequence of \eqref{eq:gradg} and \eqref{eq:gradh}.  The infinitesimal Hilbertianity of $X'\times\R$ (Theorem \ref{thm:xpr}) gives
\[
\begin{split}
|\nabla f|^2_{X'\times\R}&=\sum_{i,j\in I} g_ig_j\la\nabla h_i,\nabla h_j\ra_{X'\times\R}+ g_ih_j\la\nabla h_i,\nabla g_j\ra_{X'\times\R}\\
&\qquad\qquad+ h_ig_j\la\nabla g_i,\nabla h_j\ra_{X'\times\R}+ h_ih_j\la\nabla g_i,\nabla g_j\ra_{X'\times\R},
\end{split}
\]
similarly, writing $\bar f,\bar g_i,\bar h_i$ in place of $f\circ\mad,g_i\circ\mad,h_i\circ\mad$ for simplicity, from the infinitesimal Hilbertianity of $X$ we have
\[
\begin{split}
|\nabla \bar f|^2_{X}&=\sum_{i,j\in I}\bar g_i\bar g_j\la\nabla\bar h_i,\nabla\bar h_j\ra_{X}+ \bar g_i\bar h_j\la\nabla\bar h_i,\nabla\bar g_j\ra_{X}+\bar h_i\bar g_j\la\nabla\bar g_i,\nabla\bar h_j\ra_{X}+\bar h_i\bar h_j\la\nabla\bar g_i,\nabla\bar g_j\ra_{X}.
\end{split}
\]
Taking into account the relations \eqref{eq:gradg} and \eqref{eq:gradh}, we see that to conclude it is sufficient to show that for any $g\in\mathcal G$ and $h\in\mathcal H$ it holds
\begin{equation}
\label{eq:perp1}
\la \nabla g,\nabla h\ra_{X'\times\R}=0,\qquad\mm'\times\mathcal L^1\ae,
\end{equation}
and
\begin{equation}
\label{eq:perp3}
\la\nabla( g\circ\mad),\nabla(h\circ\mad) \ra_X=0,\qquad\mm\ae.
\end{equation}
To check \eqref{eq:perp1} let $\tilde g\in \s^2\cap L^\infty(X')$ and $\tilde h\in \s^2\cap L^\infty(\R)$ be such that $g(x',t)=\tilde g(x')$ and $h(x',t)=\tilde h(t)$. Then  apply point $(i)$ of Theorem \ref{thm:xpr} to the function $g+h$ and points $(ii),(iii)$ to $\tilde g,\tilde h$ to get
\[
2\la g,h\ra_{X'\times\R}=|\nabla(g+h)|^2_{X'\times\R}(x',t)- |\nabla \tilde g|_{X'}^2(x')- |\nabla \tilde h|_{\R}^2(t)=0,\qquad\mm'\times\mathcal L^1\ae \ (x',t).
\]
To get  \eqref{eq:perp3}, notice that  the chain rule \eqref{eq:chainf} (and the symmetry relation \eqref{eq:simm}) and the trivial identity $ h\circ\mad=\tilde h\circ\b$ grants that $\la \nabla ( g\circ\mad ),\nabla( h\circ\mad) \ra_X=\tilde h'\circ\b\la\nabla (g\circ\mad ),\nabla\b\ra_X$ $\mm$-a.e.. Hence to conclude  it is sufficient to show that $\la\nabla (g\circ\mad),\nabla\b\ra_X=0$ $\mm$-a.e.. If $g\circ\mad\in\s^2(X)$ then the result follows from the derivation rule   \eqref{eq:derfiga} applied to $f:=g\circ\mad$, indeed in this case the left hand side of \eqref{eq:derfiga} is identically 0. The general case follows by a simple cut-off argument based on the local nature of the claim,  we omit the details.
\end{proof}
\begin{proposition}\label{prop:approximation} With the same notation as above, the set $\mathcal A\cap W^{1,2}(X'\times\R)$ is dense in $W^{1,2}(X'\times\R)$ and the set $\mathcal A\circ\mad\cap W^{1,2}(X)$ is dense in $W^{1,2}(X)$, where by $\mathcal A\circ \mad$ we intend the set of functions of the kind $f\circ\mad$ with $f\in\mathcal A$.
\end{proposition}
\begin{sketch} We start with the first claim. With a diagonalization argument it is sufficient to prove that for $f\in W^{1,2}(X'\times\R)$ bounded with compact support there exists a sequence $(f_n)\subset\mathcal A\cap W^{1,2}(X'\times\R)$ converging to $f$ in $W^{1,2}(X'\times\R)$. Fix such $f$ and for $n\in\N$ and $i\in \Z$ define
\[
g_{i,n}(x'):=n\int_{i/n}^{(i+1)/n}f(x',s)\,\d s,\qquad\text{ and }\qquad h_{i,n}(t):=\nchi_n(t-i/n),
\]
where $\nchi_n:\R\to\R$ is given by
\[
\nchi_n(t):=\left\{\begin{array}{ll}
0,&\qquad\textrm{ if }t<-1/n,\\
nt+1,&\qquad\textrm{ if }-1/n\leq t< 0,\\
1-nt,&\qquad\textrm{ if }0\leq t<1/n,\\
0,&\qquad\textrm{ if }1/n<t.\\
\end{array}
\right.
\]
Then define $f_n:X'\times\R\to\R$ by $f_n(x',t):=\sum_{i\in \Z}h_{i,n}(t)g_{i,n}(x')$.
It is obvious that  $f_n\in \mathcal A\cap W^{1,2}(X'\times\R)$ and with simple computations we also see that
\begin{align*}
\|f_n\|_{L^2(X'\times\R)}&\leq \|f\|_{L^2(X'\times\R)},&&\forall n\in\N\\
\lim_{n\to\infty}\int \varphi f_n \,\d\mm'\,\d \mathcal L^1&=\int \varphi f\,\d\mm'\,\d \mathcal L^1,\quad &&\forall \varphi:X'\times\R\to\R\textrm{ Lipschitz with compact support,}
\end{align*}
which ensures that $f_n\to f$ in $L^2(X'\times\R)$.

Also, some algebraic manipulation - we omit the details - shows that
\[
\begin{split}
\int_{X'\times\R}|\nabla f_n^{(t)} |_{X'}^2(x')\,\d(\mm'\times\mathcal L^1)(x',t)\leq\int_{X'\times\R}|\nabla f^{(t)}|^2_{X'}(x')\,\d(\mm'\times\mathcal L^1)(x',t),
\end{split}
\]
and
\[
\int_{X'\times\R}|\nabla f_n^{(x')}|^2_{\R}(t)\,\d(\mm'\times\mathcal L^1)(x',t) \leq \int_{X'\times\R}|\nabla f^{(x')}|^2_{\R}(t)\,\d(\mm'\times\mathcal L^1)(x',t).
\]
Taking into account the characterization of the Sobolev space $W^{1,2}(X'\times\R)$ given in Theorem \ref{thm:xpr} and the $L^2$-lower semicontinuity of the $W^{1,2}$-norm we deduce that $\|f_n\|_{W^{1,2}}\to\|f\|_{W^{1,2}}$. Since $W^{1,2}(X'\times\R)$ is Hilbert, $L^2$-convergence  plus convergence of the Sobolev norm yield $W^{1,2}$-convergence.

For the second part of the proof notice that directly from the definition of Sobolev space we have that if $(Y,\sfd_Y,\mm_Y)$ is a metric measure space and $\sfd_Y'\leq \sfd_Y$ is another distance on $Y$ inducing the same topology of $\sfd_Y$, then for every $f\in W^{1,2}(Y,\sfd_Y',\mm_Y)$ it holds $f\in W^{1,2}(Y,\sfd_Y,\mm_Y)$ with $|\nabla f|_{(Y,\sfd_Y,\mm_Y)}\leq |\nabla f|_{(Y,\sfd_Y',\mm_Y)}$ $\mm_Y$-a.e.. Similarly, if a distance is scaled by a factor $\lambda>0$, the corresponding gradient part of the Sobolev norm is scaled by $\frac1\lambda$. The conclusion then comes from the first part of the proof, the identity   $\mau_\sharp(\mm'\times\mathcal L^1)=\mm$ and the inequalities \eqref{eq:bilip}.
\end{sketch}

The main theorem of this section now follows easily.
\begin{theorem}[``Pythagoras' theorem'' holds]\label{thm:pitagora} The maps $\mau,\mad$ are isomorphisms of the spaces $(X,\sfd,\mm)$ and $(X'\times\R,\sfd'\times\sfd_{\rm Eucl},\mm'\times\mathcal L^1)$, i.e. they are measure preserving isometries.
\end{theorem}
\begin{proof} We already know that
\begin{equation}
\label{eq:mp2}
\mau_\sharp\mm=\mm'\times\mathcal L^1,\qquad\qquad\text{and}\qquad\qquad\mad_\sharp(\mm'\times\R)=\mm,
\end{equation}
and  by Proposition \ref{prop:astable} we know that the equality
\begin{equation}
\label{eq:en2}
\mathcal E_{X'\times\R}(f)=\mathcal E_{X}(f\circ\mad),
\end{equation}
holds for every $f\in\mathcal A\cap L^2(X'\times\R)$. Hence  using \eqref{eq:mp2} and Proposition \ref{prop:approximation} we deduce that \eqref{eq:en2} holds for every $f\in L^2(X'\times\R)$.

Now recall that  $(X,\sfd,\mm)$ is an $\RCD(0,N)$ space and thus $\RCD(0,\infty)$ with  the measure $\mm$ being doubling. Similarly, we know by Theorem \ref{thm:xpr} that $(X'\times\R,\sfd'\times\sfd_{\rm Eucl},\mm'\times\mathcal L^1)$ is $\RCD(0,\infty)$ and from the fact that both $\mm'$ and $\mathcal L^1$ are doubling measures it is easy to get that $\mm'\times\mathcal L^1$ is doubling as well. 

Hence we can apply Theorem \ref{thm:dual} to deduce that $\mau,\mad$ have 1-Lipschitz representatives. Given that we already know that they are continuous (Proposition \ref{prop:homeo}), the proof is complete.
\end{proof}

\subsection{The quotient space has dimension $N-1$}
It remains to prove that the quotient space $(X',\sfd',\mm')$ has `1 dimension less' than $(X,\sfd,\mm)$. This, of course, should be interpreted in terms of the synthetic treatment of curvature-dimension bounds, the precise statement being given below. Notice that our argument for such dimension reduction is in fact the same used by Cavalletti-Sturm in \cite{Cavalletti-Sturm12}.
\begin{theorem}[The quotient space has dimension $N-1$]\label{thm:dimm1} The following holds.
\begin{itemize}
\item[i)] If $N\geq 2$, then $(X',\sfd',\mm')$ is an $\RCD(0,N-1)$ space.
\item[ii)] If $N\in[1,2)$, then $X'$ contains exactly one point.
\end{itemize}
\end{theorem}
\begin{proof} $\ $\\
\noindent{$\mathbf{ (i)}$}
We already know by Corollary \ref{cor:xp} that $(X',\sfd',\mm')$ is an   $\RCD(0,N)$ space and a simple approximation argument ensures that to conclude it is sufficient to check the $\CD(0,N-1)$ condition for given   $\mu_0,\mu_1\in\probt{X'}$ with bounded support and absolutely continuous w.r.t. $\mm'$, say $\mu_i=\rho_i\mm'$, $i=0,1$. By Proposition \ref{prop:varphi} we know that there exists a unique $\ppi\in\gopt(\mu_0,\mu_1)$, and   that the measures $\mu_t:=(\e_t)_\sharp\ppi$ are absolutely continuous w.r.t. $\mm'$, say $\mu_t=\rho_t\mm'$, for every $t\in[0,1]$.

Let $\alpha,\beta>0$  be arbitrary, put $\nu_0:=\frac1\alpha\mathcal L^1\restr{[0,\alpha]}$, $\nu_1:=\frac1\beta\mathcal L^1\restr{[0,\beta]}$ so that $\nu_0,\nu_1\in\probt \R$, let $t\mapsto\nu_t=\frac{1}{(1-t)\alpha+t\beta}\mathcal L^1\restr{[0,(1-t)\alpha+t\beta]}$ be the unique geodesic connecting $\nu_0$ to $\nu_1$ and $\tilde\ppi$  the unique element of $\gopt(\nu_0,\nu_1)$. 

Define   $\mathcal J:C([0,1],X')\times C([0,1],\R)\to C([0,1],X'\times\R)$ by
\[
\mathcal J(\gamma_1,\gamma_2)_t:=(\gamma_{1,t},\gamma_{2,t}),
\]
and the plan $\ppi\otimes\tilde\ppi\in\prob{C([0,1],X'\times \R)}$ as $\mathcal J_\sharp(\ppi\times\tilde\ppi)$. It is then immediate to verify that $\ppi\otimes\tilde\ppi\in \gopt(\mu_0\times\nu_0,\mu_1\times\nu_1)$ and that it   satisfies  $(\e_t)_\sharp(\ppi\otimes\ppi')=\mu_t\times\nu_t$. Thus
\begin{equation}
\label{eq:densprod}
\frac{\d(\e_t)_\sharp(\ppi\otimes\ppi')}{\d(\mm'\times\mathcal L^1)}(\gamma_t,\tilde\gamma_t)=\frac{\rho_t(\gamma_t)}{(1-t)\alpha+t\beta},\qquad\ppi\times\ppi'\ae \ (\gamma,\tilde\gamma).
\end{equation}
By assumption we know that $(X,\sfd,\mm)$ is an $\RCD(0,N)$ space and by Theorem \ref{thm:pitagora} that it is isomorphic to $(X'\times\R,\sfd'\times\sfd_{\rm Eucl},\mm'\times\mathcal L^1)$. Thus the latter is an  $\RCD(0,N)$ space and T  by \eqref{eq:densprod} and Proposition \ref{prop:varphi} applied to the plan $\ppi\otimes\tilde\ppi$  we get that for every $t\in[0,1]$, $\alpha,\beta>0$ and $\ppi$-a.e. $\gamma$ it holds
\begin{equation}
\label{eq:inQ}
\left(\frac{\rho_t(\gamma_t)}{(1-t)\alpha+t\beta}\right)^{-\frac1{N'}}\geq(1-t)\left(\frac{\rho_0(\gamma_0)}{\alpha}\right)^{-\frac1{N'}}+t\left(\frac{\rho_1(\gamma_1)}{\beta}\right)^{-\frac1{N'}},\qquad\forall N'\geq N.
\end{equation}
In particular, for every $t\in[0,1]$ the inequality \eqref{eq:inQ} holds for  $\ppi$-a.e. $\gamma$ and every $\alpha,\beta\in\Q$, $\alpha,\beta>0$. Being the terms in \eqref{eq:inQ}  continuous on $\alpha,\beta\in\Q$, $\alpha,\beta>0$, \eqref{eq:inQ} also holds for  $\ppi$-a.e. $\gamma$ and every $\alpha,\beta\in\R$, $\alpha,\beta>0$. Choosing $\alpha:=(\rho_0(\gamma_0))^{-\frac1{N'-1}}$ and $\beta:=(\rho_1(\gamma_1))^{-\frac1{N'-1}}$, after little algebraic manipulation we obtain 
\[
\rho_{t}(\gamma_{t})^{-\frac1{N'-1}}\geq (1-t)\rho_0(\gamma_0)^{-\frac1{N'-1}}+t\rho_1(\gamma_1)^{-\frac1{N'-1}},\qquad\ppi\ae \ \gamma,
\]
which integrated w.r.t. $\ppi$ yields $\u_{N'-1}(\mu_t)\leq (1-t)\u_{N'-1}(\mu_0)+t\u_{N'-1}(\mu_1)$ for every $N'\geq N$, as desired.

\noindent{$\mathbf{ (ii)}$}  It is clear that $X'$ is non empty. Assume by contradiction that it contains more than one point. Then, since $(X',\sfd')$ is geodesic, it contains an isometric copy $I\subset X'$ of some non-trivial interval in $\R$. Given that  $X'\times\R\supset I\times\R$, the Hausdorff dimension of $X'\times\R$ is at least 2 and since by Theorem \ref{thm:pitagora} we know that $(X'\times\R,\sfd'\times\sfd_{\rm Eucl})$ is isometric to $(X,\sfd)$, to conclude it is sufficient to show that for any $R>0$ and  $N'\in(N,2)$ we have $\mathcal H^{N'}(B_R(x_0))=0$, where $\mathcal H^{N'}$ is the $N'$-dimensional Hausdorff measure and $x_0\in X$ a fixed point. As pointed out in \cite{Sturm06II}, this is a standard consequence of the doubling condition \eqref{eq:BG}. We sketch the argument. We have
\[
\begin{split}
\mathcal H^{N'}_\delta(B_R(x_0))&=\omega_{N'}\inf\bigg\{\sum_{i\in \N}\left(\frac{{\rm diam}(E_i)}{2}\right)^{N'}: {\rm diam}(E_i)\leq\delta,\ \forall i\in \N\ \text{ and }\ B_R(x_0)\subset\bigcup_{i\in \N}E_i\bigg\}\\
&\leq \omega_{N'}\left(\frac\delta 2\right)^{N'}\inf\bigg\{k\in\N : \exists\, x_1,\ldots,x_k\in X\text{ with }B_R(x_0)\subset\bigcup_{j\in 1}^k B_{\delta/2}(x_i)\bigg\}.
\end{split}
\]
The inequality \eqref{eq:BG} grants that there are at most $(8R/\delta)^{N}$ disjoint balls of radius $\delta/4$ with center in $B_R(x_0)$. Fixing a maximal package of such disjoint balls, the balls with same centers and radius $\delta/2$ cover $B_R(x_0)$ and therefore $\mathcal H^{N'}_\delta(B_R(x_0))\leq \delta^{N'-N}\frac{(8R)^N}{2^{N'}}$. The conclusion follows recalling that $\mathcal H^{N'}(B_R(x_0))=\lim_{\delta\downarrow0}\mathcal H^{N'}_\delta(B_R(x_0))$.
\end{proof}

\def\cprime{$'$}

\end{document}